\documentclass[a4paper,12pt]{book}
\usepackage[T1]{fontenc}
\usepackage[english,francais]{babel}
\usepackage[cp1252]{inputenc}
\usepackage{amsthm}
\usepackage{amsmath}
\usepackage{amsfonts}
\usepackage{stmaryrd}
\usepackage{amssymb}
\usepackage[pagebackref]{hyperref}

\newcommand{\di}{\mathrm{d}}

\newcommand{\alp}{\omega+2}
\newcommand{\gt}{\tilde{g}}
\newcommand{\an}{\frac{n-2}{4(n-1)}}
\newcommand{\N}{\frac{4}{n-2}}
\newcommand{\Nn}{\frac{2n}{n-2}}
\newcommand{\fia}{\varphi_\alpha}
\newcommand{\intl}{[ \hspace{-0.6mm} [}
\newcommand{\intr}{] \hspace{-0.6mm} ]}

\newcommand{\fii}{\varphi_{\alpha_i}}
\newcommand{\e}{\varepsilon}
\newcommand{\Ca}{\mathrm{card}}

\newtheorem{conjecture}{Conjecture}[chapter]
\newtheorem{theorem}{Théorème}[chapter]
\newtheorem{definition}{Définition}[chapter]
\newtheorem{proposition}{Proposition}[chapter]
\newtheorem{corollary}{Corollaire}[chapter]
\newtheorem{lemma}{Lemme}[chapter]

\newtheorem{proprietes}{Propriétés}[chapter]
\newtheorem{problem}{Problème}[chapter]

\newtheorem{conjecturee}{Conjecture}[chapter]
\newtheorem{theoreme}{Theorem}[chapter]

\newtheorem{corollarye}{Corollary}[chapter]

\newtheorem{probleme}{Problem}[chapter]

\newcommand*\chapterstar[1]{%
  \chapter*{#1}%
  \addcontentsline{toc}{chapter}{#1}%
     \markboth{INTRODUCTION}{INTRODUCTION}}

\title{Le problème de Yamabe avec singularités et la conjecture de Hebey--Vaugon}
\author{Farid Madani}
\date{}
\begin{document}

\frontmatter
\thispagestyle{empty}
\begin{center}
\vspace{0.3cm}
\textsc{\textbf{{\Large Université Pierre et Marie Curie}}} \\
\textbf{\textsc{École Doctorale de Sciences Mathématiques de Paris Centre}}

\vspace{2.2cm}

{\huge \textsc{Thèse de doctorat}}

\vspace{0.5cm}

{\large Discipline:} {\Large Mathématiques}

\vspace{2cm}

{\large présentée par}
\vspace{0.5cm}

\textbf{{\LARGE Farid \textsc{Madani}}}

\vspace{1cm}
\hrule
\vspace{0.5cm}
{\LARGE \textbf{Le problème de Yamabe avec singularités et la conjecture de Hebey--Vaugon}}
\vspace{0.5cm}
\hrule
\vspace{0.5cm}

{\Large 
dirigée par Thierry \textsc{Aubin}}

\vspace{3.5cm}

{\large Soutenue le 29 septembre 2009 devant le jury composé de :}

\vspace{1cm}

\begin{tabular}{lll}
M. Bernd \textsc{Ammann} &  Universit\"at Regensburg & Rapporteur \\
M. Emmanuel \textsc{Hebey} &    Université de Cergy--Pontoise  &   \\
M. Frédéric \textsc{Hélein} & Université Paris-Diderot &   \\
M. Emmanuel \textsc{Humbert} & Université Nancy I & Rapporteur \\
M. Michel \textsc{Vaugon} & Université Paris 6 &  Directeur de thèse\\
\end{tabular}

\end{center}
\newpage
\thispagestyle{empty}
\vspace*{\fill}

\noindent\begin{center}
\small{\begin{tabular}{lcl}
Institut de Mathématiques de Jussieu & \hspace{2cm} & \'Ecole doctorale Paris centre Case 188\\
175, rue du chevaleret               &  & 4 place Jussieu \\
75013 Paris                          &  & 75 252 Paris cedex 05\\
\end{tabular}}
\end{center}
\newpage
\thispagestyle{empty}
\begin{center}
\quad\\
\vspace{5cm}
\hspace{10cm}\emph{\`A la mémoire de Thierry Aubin}.
\end{center}
\newpage
\;\quad
\newpage
\begin{center}
\section*{\sc Remerciements} 
\end{center}
\vspace{2cm}

Je tiens  tout d'abord à exprimer ma profonde gratitude et reconnaissance envers mon directeur de thèse Thierry Aubin. J'ai eu la douleur de le perdre au début de cette année. Il m'a introduit à la recherche mathématique, et j'ai particulièrement apprécié son honnêteté mathématique et sa façon de raisonner.\\  

J'aimerais aussi exprimer ma gratitude envers Michel Vaugon, qui a accepté de reprendre la direction de ma thèse. En très peu de temps il a lu ma thèse et fait beaucoup de précieux commentaires. Je le remercie pour sa disponibilité et sa sympathie. \\

Bernd Ammann et Emmanuel Humbert ont accepté d'être rapporteurs de ma thèse et de participer à mon jury. Je les remercie pour les remarques et suggestions qu'ils ont faites sur mon travail.\\

Je remercie Emmanuel Hebey et Frédéric Hélein pour avoir accept\'e d'être membres de mon jury. Un remerciement particulier pour Emmanuel Hebey pour ses commentaires et suggestions pertinentes.\\

Je tiens à remercier Tien-Cuong Dinh et Elisha Falbel pour leur soutien et leurs conseils au cours des ces trois années de thèse. \\

Durant ma thèse, j'ai partagé, avec les thésards du 7ème étage, pas mal de déjeuners (presque tous les jours). Je les remercie pour les pauses de détente que l'on a partagé à l'institut et même à l'extérieur. Je pense qu'ils se reconnaissent sans les citer un par un. Je les remercie aussi pour ces séminaires mathématiques, où l'on peut comprendre jusqu'à 100\% du contenu. Un  remerciement spécial pour Johan, Julien, Nicolas et pour mon "frère d'armes" Nabil. \\

Je salue chaleureusement tous mes amis qui sont toujours de mon coté. Enfin, je remercie profondément tous les membres de ma famille pour leur soutien constant durant toutes mes études. Ils occupent une place particulière au fond de moi.



\mainmatter

\tableofcontents
\markboth{}{}


%
\newpage
\section*{Notations}\addcontentsline{toc}{chapter}{Notations}
\begin{tabular}{l}

$[q]$ la partie entière de $q$\\
$N=\frac{2n}{n-2}$ \\
$\intl 1,n\intr=\{1,2,\cdots,n\}$\\
$S_n$ la sphère unité de dimension $n$\\
$S_n(r)$ la sphère de rayon $r$\\
$g_{can}$ la métrique canonique sur $S_n$\\
$\mathcal{E}$ la métrique euclidienne\\ 
$\di\sigma$ l'élément de volume associé à $(S_{n-1},g_{can})$  \\
$\di\sigma_r$ l'élément de volume de $S_{n-1}(r)$ \\ 
$vol(M)$ volume de la variété $M$ \\
$\omega_n$ volume de la sphère $S_n$ \\  
$\Delta_g$ le Laplacien de la métrique $g$ \\ 
$\Delta_{\mathcal{E}}$  le Laplacien de la métrique euclidienne $\mathcal{E}$\\
$|\beta|=k$ si $\beta\in\mathbb N^k$\\
$K(n,2)^{-2}=\frac{1}{4}n(n-2)\omega_n^{2/n}$\\
$\nabla_i=\nabla_{\partial_i}$ la dérivée covariante\\ 
$\nabla_\beta=\nabla_{\beta_1}\cdots\nabla_{\beta_k}$ \\
$R_g$ la courbure scalaire associée à $g$ \\ 
$L_g=\Delta_g+\frac{n-2}{4(n-1)}R_g$ le Laplacien conforme  \\ 
$G_P$ fonction de Green en $P$.\\
$T(M)$ l'espace tangent de $M$\\
$T^*M$ l'espace cotangent de $M$\\
$\Gamma(M)$ l'espace des champs de vecteurs $C^\infty$\\
$L^p(M)$ espace de Lebesgue sur $M$\\
$H^p_q(M)$ Espace de Sobolev \\ 
$H^p_{q,G}(M)$ Espace de Sobolev $G-$invariant\\
$H_1(M)=H^2_1(M)$, $H_{1,G}(M)=H^2_{1,G}(M)$\\
$\|\cdot\|_p$ norme sur $L^p$\\ 
$\|\cdot\|_{H_1}$ norme sur $H_1$\\ 
$(\cdot,\cdot)_{g,L^2}=(\cdot,\cdot)_{L^2}$ produit scalaire sur $L^2$ avec la métrique $g$\\
$(\cdot,\cdot)_{g,H_1}=(\cdot,\cdot)_{H_1}$ produit scalaire sur $H_1$ avec la métrique $g$\\
$\mu(g)=\mu_N(g)$ l'invariant conforme de Yamabe \\ 
$\mu_G(g)=\mu_{N,G}(g)$ l'invariant  $G-$conforme de Yamabe  \\
$E(\varphi)$ énergie de $\varphi$ \\ 
$I_g$ La fonctionnelle de Yamabe\\ 
$I(M,g)$ le groupe d'isométries de $(M,g)$ \\ 
$C(M,g)$ le groupe conforme de $(M,g)$\\
$G$ sous groupe de $I(M,g)$
\end{tabular}

\chapterstar{Introduction}

Le travail présenté dans cette thèse est séparé en deux parties. La première partie est consacrée à l'étude d'un certain type d'équations aux dérivées partielles non linéaires sur une variété compacte. Ensuite, on donne une signification géométrique de ces équations. La particularité ici est que l'un des coefficients de ces équations n'a pas la régularité habituellement supposée, ce qui permettra d'obtenir un "théorème de Yamabe" avec singularités. La seconde partie est consacrée à l'étude d'une conjecture de Hebey--Vaugon dans le cadre du problème de Yamabe équivariant.

\subsection*{Première partie}

On considère une variété riemannienne $(M,g)$ compacte de dimension $n\geq 3$. On  note $R_g$ la courbure scalaire de $g$. Le problème de Yamabe est le suivant:

\begin{problem}\label{yamabe problem intro}
Existe-t-il une métrique conforme à $g$ de courbure scalaire constante?
\end{problem}
On pose $\gt=\varphi^\N g$, où $\varphi$ est une fonction $C^\infty$ strictement positive. $\gt$ est une solution du problème de Yamabe si et seulement si $\varphi$ est solution de l'équation suivante:
\begin{equation}\label{yamabei}
\frac{4(n-1)}{n-2}\Delta_g \varphi+ R_g\varphi=  R_{\gt} \varphi^{\frac{n+2}{n-2}}
\end{equation}
où $\Delta_g=-\nabla^i\nabla_i$ est le Laplacien de $g$ et $R_{\gt}$ est une constante qui joue le rôle de la courbure scalaire de $\gt$. T.~Aubin a ramené la résolution de ce problème à la résolution de la conjecture suivante:
\begin{conjecture}[T.~Aubin \cite{Aub}]\label{Aubincon intro}
Si $(M,g)$ est une variété riemannienne compacte $C^\infty$ de dimension $n\geq 3$ et non conformément difféomorphe à $(S_n,g_{can})$ alors 
 \begin{equation}\label{cau}
 \mu(M,g)<\mu(S_n,g_{can})
\end{equation}
où $\mu(M,g)=\inf\biggl\{\displaystyle\frac{\int_M |\nabla\psi|^2+\frac{n-2}{4(n-1)}R_g\psi^2\di v}{\|\psi\|^2_{\frac{2n}{n-2}}},\; \psi\in H_1(M)-\{0\}\biggr\}$.
\end{conjecture}
Il est bien connu que $\mu(S_n,g_{can})=\frac{1}{4}n(n-2)\omega_n^{2/n}$. Les travaux de T.~Aubin \cite{Aub}, R.~Schoen~\cite{Schoen} et H.~Yamabe \cite{Yam} ont montré que cette conjecture est toujours vraie, et le problème de Yamabe admet toujours des solutions. En d'autres termes, dans chaque classe conforme $[g]$, on peut toujours trouver une métrique à courbure scalaire constante. \\

On note par $I(M,g)$ et $C(M,g)$  le groupe d'isométries et le groupe conforme de $(M,g)$ respectivement. Soit $G$ un sous groupe de $I(M,g)$.  E.~Hebey et  M.~Vaugon~\cite{HV} ont étudié le problème de Yamabe équivariant, qui généralise le problème de Yamabe, et que l'on peut exprimer de la manière suivante:
\begin{problem}\label{HVProbleme intro}
 Existe-t-il une métrique $g_0$, $G-$invariante qui minimise la fonctionnelle
$$J(g')=\frac{\int_MR_{g'}\di v(g')}{(\int_M\di v(g'))^{\frac{n-2}{n}}}$$
où $g'$ appartient à la classe $G-$conforme de $g$:
$$[g]^G:=\{\tilde g=e^fg/ f\in C^\infty(M),\; \sigma^*\tilde g=\tilde g\quad \forall \sigma\in G\}$$
\end{problem}
E.~Hebey et  M.~Vaugon ont montré que ce problème à toujours des solutions, ce qui a pour première conséquence l'existence d'une métrique $g_0$, $G-$invariante et conforme à $g$, telle que la courbure scalaire de $g_0$ est constante. La deuxième conséquence est que la conjecture suivante est démontrée.
\begin{conjecture}[\textbf{Lichnerowicz \cite{Lic}}]
Pour toute variété riemannienne $(M,g)$, compacte $C^\infty$, de dimension $n$ et qui n'est pas conformément difféomorphe à $(S_n,g_{can})$, il existe une métrique $\gt$ conforme à $g$ de courbure scalaire $R_{\tilde g}$ constante et pour laquelle $I(M,\tilde g)=C(M,g)$.
\end{conjecture}

Le travail présenté dans la première partie de la thèse est l'étude du problème de Yamabe~\ref{yamabe problem intro} (sans et avec la présence de symétries), lorsque la métrique $g$ n'est pas nécessairement $C^\infty$.  On suppose que la métrique $g$ est dans $H^p_2$, où $p>n$, l'espace de Sobolev des métriques dont on donnera la définition plus loin. Grâce aux inclusions de Sobolev $H^p_2\subset C^{1,\beta}$ (l'espace de Hölder d'exposant $\beta\in ]0,1[$), les métriques sont donc de classe $C^{1,\beta}$. Les tenseurs de courbures de Riemann, de Ricci et la courbure scalaire sont dans $L^p$. Plus précisément, si on suppose que $g$ satisfait l'hypothèse suivante:\\

\textbf{Hypothèse $\boldsymbol{(H)}$:} \emph{$g$ est une métrique dans l'espace de Sobolev $H_2^p(M,T^*M\otimes T^*M)$ avec $p>n$. Il existe un point $P_0\in M$ et $\delta>0$ tels que $g$ est $C^\infty$ sur la boule $B_{P_0}(\delta)$.}\\

Alors le problème que l'on résout est le suivant:
\begin{problem}\label{yam sing intro}
 Soit $g$ une métrique qui satisfait l'hypothèse $(H)$. Existe-t-il une métrique $\tilde g$ conforme à $g$ pour laquelle la courbure scalaire $R_{\gt}$ est constante (même aux points où $R_g$ n'est pas régulière)?
\end{problem}

Avant de résoudre ce problème, on commence  par étudier plus généralement les équations suivantes:
 \begin{equation}\label{dfgfg}
 \Delta_g\varphi+h\varphi=\tilde h\varphi^{\frac{n+2}{n-2}}
\end{equation}
 où $h$ est une fonction qui est supposée seulement être dans $ L^p(M)$ (c'est là l'originalité de cette étude) et $\tilde h\in \mathbb R$. La métrique $g$ est supposée $C^\infty$ (la supposer $C^2$ donnerait les même résultats, ce n'est pas un point important). On appellera ces équations les équations de type Yamabe. Comme ces équations sont non linéaires et que $h$ est dans $L^p$, les théorèmes de régularité standard ne s'appliquent pas directement. On établit le résultat suivant (adaptation d'un théorème de N.~Trudinger \cite{Trud} au cas où $h$ n'est que dans $L^p$)

\begin{theorem}\label{regYS intro}
Soit $(M,g)$ une variété riemannienne compacte $C^\infty$ de dimension $n\geq 3$, $p$ et $\tilde h$ sont deux nombres réels, avec $p>n/2$. Si $\varphi\in H_1(M)$ est une solution faible positive non triviale de l'équation  \ref{dfgfg}  alors $\varphi\in H_2^p(M)\subset C^{1-[n/p],\beta}(M)$ et $\varphi$ est strictement positive.
\end{theorem}
La régularité donnée par ce théorème est optimale. \\
En ce qui concerne l'existence des solutions, on démontre que la fonctionnelle $I_g$, définie pour tout $\psi\in H_1(M)-\{0\}$ par 
\begin{equation*}
I_g(\psi)=\frac{\int_M |\nabla\psi|^2+h\psi^2\di v}{\|\psi\|^2_{\frac{2n}{n-2}}}\label{ca} 
\end{equation*}
 atteint son minimum $\mu(g)$, si $\mu(g)<\frac{1}{4}n(n-2)\omega_n^{2/n}$ (où $\omega_n$ est le volume de la sphère standard $S_n$). On obtient alors le résultat suivant: 
\begin{theorem}\label{cg intro}
Soit $(M,g)$ une variété riemannienne compacte $C^\infty$ de dimension $n\geq 3$ et $p>n/2$. Si 
$$\mu(g)<\frac{1}{4}n(n-2)\omega_n^{2/n}$$ 
alors l'équation \eqref{dfgfg} admet une solution strictement positive $\varphi\in H^p_2(M)\subset C^{1-[n/p],\beta}(M)$,  qui minimise la fonctionnelle $I_g$, où $\beta\in ]0,1[$.
\end{theorem}
Si $h$ est $G-$invariante, on définit 
$$\mu_{G}(g)=\inf_{\psi\in H_{1,G}(M)-\{0\}} I_g(\psi)$$
où $H_{1,G}(M)$ est l'espace des fonctions dans $H_1(M)$,  $G-$invariantes. On note par $O_G(Q)$ l'orbite du point $Q\in M$.  On obtient le résultat suivant:
\begin{theorem}\label{theginv intro}
Si $0<\mu_{G}(g)<\frac{1}{4}n(n-2)\omega_n^{2/n}(\inf_{Q\in M}\Ca O_G(Q))^{2/n}$ alors  l'équation \eqref{dfgfg}  admet une solution $\varphi\in H_{2,G}^p(M)\subset C^{1-[n/p],\beta}(M)$ strictement positive, $G-$invariante et minimisante pour 
la fonctionnelle $I_{g}$.
\end{theorem}
Ce théorème se démontre  en utilisant la méthode variationnelle (comme dans les cas classiques où $h$ est très régulière), les inclusions de Sobolev en présence de symétries, trouvées par E.~Hebey et  M.~Vaugon \cite{HV2} et l'inégalité de la meilleure constante en présence symétries calculée par Z.~Faget \cite{Fag}.\\

Dans le chapitre \ref{CHYS}, on étudie l'équation \eqref{dfgfg} lorsque $h=\frac{n-2}{4(n-1)}R_g$ et $g$ est une métrique qui satisfait l'hypothèse $(H)$.  Ce cas a une signification géométrique, il permet de résoudre le problème  \ref{yam sing intro}  (le problème de Yamabe avec singularités).  La courbure scalaire  $R_g$ est dans $ L^p(M)$ et l'équation \eqref{dfgfg} devient l'équation de Yamabe \eqref{yamabei}. D'après le théorème~\ref{cg intro}, la résolution du problème~\ref{yam sing intro} est ramenée à  la preuve de l'inégalité $\mu(g)<\frac{1}{4}n(n-2)\omega_n^{2/n}$ (cette inégalité a déjà été démontrée lorsque $g$ est $C^\infty$). Dans le cas où $g$ satisfait l'hypothèse $(H)$, on commence par démontrer certaines propriétés (connues dans le cas $C^\infty$): l'invariance conforme de $\mu(g)$, l'invariance conforme faible du Laplacien conforme $L_g=\Delta_g+\frac{n-2}{4(n-1)}R_g$ et l'existence de la fonction de Green pour cet opérateur. Ensuite, on démontre  le résultat suivant:
\begin{theorem}\label{conj aub intro}
 Soit $M$ une variété compacte $C^\infty$ de dimension $n$, $g$ une métrique riemannienne qui satisfait l'hypothèse $(H)$. Si $(M,g)$ n'est pas conformément difféomorphe à la sphère $(S_n, g_{can})$  alors $\mu(g)<\frac{1}{4}n(n-2)\omega_n^{2/n}$.
\end{theorem}
Lorsque la métrique  $g$ est $C^\infty$, ce théorème a résolu la conjecture \ref{Aubincon intro}. Les arguments utilisés pour le démontrer dans ce cas sont encore valables lorsque $g$ satisfait l'hypothèse $(H)$. En effet, il suffit de construire une certaine  fonction test $\varphi$ qui vérifie  $I_g(\varphi)<\frac{1}{4}n(n-2)\omega_n^{2/n}$. Les fonctions test construites par T.~Aubin \cite{Aub} et R.~Schoen \cite{Schoen}, sont encore utilisables dans ce cas singulier. \\
Dans le cas équivariant (en présence de symétries), le résultat obtenu est le suivant:
 
\begin{theorem}\label{HVGINVD intro}
Soit $M$ une variété compacte $C^\infty$ de dimension $n\geq 3$. $g$ une métrique  riemannienne qui appartient à $H_2^p(M,T^*M\otimes T^*M)$ avec $p>n/2$. Si 
\begin{equation}\label{inf intro}
\mu_{G}(g)<\frac{1}{4}n(n-2)\omega_n^{2/n}(\inf_{Q\in M}\Ca O_G(Q))^{2/n}
\end{equation}
alors l'équation \eqref{yamabei} admet une solution strictement positive $\varphi\in H_{2,G}^p(M)\subset C^{1-[n/p],\beta}(M)$ $G-$invariante. 
\end{theorem}

Les résultats sur l'unicité des solutions de l'équation de Yamabe \eqref{yamabei}, connus lorsque la métrique est $C^\infty$, restent valables dans le cas singulier. On obtient le résultat suivant:

\begin{theorem}\label{unique intro}
 Soit $g$ une métrique dans  $H^p_2(M,T^*M\otimes T^*M)$, avec $p>n$. Si $\mu(g)\leq 0$ alors les solutions de l'équation \eqref{yamabei} sont uniques à une constante multiplicative près.
\end{theorem}

Dans cette première partie, on a montré que la majorité des résultats connus sur le problème de Yamabe et certains dans le cas équivariant, lorsque la métrique est $C^\infty$, restent vrais lorsque la métrique satisfait l'hypothèse $(H)$, définie ci-dessus. Une question naturelle que l'on peut se poser est de savoir s'il est possible de supprimer certaines conditions dans l'hypothèse $(H)$. Par exemple, peut on considérer  des métriques dans $H^p_2$, sans qu'elles soit $C^\infty$ dans une boule? La réponse semble difficile et le sujet ne sera pas abordé dans cette thèse (mais sera traité ultérieurement). \\

\subsection*{Deuxième partie}

La deuxième partie de cette thèse est indépendante de la première (elles sont mathématiquement liées, mais aucun résultat de la première partie n'est utilisé dans la seconde partie). \\
On suppose que $(M,g)$ est une variété riemannienne compacte $C^\infty$ de dimension $n\geq 3$. Le but principal des deux chapitres de cette partie est d'étudier la conjecture de Hebey--Vaugon qui s'énonce comme suit:
\begin{conjecture}[E.~Hebey et M.~Vaugon \cite{HV}]\label{HVcon intro}
Soit $G$ un sous groupe d'isométries de $I(M,g)$. Si $(M,g)$ n'est pas conformément difféomorphe à $(S_n, g_{can})$ ou bien si $G$ n'a pas de point fixe, alors l'inégalité stricte suivante a toujours lieu 
\begin{equation}\label{HVI intro}
\inf_{g'\in [g]^G} J(g')<n(n-1)\omega_n^{2/n}(\inf_{Q\in M}\Ca O_G(Q))^{2/n}
\end{equation}
\end{conjecture}
Cette conjecture généralise la conjecture de T.~Aubin \ref{Aubincon intro} puisque:
$$\inf_{g'\in [g]^G} J(g')=4\frac{n-1}{(n-2)}\mu_G(g)$$ 
(si $G=\{\mathrm{id}\}$ les deux conjectures sont identiques). \\
On note par $W_g$ le tenseur de Weyl associé à $g$. Pour tout $P\in M$, on définit $\omega(P)$ par
$$\omega(P)=\inf \{|\beta|\in \mathbb N/\|\nabla^\beta W_g(P)\|\neq 0\},\;\omega(P)=+\infty\text{ si }\forall \beta\;\;\|\nabla^\beta W_g(P)\|= 0$$
où $\beta$ est un multi-indice de longueur $|\beta|$.\\
Pour prouver la conjecture, on doit construire une fonction test $G-$invariante $\phi$ telle que 
$$I_g(\phi)< n(n-1)\omega_n^{2/n}(\inf_{Q\in M}\Ca O_G(Q))^{2/n}$$
Toute la difficulté est dans la construction d'une telle fonction. Dans certains cas, on peut utiliser les fonctions test introduites par T.~Aubin \cite{Aub} et R.~Schoen \cite{Schoen} pour démontrer la conjecture \ref{Aubincon intro}. De nombreux cas ont été traités ainsi par E.~Hebey et M.~Vaugon \cite{HV}, par contre le cas numéro 3 présenté dans le théorème suivant utilise des fonctions test qui sont différentes de celles de T.~Aubin et R.~Schoen.

\begin{theorem}[E.~Hebey et M.~Vaugon]\label{HV theorem intro}
Soit $(M,g)$ une variété riemannienne compacte de dimension $n$ et $G$ un sous groupe d'isométries du groupe $I(M,g)$.
On a toujours:  
$$\inf_{g'\in [g]^G} J(g')\leq n(n-1)\omega_n^{2/n}(\inf_{Q\in M}\Ca O_G(Q))^{2/n}$$
 et l' inégalité  stricte \eqref{HVI intro} est au moins vérifiée dans chacun des cas suivants:
\begin{enumerate}
 \item $G$ opère librement sur $M$
 \item $3\leq \dim M\leq 11$
 \item\label{item 3 intro} Il existe un point $P$ d'orbite minimale (finie) sous $G$ pour lequel soit $\omega(P)>(n-6)/2$, soit $\omega(P)\in \{0,1,2\}$.
 \end{enumerate}
\end{theorem}

Les cas restant pour démontrer complètement la conjecture  sont les cas où $n\geq 12$ et 
$\omega\in\intl 3, [(n-6)/2]\intr$. Dans le chapitre \ref{CHYSr}, on démontre les résultats suivants:

\begin{theorem}\label{con HVa intro}
La conjecture \ref{HVcon intro} est vraie s'il existe un point $P$ d'orbite minimale (finie) pour lequel $\omega(P)\leq15$ ou si le degré de la partie principale de $R_g$, au voisinage de $P$ est plus grand ou égal à $\omega(P)+1$.
\end{theorem} 

\begin{corollary}\label{cor prin intro}
La conjecture \ref{HVcon intro} est vraie si $M$ est de dimension $n\in \intl 3, 37 \intr$.
\end{corollary} 

Ce théorème se démontre en effectuant des calculs longs et délicats (introduits par T.~Aubin \cite{Aub5}). Les fonctions test $\varphi_\e$ choisies sont définies comme suit: pour un point $P$ quelconque de $M$, on pose pour tout $Q\in M$ 

\begin{gather}\label{fonction test intro}
\varphi_{\e}(Q)=(1- r^{\alp} f(\xi))u_{\e}(Q)\\
\text{avec }u_{\e}(Q)=\begin{cases}\biggl(\displaystyle\frac{\varepsilon}{r^2+\varepsilon^2}\biggr)^{\frac{n-2}{2}}-\biggl(\frac{\varepsilon}{\delta^2+\varepsilon^2}\biggr)^{\frac{n-2}{2}} &\mbox{ si }Q\in B_{P}(\delta)\label{uepsilon intro}\\
\hspace{2cm}0 &\mbox{ si }Q\in M-B_{P}(\delta)
\end{cases}
\end{gather}
où $r=d(Q,P)$ est la distance entre $P$ et $Q$. $(r,\xi^j)$ sont les coordonnées  géodésiques  de $Q$ au voisinage de $P$ et 
$B_{P}(\delta)$ est une boule géodésique de centre $P$, de rayon $\delta$, fixé suffisamment petit. $f$ est une fonction qui dépend seulement de  $\xi$ et telle que 
$\int_{S_{n-1}}f d\sigma=0$. C'est la précision sur le choix de cette fonction $f$ qui va permettre d'obtenir les résultats énoncés dans cette deuxième partie. \\

On obtient d'abord le théorème suivant: 
\begin{theorem}\label{theo intro}
Soit $(M,g)$ une variété riemannienne compacte de dimension $n$. Pour tout $P\in M$ tel que $\omega(P)\leq (n-6)/2$, il existe  $f\in C^\infty(S_{n-1})$, d'intégrale  nulle, telle que 
$$\mu(g)\leq I_g(\varphi_\e)< \frac{n(n-2)}{4}\omega_{n-1}^{2/n}$$
\end{theorem}
(Ce résultat généralise donc le théorème de T.~Aubin \cite{Aub}, qui correspond   à $\omega=0$ et qui démontre la conjecture \ref{Aubincon intro}, dans certains cas). La fonction $f$ de ce théorème est définie par
$$f=\sum_{k=1}^qc_k\nu_k\varphi_k$$
où  $\varphi_k$ sont des fonctions propres du Laplacien sphérique de la sphère $S_{n-1}$, $\nu_k$ sont les valeurs propres associées et $q\in \intl 1,[\frac{\omega}{2}]\intr$, les constantes $c_k$ sont données explicitement.  Si $f$ était $G-$invariante, on pouvait construire, à l'aide des $\varphi_k$, des fonctions test $G-$invariantes qui permettraient de démontrer la conjecture \ref{HVcon intro} dans tous les cas. Malheureusement, $f$ n'est $G-$invariante que pour un choix particulier des $c_k$, et ce choix particulier ne permet de montrer la conjecture que dans les cas énoncés dans le théorème  \ref{con HVa intro}.

\chapterstar{Introduction (English version)}
\markboth{INTRODUCTION}{INTRODUCTION}

In the first part of this thesis, we study a certain kind of nonlinear partial differential equations on compact manifolds. Solutions of these PDEs have a geometric meaning. The particularity here is that one of the coefficients of this equations doesn't have the usual regularity, which allow us to obtain a Yamabe theorem with singularities.\\
The Second part is dedicated to the study of Hebey--Vaugon conjecture.

\subsection*{First part}   

Consider $(M,g)$ a compact Riemannian manifold of dimension $n\geq 3$. Denote by $R_g$ the scalar curvature of $g$. The Yamabe problem is the following:

\begin{probleme}\label{yamabe problem introe}
Does there exists a constant scalar curvature metric conformal to $g$?
\end{probleme}
Let $\gt=\varphi^\N g$ be a conformal metric, where $\varphi$ is a smooth positive function. $\gt$ is a solution of the Yamabe problem if and only if $\varphi$ satisfies the following equation:
\begin{equation}\label{yamabeie}
\frac{4(n-1)}{n-2}\Delta_g \varphi+ R_g\varphi=  R_{\gt} \varphi^{\frac{n+2}{n-2}}
\end{equation} 
where $\Delta_g=-\nabla^i\nabla_i$ is the Laplacian of $g$ and $R_{\gt}$ is a constant which plays the role of the scalar curvature of $\gt$. T.~Aubin showed that it is sufficient to prove the following conjecture: 
\begin{conjecturee}[T.~Aubin \cite{Aub}]\label{Aubincon introe}
For every smooth compact Riemannian manifold $(M,g)$ of dimension $n\geq 3$, non conformal to $(S_n,g_{can})$, 
 \begin{equation}\label{caue}
 \mu(M,g)<\mu(S_n,g_{can})
\end{equation}
where $\mu(M,g)=\inf\biggl\{\displaystyle\frac{\int_M |\nabla\psi|^2+\frac{n-2}{4(n-1)}R_g\psi^2\di v}{\|\psi\|^2_{\frac{2n}{n-2}}},\; \psi\in H_1(M)-\{0\}\biggr\}$
\end{conjecturee}
It is known that $\mu(S_n,g_{can})=\frac{1}{4}n(n-2)\omega_n^{2/n}$. The works of T.~Aubin \cite{Aub}, R.~Schoen \cite{Schoen} and H.~Yamabe \cite{Yam} showed that this conjecture is always true, and the Yamabe problem has a solution. Namely, in each conformal class $[g]$, there exists a constant scalar~curvature~metric.\\

Denote by $I(M,g)$ and $C(M,g)$ the isometry group and the conformal group respectively. Let $G$ be a subgroup of $I(M,g)$. E.~Hebey and  M.~Vaugon \cite{HV} studied the equivariant Yamabe problem, which generalizes the Yamabe problem, and which can be formulated in the following way:
\begin{probleme}\label{HVProbleme introe}
Is there some $G-$invariant metric $g_0$ which minimizes the functional
$$J(g')=\frac{\int_MR_{g'}\di v(g')}{(\int_M\di v(g'))^{\frac{n-2}{n}}}$$
where $g'$ belongs to the $G-$conformal class of $g$:
$$[g]^G:=\{\tilde g=e^fg/ f\in C^\infty(M),\; \sigma^*\tilde g=\tilde g\quad \forall \sigma\in G\}$$
\end{probleme} 
E.~Hebey and  M.~Vaugon proved that this problem has always solutions. The positive answer would have two consequences. The first is that there exists a $I(M,g)-$invariant metric $g_0$ conformal to $g$ such that the scalar curvature $R_{g_0}$ is constant. The second is that the following conjecture is true.

\paragraph{\textbf{Lichnerowicz conjecture}}\emph{For every compact Riemannian manifold $(M,g)$ which is not conformal to the unit sphere $S_n$ endowed with its standard metric, there exists a metric $\tilde g$ conformal to $g$ for which $I(M,\tilde g)=C(M,g)$, and the scalar curvature $R_{\tilde g}$ is constant.}\\

In this part, we study the Yamabe problem \ref{yamabe problem introe} (without and in presence of the isometry group), when the metric $g$ is not necessarily smooth. We suppose that the metric is in the Sobolev space $H^p_2$, where $p>n$. Riemann curvature tensor, Ricci tensor and the scalar curvature are in $L^p$. More precisely, we make the following assumption on $g$ :\\

\textbf{Assumption $\boldsymbol{(H)}$:} \emph{$g$ is a metric which belongs to the Sobolev space $H_2^p(M,T^*M\otimes T^*M)$ with $p>n$. There exists a point $P_0\in M$ and $\delta>0$ such that $g$ is smooth in the ball $B_{P_0}(\delta)$.}\\

The problem that we solve is the following:

\begin{probleme}\label{yam singe}
 Let $g$ be a metric satisfying the assumption $(H)$. Does there exists a constant scalar curvature metric $\tilde g$ conformal to $g$?
\end{probleme}
Before  solving this problem, we start by studying these equations:

\begin{equation}\label{dfgfge}
 \Delta_g\varphi+h\varphi=\tilde h\varphi^{\frac{n+2}{n-2}}
\end{equation}
where $h$ is a function in $L^p(M)$ (which makes this work original) and $\tilde h\in \mathbb R$. The metric $g$ is assumed to be smooth. The smoothness of $g$ is not an important point. Indeed, if $g$ is $C^2$, we will obtain the same results. This kind of equations are called "Yamabe type equations". We can not apply for these equations the standard regularity theorems  because of the nonlinearity and the fact that $h\in L^p(M)$. Thus, we establish the following result (it is an adaptation of Trudinger's theorem when $h$ is more regular).
\begin{theoreme}\label{regYS introe}
Let $(M,g)$ be a smooth compact Riemannian manifold of dimension $n\geq 3$, $p$ and $\tilde h$  are two reel numbers such that $p>n/2$. If $\varphi\in H_1(M)$ is  nontrivial, nonnegative, weak solution of  \eqref{dfgfge}, then $\varphi$ is positive and belongs to $ H_2^p(M)\subset C^{1-[n/p],\beta}(M)$.
\end{theoreme} 

For the existence of solutions of \eqref{dfgfge}, we prove that the functional $I_g$, defined for all $\psi\in H_1(M)-\{0\}$ by 
\begin{equation*}
I_g(\psi)=\frac{\int_M |\nabla\psi|^2+h\psi^2\di v}{\|\psi\|^2_{\frac{2n}{n-2}}}\label{cae} 
\end{equation*}
has a minimum $\mu(g)$ if $\mu(g)<\frac{1}{4}n(n-2)\omega_n^{2/n}$ (where $\omega_n$ is the volume of the unit sphere $S_n$). Therefore, we obtain the following: 

\begin{theoreme}\label{cg introe}
Let $(M,g)$ be a smooth compact Riemannian manifold of dimension $n\geq 3$ and $p>n/2$. If 
$$\mu(g)<\frac{1}{4}n(n-2)\omega_n^{2/n}$$ 
then equation \eqref{dfgfge} admits a positive solution $\varphi\in H^p_2(M)\subset C^{1-[n/p],\beta}(M)$, which minimizes the functional $I_g$, where $\beta\in (0,1)$.
\end{theoreme}
If $h$ is $G-$invariant, we define
$$\mu_{G}(g)=\inf_{\psi\in H_{1,G}(M)-\{0\}} I_g(\psi)$$
where $H_{1,G}(M)$ is the space of $G-$invariant functions in $H_1(M)$. We denote by $O_G(Q)$ the orbit of $Q\in M$. Then,
\begin{theoreme}\label{theginv introe}
If $0<\mu_{G}(g)<\frac{1}{4}n(n-2)\omega_n^{2/n}(\inf_{Q\in M}\Ca O_G(Q))^{2/n}$ then equation \eqref{dfgfge}  admits a positive $G-$invariant solution $\varphi\in H_{2,G}^p(M)\subset C^{1-[n/p],\beta}(M)$, which minimizes the functional $I_{g}$.
\end{theoreme}
We prove this theorem by using the variational method (known in the classical case when $h$ is smooth), Sobolev embedding in the presence of symmetries, proven by E.~Hebey and M.~Vaugon \cite{HV2} and the best constant inequality, computed by Z.~Faget \cite{Fag}.\\

In chapter \ref{CHYS}, we consider the particular case when $h=\frac{n-2}{4(n-1)}R_g$ and the metric $g$ satisfies the assumption $(H)$. This case has a geometric meaning. It allows us to solve the problem \ref{yam singe} (Yamabe problem with singularities). The scalar curvature $R_g$ is in $L^p(M)$ and equation \eqref{dfgfge} becomes the Yamabe equation \eqref{yamabeie}. Using theorem \ref{cg intro} to solve problem \ref{yam singe}, it is sufficient to prove the inequality $\mu(g)<\frac{1}{4}n(n-2)\omega_n^{2/n}$ (this inequality has been proven when $g$ is smooth). When $g$ satisfies the assumption $(H)$, we establish some properties (known in the smooth case) : conformal invariance of $\mu(g)$, weak conformal invariance of the conformal Laplacian $L_g=\Delta_g+\frac{n-2}{4(n-1)}R_g$ and the existence of the Green function for this operator. We show afterwards the following theorem : 
\begin{theoreme}\label{conj aub introe}
 Let $M$ be a smooth compact manifold of dimension $n\geq 3$ and $g$ be a Riemannian metric satisfying the assumption $(H)$. If $(M,g)$ is not conformal to  $(S_n, g_{can})$, then $\mu(g)<\frac{1}{4}n(n-2)\omega_n^{2/n}$.
\end{theoreme}

When the metric $g$ is smooth, this theorem solves the conjecture \ref{Aubincon introe}. The arguments used to prove it are still valid when the metric $g$ satisfies the assumption $(H)$. In fact, it is sufficient to construct a test function $\varphi$ which satisfies  $I_g(\varphi)<\frac{1}{4}n(n-2)\omega_n^{2/n}$. Test functions, constructed by  T.~Aubin \cite{Aub} and R.~Schoen \cite{Schoen}, are still useful in the singular case.\\

In the equivariant case (in the presence of the isometry group), the result obtained is the following:

\begin{theoreme}\label{HVGINVD introe}
Let $M$ be a smooth compact manifold of dimension $n\geq 3$ and $g\in H_2^p(M,T^*M\otimes T^*M)$  be a Riemannian metric  with $p>n/2$. If 
\begin{equation}\label{inf introe}
\mu_{G}(g)<\frac{1}{4}n(n-2)\omega_n^{2/n}(\inf_{Q\in M}\Ca O_G(Q))^{2/n}
\end{equation}
then equation \eqref{yamabeie} has a positive $G-$invariant solution  $\varphi\in H_{2,G}^p(M)\subset C^{1-[n/p],\beta}(M)$. 
\end{theoreme}

The known result about uniqueness of solutions of the Yamabe equation \eqref{yamabeie}, for  smooth metrics, is valid in the singular case. Therefore, we have the following result:

\begin{theoreme}\label{unique introe}
 Let $g$ be a metric in $H^p_2(M,T^*M\otimes T^*M)$, with $p>n$. If $\mu(g)\leq 0$ then the solutions of  \eqref{yamabeie} are proportional.
\end{theoreme}
  
In this part, we showed that almost all of the results and properties known about the Yamabe problem, and some properties in the equivariant case, holds in the singular case (when the metric satisfies the assumption $(H)$ defined above). A question that naturally arises is the  possibility of deleting some conditions in the assumption $(H)$. For example, can we consider metrics in $H^p_2$ without the smoothness condition in a ball? The answer seems difficult and this question will not be treated in this thesis.

\subsection*{Second part }

The second part is independent from the first (they are mathematically linked, but the results of the first part are not used in the second).\\
Suppose that $(M,g)$ is a smooth compact Riemannian manifold of dimension $n\geq 3$. The principal goal of the two last chapters of this part is to study Hebey--Vaugon conjecture that  can be stated in the following way:
\begin{conjecturee}[E.~Hebey and M.~Vaugon \cite{HV}]\label{HVcon introe}
Let $G$ be a subgroup of $I(M,g)$. If $(M,g)$ is not conformal to $(S_n, g_{can})$ or if the action of $G$ has no fixed point, then the following inequality holds 
\begin{equation}\label{HVI introe}
\inf_{g'\in [g]^G} J(g')<n(n-1)\omega_n^{2/n}(\inf_{Q\in M}\Ca O_G(Q))^{2/n}
\end{equation}
\end{conjecturee}
This conjecture generalizes naturally T.~Aubin's conjecture  \ref{Aubincon introe}. In fact
$$\inf_{g'\in [g]^G} J(g')=4\frac{n-1}{(n-2)}\mu_G(g)$$ 
(if $G=\{\mathrm{id}\}$ then the two conjectures are the same).\\
Denote by $W_g$ the Weyl tensor associated to $g$. For all $P\in M$, we define $\omega(P)$ by
 $$\omega(P)=\inf \{|\beta|\in \mathbb N/\|\nabla^\beta W_g(P)\|\neq 0\},\;\omega(P)=+\infty\text{ if }\forall \beta\;\;\|\nabla^\beta W_g(P)\|= 0$$
To prove the conjecture, we need to construct a $G-$invariant test function $\phi$ such that  
$$I_g(\phi)< n(n-1)\omega_n^{2/n}(\inf_{Q\in M}\Ca O_G(Q))^{2/n}$$
Thus, all of the difficulties are in the construction of a such function. For some cases, we can use the test functions constructed by T.~Aubin \cite{Aub} and R.~Schoen \cite{Schoen} to prove the conjecture \ref{Aubincon introe}. They have been already proven by E.~Hebey and M.~Vaugon \cite{HV}. But the item \ref{item 3 introe}, presented in the following theorem, uses test functions  
 different than  T.~Aubin and R.~Schoen ones.

\begin{theoreme}[E.~Hebey and M.~Vaugon]\label{HV theorem introe}
Let $(M,g)$ be a smooth compact Riemannian manifold of dimension $n\geq 3$  and $G$ be a subgroup of $I(M,g)$.
We always have :  
$$\inf_{g'\in [g]^G} J(g')\leq n(n-1)\omega_n^{2/n}(\inf_{Q\in M}\Ca O_G(Q))^{2/n}$$
 and inequality \eqref{HVI introe} holds if one of the following items is satisfied. \begin{enumerate}
 \item The action of $G$  on $M$ is free
 \item $3\leq \dim M\leq 11$
 \item\label{item 3 introe}There exists a point $P$ with minimal orbit (finite) under $G$ such that $\omega(P)>(n-6)/2$ or $\omega(P)\in \{0,1,2\}$.
 \end{enumerate}
\end{theoreme}

The remaining case of the conjecture, is the case when  $n\geq 12$ and $\omega\in\intl 3, [(n-6)/2]\intr$.   In chapter \ref{CHYSr},  we prove the following result: 

\begin{theoreme}\label{con HVa introe}
The conjecture \ref{HVcon introe} holds  if there exists a point $P\in M$ with minimal orbit (finite) for which $\omega(P)\leq15$ or if the degree of the leading part of $R_g$ is greater or equal to $\omega(P)+1$, in the neighborhood of this point $P$.
\end{theoreme} 

\begin{corollarye}\label{cor prin introe}
The conjecture \ref{HVcon introe} holds  for every smooth compact Riemannian manifold $(M,g)$ of dimension $n\in [ 3, 37]$.
\end{corollarye} 

We prove this theorem using long and subtle computations (introduced by T.~Aubin \cite{Aub5}). We use the test function $\varphi_\e$, defined in the following way: for an arbitrary fixed point $P$ in $M$, for any $Q\in M$ 

\begin{gather}\label{fonction test introe}
\varphi_{\e}(Q)=(1- r^{\alp} f(\xi))u_{\e}(Q)\\
\text{with }u_{\e}(Q)=\begin{cases}\biggl(\displaystyle\frac{\varepsilon}{r^2+\varepsilon^2}\biggr)^{\frac{n-2}{2}}-\biggl(\frac{\varepsilon}{\delta^2+\varepsilon^2}\biggr)^{\frac{n-2}{2}} &\mbox{ if }Q\in B_{P}(\delta)\label{uepsilon introe}\\
\hspace{2cm}0 &\mbox{ if }Q\in M-B_{P}(\delta)
\end{cases}
\end{gather}
where $r=d(Q,P)$ is the distance between  $P$ and $Q$. $(r,\xi^j)$ is a geodesic coordinates system of $Q$, defined in the neighborhood of $P$, and $B_{P}(\delta)$ is a geodesic ball of center $P$ and of radius $\delta$, fixed sufficiently small, and $f$ is a function depending only on $\xi$  such that $\int_{S_{n-1}}f d\sigma=0$. The choice of the this function $f$ allow us to prove the results of this part.\\

We obtain also the following theorem: 
\begin{theoreme}\label{theo introe}
Let $(M,g)$ be a compact Riemannian manifold of dimension $n\geq 3$. For any $P\in M$ such that $\omega(P)\leq (n-6)/2$, there exists  $f\in C^\infty(S_{n-1})$ with  vanishing mean integral, such that 
$$\mu(g)\leq I_g(\varphi_\e)< \frac{n(n-2)}{4}\omega_{n-1}^{2/n}$$
\end{theoreme}
This result generalizes  T.~Aubin's \cite{Aub} theorem (which corresponds to  $\omega=0$ and proves conjecture \ref{Aubincon introe}). For the  above theorem, the function  $f$ is defined as 
$$f=\sum_{k=1}^qc_k\nu_k\varphi_k$$
where  $\varphi_k$ are the eigenfunctions of the Laplacian on the sphere  $S_{n-1}$, $\nu_k$ are the associated eigenvalues, and  $q\in \intl 1,[\frac{\omega}{2}]\intr$. The constants $c_k$ are given explicitly. If $f$ is $G-$invariant, then we would construct, using  $\varphi_k$, a $G-$invariant test function, which would prove the conjecture \ref{HVcon introe}, in all the remaining cases. Unfortunately, $f$ is only $G-$invariant for a special choice of $c_k$, and this particular choice allows us to prove the conjecture only in the cases stated in theorem  \ref{con HVa introe}.

\chapter{Théorèmes de régularité et généralités }

\renewcommand{\proofname}{\textbf{Preuve}}
Tout au long de cette thèse, on utilise la convention d'Einstein pour les indices. $M$ sera toujours une variété compacte, sans bord, $C^\infty$ de dimension $n\geq 3$, sauf mention contraire. On commence par rappeler les définitions des courbures de Riemann, Ricci, scalaire et de Weyl. 
\section{Les courbures}\label{curvature}
\begin{definition}
 Soient $(M,g)$ une variété riemannienne $C^\infty$ et $\nabla_g$ (ou simplement $\nabla$) la connexion riemannienne associée (i.e. la connexion sans torsion pour laquelle $g$ est à dérivée covariante nulle). On note par $\Gamma(M)$ l'ensemble des champs de vecteurs $C^\infty$ définis sur $M$.
\begin{enumerate}
 \item $X,\;Y$, $Z$ et $T$ étant quatre  champs de vecteurs dans $\Gamma(M)$. La courbure de Riemann $R$ est l'application bilinéaire antisymétrique de $\Gamma(M)\times\Gamma(M)$ dans $Hom(\Gamma(M),\Gamma(M))$, définie par 
$$R(X,Y)Z=\nabla_X\nabla_YZ-\nabla_Y\nabla_XZ-\nabla_{[X,Y]}Z$$
On appelle tenseur de courbure de Riemann de $g$ le champ de tenseur $C^\infty$ quatre fois covariants défini par
$$R(X,Y,Z,T)=g(X,R(Z,T)Y)=R_{ijkl}X^iY^jZ^kT^l$$ 
dans une carte locale; $R_{ijkl}$ sont les composantes du tenseur de courbure. 
\item La courbure de Ricci de $g$ est le champ de tenseur $C^\infty$, deux fois covariants, obtenu en contractant par $g$ le tenseur de courbure de Riemann de $g$ de la manière suivante $$Ric_{ij}=g^{kl}R_{kilj}$$
où $g^{kl}$ sont les composantes de $g^{-1}$.
\item La courbure scalaire de $g$ est la trace du tenseur de Ricci, notée $R_g$. Dans une carte locale $R_g=g^{ij}Ric_{ij}$
\end{enumerate}

\begin{proprietes}\label{prop courbure}
 Soient $X$ un champ de vecteurs et $\omega$ une $1-$forme. Dans un système de coordonnées locales, $(\nabla_{\partial_i} X)^k$ est notée $\nabla_i X^k$  et  $(\nabla_{\partial_i}\omega)_k$ est notée $\nabla_i\omega_k$. Rappelons les formules de permutation des dérivées covariantes suivantes
$$\nabla_{ij}X^l-\nabla_{ji}X^l=R^l_{kij}X^k,\qquad\nabla_{ij}\omega_l-\nabla_{ji}\omega_l=-R^k_{lij}\omega_k$$
où $ R^l_{kij}=g^{lm}R_{mkij}$.\\ 
Pour tout champ de tenseur $C^2$ deux fois covariants $T$:
$$\nabla_{ij}T_{kl}-\nabla_{ji}T_{kl}=-R^m_{kij}T_{ml}-R^m_{lij}T_{km}$$

\end{proprietes}
\end{definition}

On aura l'occasion d'utiliser ces propriétés dans le  chapitre \ref{CHAub} et l'appendice.
\begin{definition}\label{Weyl def}
 La courbure de Weyl $W$ de la variété riemannienne $(M,g)$, de dimension $n\geq 3$ est définie par le champ de tenseurs quatre fois covariants dont les composantes sont
$$W_{ijkl}=R_{ijkl}-\frac{1}{n-2}(R_{ik}g_{jl}-R_{il}g_{jk}+R_{jl}g_{ik}-R_{jk}g_{il})+\frac{R_g}{(n-1)(n-2)}(g_{ik}g_{jl}-g_{il}g_{jk})$$ 
\end{definition}
Le tenseur de Weyl est obtenu à partir du tenseur de courbure de Riemann, en recherchant un tenseur invariant par transformation conforme de la variété: si $\gt=e^fg$ est une métrique conforme à $g$ alors $W_{\gt}=e^fW_{g}$. 
\begin{definition}
Une variété riemannienne $(M,g)$ est dite conformément plate si pour tout $Q\in M$, il existe un voisinage ouvert $\Omega$ de $Q$ et  une métrique $\gt$ conforme à $g$, tels que le tenseur de courbure de Riemann associé à la métrique $\gt$ est identiquement nul sur $\Omega$.
\end{definition}
Le tenseur de Weyl  est identiquement nul si la variété est de dimension 3 ou si elle est conformément plate.
\section{Le Laplacien}
\begin{definition}
 Sur $(M,g)$ une variété riemannienne $C^\infty$, le Laplacien $\Delta_gf$ d'une fonction $f\in C^2(M)$ est l'opposé de la trace de la hessienne de $f$, donné par 
$$\Delta_g f=-\nabla_i\nabla^i f=-g^{ij}\nabla_i\nabla_j f=-g^{ij}(\partial_{ij}f-\Gamma^k_{ij}\partial_kf)$$ 
Dans un système de coordonnées polaires  $(r,\xi^i)$ ($i.e.\; g_{rr}=1$, $g_{r\xi^i}=0$) si $f(r)$  est une fonction radiale alors le Laplacien de $f$ s'écrit
$$\Delta_gf(r)=-f''(r)-\frac{n-1}{r}f'(r)-f'(r)\partial_r\log\sqrt{\det g}$$
\end{definition}
\paragraph{Remarques.}Tout au long de cette thèse, on utilise le Laplacien géométrique défini ci-dessus, avec des valeurs propres positives.\\ 
On  définit le Laplacien $\Delta_g f$ d'une fonction $f\in H_1(M)$ (voir plus bas pour la définition de $H_1(M)$) par :
 pour tout $\psi\in H_1(M)$
$$(\Delta_gf,\psi)_{g,L^2}=(\nabla f,\nabla\psi)_{g,L^2}$$
où $(\cdot,\cdot)_{g,L^2}$ est le produit scalaire standard dans $L^2(M)$ muni de la métrique $g$, dont on omettra la lettre $g$ lorsque il n y a pas d'ambiguïté.

\section{Les espaces de Sobolev}\label{sobddd}

\begin{definition}
Soit $(M,g)$ une variété riemannienne $C^\infty$ de dimension $n$, $p\geq 1$ un nombre réel, $k$ et $r$ sont deux entiers naturels
\begin{enumerate}
\item L'espace de Sobolev $H_k^p(M)$ est le complété de l'espace $\{f\in  C^{\infty}(M),\;
 |\nabla^lf|\in L^p(M)\quad\forall\;0\leq l\leq k \}$ pour la norme
$$\|f\|_{p,k}=\sum_{l=0}^{k}\|\nabla^l f\|_p$$
\item$C^{r,\beta}(M)$ est l'espace de Hölder des fonctions $C^r$ dont la r-ème dérivée appartient
à $$C^\beta(M)=\{f\in C^0(M),\; \|f\|_{C^\beta}:=\|f\|_\infty+\sup_{P\neq Q}
\frac{|f(P)-f(Q)|}{d(P,Q)^\beta}<+\infty\}$$
avec $\beta\in[0,1[$.\\
$C^{0,1}(M)$ est l'ensemble des fonctions lipschitzienne.
\end{enumerate}
L'espace $H^2_k(M)$ est un espace de Hilbert pour le produit scalaire suivant
$$(f,h)_{H_k}=\sum_{l=0}^k(\nabla^l f,\nabla^l h)_{L^2} $$
Dans  la suite, $H^2_k(M)$ est  noté $H_k(M)$.
\end{definition}
La norme correspondante au produit scalaire  sur $H_k(M)$ est équivalente à la norme $\|\cdot\|_{2,k}$. \\
\begin{definition}\label{g sobolev}
 Soit $(M,g_0)$ une variété riemannienne compacte de dimension $n$. On note par $T^*(M)$ le fibré cotangent de $M$. L'espace $H^p_k(M,T^*M\otimes  T^*M)$ est l'ensemble des sections $g$ (des tenseurs 2 fois covariants) telles que dans toute carte exponentielle, les composantes $g_{ij}$ de $g$ sont dans $H^p_k$.
\end{definition}

L'espace $H^p_k(M,T^*M\otimes  T^*M)$ ne dépend pas de la métrique $g_0$. On peut aussi définir cet espace, en utilisant le théorème du plongement isométrique de Nash. Les deux théorèmes qui suivent sont encore valables pour cet espace $H^p_k(M,T^*M\otimes  T^*M)$.

\begin{theorem}[Théorème d'inclusions de Sobolev]\label{incsob}
Soit $(M,g)$ une variété riemannienne compacte de dimension $n$.
\begin{description}
\item[$(i)$]Si $k$ et $l$ deux entiers ($k>l\geq 0$), $p$ et $q$ deux réels ($p>q\geq 1$) 
qui vérifient $1/p=1/q-(k-l)/n$ alors $H_k^q(M)$ est inclus dans $ H_l^p(M)$ et l'inclusion $H_k^q(M)\subset H_l^p(M)$ est continue.
\item[$(ii)$]Si $r\in\mathbb{N}$ et $(k-r)/n>1/q$ alors l'inclusion $H_k^q(M)\subset C^r(M)$ est continue
\item[$(iii)$]Si $(k-r-\beta)/n\geq 1/q$ alors l'inclusion $H_k^q(M)\subset C^{r,\beta}(M)$ est continue avec $\beta\in]0,1[$
\end{description} 
dans tous les cas $H_k^q(M)$ ne dépend pas de la métrique $g$
\end{theorem}
Une preuve détaillée du théorème est donnée dans le livre de T.~Aubin \cite{Aubin}, chapitre 2, celui de Adams \cite{Ada} ou de E.~Hebey \cite{Heb}. On utilisera souvent l'espace de Hilbert $H_1(M)$ muni de la norme 
$$\|\varphi\|_{H_1}^2=\|\varphi\|^2_2+\|\nabla\varphi\|^2_2$$
pour minimiser des fonctionnelles. Cet espace est inclus  continûment dans $L^{q}(M)$, pour tout $q\in [1,2n/(n-2)]$.\\

Kondrakov a montré que les inclusions de Sobolev sont compactes dans les cas suivants:
\begin{theorem}[Kondrakov]\label{kon} Soit $(M_n,g)$ une variété riemannienne compacte. $k$ un entier naturel, $p$ et $q$ deux nombres réels qui vérifient $1\geq 1/p>1/q-k/n>0$ alors
\begin{description}
\item[(i)]l'inclusion $H_k^q(M)\subset L^p(M)$ est compacte
\item[(ii)]l'inclusion $H_k^q(M)\subset C^\alpha(M)$ est compacte 
si $k-\alpha>n/q$ avec $0\leq\alpha<1$ 
\end{description}
\end{theorem}
Grâce aux inclusions de Sobolev, on montre le résultat suivant:
\begin{proposition}\label{algebre}
 Soit $(M,g)$ une variété riemannienne compacte de dimension $n$. Si $p>n/2$ alors $H^p_2(M)$ est une algèbre.
\end{proposition}

\begin{proof}
Il suffit de montrer que si $\varphi$ et $\psi$ sont dans $H^p_2(M)$ alors $\psi\varphi\in H^p_2(M)$. Par les inclusions de Sobolev (théorème \ref{incsob}), $H^p_2(M)\subset C^{\beta}(M)$ donc $\varphi$ et $\psi$ sont continues. Par la compacité de $M$ et la continuité de $\varphi$ et $\psi$: 
$$\nabla (\psi\varphi)=\psi\nabla\varphi+\varphi\nabla\psi\in L^p(M)$$
D'autre part $$\nabla^2(\psi\varphi)=\psi\nabla^2\varphi+\varphi\nabla^2\psi+\nabla\varphi\otimes\nabla\psi+\nabla\psi\otimes\nabla\varphi\in L^p(M)$$
En
 effet, $|\psi\nabla^2\varphi|+|\varphi\nabla^2\psi|\in L^p(M)$ par le même argument que précédemment, et comme 
$$\||\nabla\varphi||\nabla\psi|\|_p\leq \|\nabla\varphi\|_{2p}\|\nabla\psi\|_{2p}$$
 est borné (cf. théorème \ref{incsob}) alors $|\nabla\varphi||\nabla\psi|\in L^p(M)$. D'où $\varphi\psi\in H^p_2(M)$
\end{proof}

\subsection{Théorèmes des espaces de Banach}\label{teb}
\begin{theorem}
Un espace de Banach $\mathcal{B}$ est réflexif si et seulement si sa boule unité fermée est faiblement compact. 
\end{theorem}

Puisque les espaces de Sobolev sont réflexifs, on utilisera ce théorème comme suit: si on a une certaine suite de fonctions $(\varphi_i)_{i\in \mathbb N}$ bornée dans $H_k(M)$ alors il existe une sous-suite $(\varphi_{q_i})_{i\in \mathbb N}$ qui converge vers  $\varphi\in H_k(M)$ et $\liminf_{i\to+\infty}\|\varphi_{q_i}\|_{H_k}\geq\|\varphi\|_{H_k}$    

\begin{theorem}
Soit $p\in]1,+\infty[$ et $(\varphi_i)_{i\in\mathbb{N}}$ une suite bornée dans $L^p(\mathcal{B})$, qui 
converge presque partout vers $\varphi$, alors $\varphi\in L^p(\mathcal{B})$ et $(\varphi_i)$ converge faiblement 
vers $\varphi$ dans  $L^p(\mathcal{B})$.
\end{theorem}

\section{Inégalité de la meilleure constante}

\begin{theorem}[Aubin--Talenti]\label{INM}
Soit $(M,g)$ une variété riemannienne compacte de dimension $n$. Pour tout $\e>0$ il existe $A(\varepsilon)>0$ tel que 
$$\forall \varphi\in H_1^p(M)\quad \|\varphi\|_{p^*}\leq (K(n,p)+\varepsilon)\|\nabla\varphi\|_p+A(\varepsilon)\|\varphi\|_p$$
$$p^*=\frac{np}{n-p} \mbox{ et }K(n,p)=\frac{p-1}{n-p}\biggl(\frac{n-p}{n(p-1)}\biggr)^{1/p}\biggl[\frac{\Gamma(n+1)}{\Gamma(n/p)\Gamma(n+1-n/p)\omega_{n-1}}\biggr]^{1/n} $$ 
$$K(n,1)=\frac{1}{n}\biggl[\frac{n}{\omega_{n-1}}\biggr]^{1/n}$$
\end{theorem}

$K(n,p)$ est la meilleure constante au sens où pour toute constante plus petite qui remplace $K(n,p)$, l'inégalité ci-dessus devient fausse pour une certaine fonction $\varphi\in H^p_1(M)$. La preuve détaillée du théorème de T.~Aubin est reprise dans le livre \cite{Aubin}. Beaucoup de travaux ont été faits depuis sur la validité de cette inégalité (sur les puissances dans cette inégalité aussi) lorsque $\e=0$. Des résultats ont été obtenus par  T.~Aubin et Y.Y.~Li \cite{AL}, R.J.~Biezuner \cite{Bie}, O.~Druet \cite{Dru,Dru2}, E.~Hebey et  M.~Vaugon \cite{HV3,HVI}...\\
Dans le chapitre suivant (cf. théorème \ref{AD}), on généralisera cette inégalité par l'inégalité de Hardy.

\section{L'inégalité de Hardy sur une variété compacte}\label{hardy section}

\begin{definition}\label{distance}
Soit $P$ un point d'une variété riemannienne $(M,g)$. $\rho_P$ est la fonction définie par:
\begin{equation} \rho_P(Q)= \begin{cases}& d(P,Q)\mbox{ si }d(P,Q)<\delta(M)\\
& \delta(M)\mbox{ si }d(P,Q)\geq\delta(M)
\end{cases}
\end{equation}
avec $\delta(M)$ le rayon d'injectivité de la variété $M$ 
\end{definition}
La fonction $\rho$ dépend évidemment du point $P\in M$ que l'on omettra parfois dans les notations.
\begin{definition}
Sur une variété riemannienne $(M,g)$, on définit $L^p(M,\rho^{\gamma})$ comme étant l'espace des 
fonctions  $u$ telles que $\rho^\gamma|u|^p$ soit intégrable. On le munit de la norme $$\|u\|^p_{p,\rho^{\gamma}}:=\int_M \rho^{\gamma}|u|^p\di v$$
où $p\geq 1$ et $\rho$ est la fonction introduite dans la définition précédente.
\end{definition}

\begin{proposition}
 Pour tout $p\geq 1$, $L^p(M,\rho^{\gamma})$ muni de la norme $\|\cdot\|_{p,\rho^{\gamma}}$ est un espace de Banach
\end{proposition}

\begin{proof}
La complétude de l'espace $L^p(M,\rho^{\gamma})$ pour la norme  $\|\cdot\|_{p,\rho^{\gamma}}$ découle du fait que $L^p(M)$ est un espace complet et que  $\|u\|_{p,\rho^{\gamma}}=\|\rho^{\gamma/p}u\|_p$ pour tout $u\in L^p(M,\rho^{\gamma})$
\end{proof}

\begin{theorem}[\textbf{Inégalité de Hardy}]\label{IH}
Pour toute fonction $u\in C_o^\infty(\mathbb R^n)$, il existe une constante $c>0$ telle que
$$\||x|^\gamma u\|_p\leq c\||x|^\beta\nabla_l u\|_q$$
où $1\leq q\leq p \leq qn/(n-lq)$, $\gamma=\beta-l+n(1/q-1/p)>-n/p$ et $n>lq$
\end{theorem}
Ce type d'inégalité à une dimension a été introduite par Hardy, puis généralisée pour toute dimension, le livre de 
V.G.~Maz'ja~\cite{Maz} est une bonne référence où on trouvera la preuve de ce théorème.  
Dans notre étude, on s'intéresse à cette inégalité dans le cas où $\beta=0$ et $l=1$. Dans ce cas précis, 
la constante $c=K(n,q,\gamma)$ est la meilleure constante dans l'inégalité ci-dessus. Si $p\gamma>-q$, cette 
constante est atteinte pour la fonction 
$$x\mapsto (1+|x|^{(q+p\gamma)/(q-1)})^{(q-n)/(q+p\gamma)}$$ et $K(n,q,-q)=q/(n-q)$. (cf. \cite{Chu}, \cite{Lieb})
\begin{theorem}\label{AD}
Soit $(M,g)$ une variété riemannienne compacte de dimension $n$ et $p,\,q\mbox{ et }\gamma$ des nombres réels qui satisfont 
$(\gamma+n)/p=-1+n/q>0$ et $1\leq q \leq p\leq qn/(n-q)$. Pour tout $\varepsilon>0$, 
il existe $A(\varepsilon,q,\gamma)$ tel que
\begin{equation}\label{abc}
\forall u\in H^q_1(M)\quad \|u\|_{p,\rho^\gamma}\leq(K(n,q,\gamma)+
\varepsilon)\|\nabla u\|_q+A(\varepsilon,q,\gamma)\|u\|_q
\end{equation}
en particulier  $K(n,q,0)=K(n,q)$ la meilleure constante dans l'inégalité de Sobolev
\end{theorem} 

\begin{proof}
La preuve de ce théorème est quasiment identique à celle de T.~Aubin (voir \cite{Aubin}, chapitre 2)
dans le cas des inclusions de Sobolev sur les variétés riemanniennes complètes à courbure bornée.\\ 

On commence par montrer le lemme suivant:
\begin{lemma}\label{aa}Pour tout $f\in H^q_1(M)$ à support dans $B_P(\delta)$  
\begin{equation*}
\|f\|_{p,\rho^\gamma}\leq K_\delta(n,q,\gamma)\|\nabla f\|_q
\end{equation*}
avec $B_P(\delta)$ une boule de centre $P$ et de rayon $\delta<\delta(M)$. Lorsque $\delta \to 0$, $K_\delta(n,q,\gamma)\to K(n,q,\gamma)$
\end{lemma}
\paragraph{\small\emph{Preuve du lemme.}} On se place dans un système de coordonnées géodésiques $\{r,\theta^i\}$, 
centré en $P$. Soit $\varepsilon>0$ donné, si $\delta$ est choisi suffisamment petit, on a les estimées 
de la métrique suivantes (Aubin \cite{Aubin}, p. 20) :
\begin{equation*}
1-\varepsilon\leq \sqrt{g_{\theta^i\theta^i}(r,\theta)}\leq 1+\varepsilon \mbox{ et } (1-\varepsilon)^{n-1}\leq 
\sqrt{\det g(r,\theta)}\leq (1+\varepsilon)^{n-1}
\end{equation*}
où $g=dr^2+r^2g_{\theta^i\theta^j}d\theta^i d\theta^j$.\\
Si on pose $\tilde f(x)=f(\exp_Px)$, on obtient une fonction bien définie sur $\mathbb R^n$ à 
support dans $\{x\in\mathbb R^n;\; |x|<1 \}$ qui vérifie, d'après le théorème \ref{IH}:
\begin{equation*}
\biggl(\int_{\mathbb R^n}|x|^\gamma|\tilde f|^pdx\biggr)^{1/p}\leq K(n,q,\gamma)\biggl(\int_{\mathbb R^n}
|\nabla\tilde f|^qdx\biggr)^{1/q}
\end{equation*}
de plus si $Q=\exp_Px\in B_P(\delta)$ alors $$|x|=d(P,Q)=\rho(Q)\mbox{ et } 
(1-\varepsilon)|\nabla \tilde f|_{\mathcal E}(x)\leq |\nabla f|_g(\exp_P x)$$
On déduit que 
$$\|f\|_{p,\gamma}\leq (1+\varepsilon)^{(n-1)/p}\|\tilde f\|_{p,\gamma}
\mbox{ et }\|\nabla f\|_q\geq (1-\varepsilon)^{1+(n-1)/q}\|\nabla \tilde f\|_q$$
Finalement 
$$\|f\|_{p,\rho^\gamma}\leq K_\delta(n,q,\gamma)\|\nabla f\|_q$$
avec $K_\delta(n,q,\gamma)=(1-\varepsilon)^{-1+(1-n)/q}(1+\varepsilon)^{(n-1)/p}K(n,q,\gamma)$. Ce qui achève la preuve du lemme.\\

Pour terminer la preuve du théorème, on considère un recouvrement fini $\{B_{P_i}(\delta)\}_{1\leq i\leq m}$ de $M$ qui 
existe puisque la variété est compacte. Soit $\{h_i\}_{1\leq i\leq m}$ une partition de l'unité associée à ce recouvrement. 
On pose 
$$\eta_i=\frac{h_i^{[q]+1}}{\sum_{k=1}^m h_k^{[q]+1}}$$
où $[q]$ est la partie entière de $q$. $\{B_{P_i}(\delta),\eta_i\}_{1\leq i\leq m}$ est aussi une partition de l'unité de 
$M$ et $\eta_i^{1/q}\in C^1(M)$, donc
il existe $H>0$ tel que, pour tout $i\leq m$: $|\nabla\eta_i^{1/q}|\leq H$\\
Pour tout $u\in H_1^q(M)$, on a
$$\|u\|_{p,\rho^\gamma}^q=\|u^q\|_{p/q,\rho^\gamma}=\|\sum_{i=1}^m\eta_i u^q\|_{p/q,\rho^\gamma}\leq
\sum_{i=1}^m\|\eta_i u^q\|_{p/q,\rho^\gamma}\leq \sum_{i=1}^m\|\eta^{1/q}_i u\|^q_{p,\rho^\gamma}$$ 
Or d'après le lemme \ref{aa}, on a pour tout $i\leq m$ 
$$\|\eta^{1/q}_i u\|^q_{p,\rho^\gamma}\leq K^q_\delta(n,q,\gamma)\|\nabla (\eta^{1/q}_i u)\|^q_q$$
donc
\begin{align*}
\|u\|_{p,\rho^\gamma}^q &\leq K^q_\delta(n,q,\gamma)\sum_{i=1}^m\int_M(|\nabla\eta^{1/q}_i||u|+\eta^{1/q}_i|\nabla u|)^q\di v\\
&\leq K^q_\delta(n,q,\gamma)\sum_{i=1}^m\int_M \eta_i|\nabla u|^q+\mu|\nabla u|^{q-1}|\nabla\eta_i^{1/q}|
\eta_i^{(q-1)/q}|u|+\nu|\nabla\eta^{1/q}_i|^q|u|^q\di v\\
&\leq K^q_\delta(n,q,\gamma)(\|\nabla u\|_q^q+\mu m H \|\nabla u\|_q^{q-1}\|u\|_q+\nu mH^q\|u\|_q^q)
\end{align*}
car il existe $\mu,\;\nu\in\mathbb R_+$ tel que pour tout $t\geq 0,\quad (1+t)^q\leq 1+\mu t+\nu t^q$\\
On a aussi pour tout $z,\;y,\;\lambda\in\mathbb R_+^*\qquad qz^{q-1}y\leq \lambda (q-1)z^q+\lambda^{1-q}y^q$\\
Si on pose $z=\|\nabla u\|_q$, $y=\|u\|_q$ et $\lambda=q\varepsilon_0/(\mu mH(q-1))$ avec $\varepsilon_0>0$ petit, on obtient
$$\|u\|_{p,\rho^\gamma}^q\leq K^q_\delta(n,q,\gamma)[(1+\varepsilon_0)\|\nabla u\|_q^q+A(\varepsilon_0)\|u\|_q^q] $$
On peut choisir $\delta$ et $\varepsilon_0$ suffisamment petits de sorte que $K_\delta(n,q,\gamma)(1+\varepsilon_0)^{1/q}
\leq K(n,q,\gamma)+\varepsilon$ et si on pose $A(\varepsilon,q,\gamma)=(K(n,q,\gamma)+\varepsilon)A(\varepsilon_0)^{1/q}$ 
alors l'inégalité \eqref{abc} est établie
\end{proof}

\begin{theorem}\label{INH}
Soit $(M,g)$ une variété riemannienne compacte de dimension $n$. 
\begin{enumerate}
\item Si $(\gamma+n)/p=-1+n/q>0$ et $1\leq q\leq p$ alors l'inclusion $H^q_1(M)\subset L^p(M,\rho^\gamma)$ est continue.
\item Si $(\gamma+n)/p>-1+n/q>0$, $\gamma\leq 0$ et $q\leq p$ alors cette inclusion est compacte.
\end{enumerate}  
\end{theorem}  

\begin{proof}
La preuve de la première partie de ce théorème est évidente compte tenu de l'inégalité démontrée dans le théorème \ref{AD}. 
La seconde partie du théorème est établie si on montre que 
$H^q_1(M)\subset L^r(M)\subset L^p(M,\rho^\gamma)$ 
continûment, où la première inclusion est compacte pour un certain $r\geq 1$ que l'on déterminera.\\
D'après l'inégalité de Hölder, on a pour tout $u\in H^q_1(M)$
\begin{equation*}
\|u\|_{p,\rho^\gamma}^p=\int_M\rho^{\gamma}|u|^p\di v\leq \biggl(\int_M\rho^{\gamma r'}\di v\biggr)^{1/r'}\|u\|^p_r
\end{equation*} 
où $r'=r/(r-p)$. Pour que le second membre de cette inégalité soit fini, il suffit que 
$\gamma r/(r-p)>-n$, pour le premier facteur, et $1/r>1/q-1/n$, pour le second facteur. De plus le théorème de Kondrakov \ref{kon} assure que si $r\geq 1$ satisfait la deuxième inégalité alors l'inclusion $H^q_1(M)\subset L^r(M)$ est compacte. On en déduit que l'on doit avoir 
$$\frac{n}{r}<\frac{\gamma+n}{p}\mbox{ et }\frac{n}{r}>-1+\frac{n}{q}$$
Puisque $(\gamma+n)/p>-1+n/q$ par hypothèse alors, pour que $u\in L^p(M,\rho^\gamma)$ et que l'inclusion soit compacte, 
il suffit de poser 
$$\frac{n}{r}=\frac{1}{2}(\frac{\gamma+n}{p}-1+\frac{n}{q}) $$
Comme $\gamma\leq 0$ on a $n/r< (\gamma+n)/p\leq n/p $ donc $r> p\geq 1$.
\end{proof}

\paragraph{\small{Remarque.}} Puisque la fonction $\rho$ (cf. définition \ref{distance}) dépend de $P\in M$, l'espace $L^p(M,\rho_P^\gamma)$ dépend aussi du point $P$ choisi, et si $P\neq P'$, il n'y a pas en général d'inclusions entre $L^p(M,\rho_{P}^\gamma)$ et $L^p(M,\rho_{P'}^\gamma)$. Cependant  les inclusions et les inégalités qu'on a déjà montrées dans les théorèmes \ref{AD} et \ref{INH} sont valables pour tout point $P\in M$.

\section{La régularité des solutions de l'équation de type Yamabe}

Lorsque on cherche des solutions d'équations aux dérivées partielles, la première étape donne fréquemment  des solutions faibles (dans notre cas, elles seront dans $H_1(M)$). Dans la plupart des cas on trouve la régularité des solutions en appliquant le théorème de régularité pour les opérateurs elliptiques à coefficients continus suivant:

\begin{theorem}\label{reg}
Soient $\Omega$ un ouvert de $\mathbb{R}^n$ et $L$ un opérateur linéaire d'ordre 2 
uniformément elliptique qui s'écrit sous la forme 
\begin{equation}\label{elli}
L(u)=a^{ij}\partial_{ij}u+ b^i\partial_iu+hu
\end{equation}
où $a^{ij},\;b^i\mbox{ et }h $ sont des fonctions bornées dans $C^k$, $k\in \mathbb N$.\\
Soit $u$ une solution  de l'équation $Lu=f$ au sens des distributions. 
\begin{enumerate}
\item[(i)] Si $f\in C^{k,\alpha}(\Omega)$ alors $u\in C^{k+2,\alpha}(\Omega)$
\item[(ii)]Si $f\in H_k^p(\Omega)$ alors $u\in H_{k+2}^p(\Omega)$
\end{enumerate}
\end{theorem}
Ce théorème est standard, on peut en trouver une preuve dans le livre de D.~Gilbarg et N.~Trudinger~\cite{GT}.\\
Les deux théorèmes suivants permettent de trouver la meilleure régularité des solutions d'un certains type d'équations. Ils sont fondamentaux pour la suite, associés aux théorèmes de régularité habituels pour les opérateurs elliptiques ci-dessus. N.~Trudinger \cite{Trud} avait montré que les solutions faibles de l'équation de Yamabe \eqref{yamabe} (voir chapitre 3) sont toujours $C^\infty$ grâce à ces deux théorèmes. Le premier théorème a été utilisé implicitement  par H.~Yamabe \cite{Yam}  et  on peut en trouver une preuve dans l'article de J.~Serrin \cite{Ser}. Le deuxième théorème est plus spécifique  car il s'applique à des équations de type Yamabe qu'on étudiera dans le prochain chapitre. 

\begin{theorem}\label{gef}
 Sur une variété riemannienne compacte $(M,g)$, si $u\geq 0$ est une solution faible dans $H_1(M)$, non triviale, de l'équation $\Delta u+hu=0$, c'est à dire si
$$\forall v\in H_1(M)\qquad (\nabla u,\nabla v)_{L^2}+(hu,v)_{L^2}=0$$
avec $h\in L^p(M)$ et $p>n/2$, alors $u\in C^{1-[n/p],\beta}(M)$ et est strictement positive bornée.\\
$[n/p]$ est la partie entière de $n/p$ et  $\beta\in]0,1[$.
\end{theorem}

Observons que si $u$ est une fonction qui satisfait les hypothèses de ce théorème alors $\Delta u\in L^p(M)$. Par le théorème de régularité  \ref{reg}, $u\in H^p_2(M)$ et par les inclusions de Sobolev,  $u\in C^{1-[n/p],\beta}(M)$\\
Le théorème \ref{gef} permet de montrer le théorème suivant: 
\begin{theorem}\label{regYS}
Soit $(M,g)$ une variété riemannienne compacte $C^\infty$ de dimension $n$. $p$ et $\tilde h$ sont deux nombres réels, avec $p>n/2$. Si $\varphi\in H_1(M)$ une solution faible positive non triviale de l'équation 
\begin{equation}\label{etyam}
 \Delta_g\psi+h\psi=\tilde h\psi^{\frac{n+2}{n-2}}
\end{equation}
 
alors $\varphi\in H_2^p(M)\subset C^{1-[n/p],\beta}(M)$ et $\varphi$ est strictement positive.
\end{theorem}

\begin{proof}
Pour montrer ce théorème, il suffit de montrer qu'il existe $\e>0$ tel que $\varphi\in L^{(\e+2n)/(n-2)}(M)$. En effet si $\varphi$ satisfait aux hypothèses du théorème et qu'elle est dans $L^{(\e+2n)/(n-2)}(M)$, alors elle est solution de l'équation 
$$\Delta_gu+(h-\tilde h\varphi^{\frac{4}{n-2}})u=0$$
avec $h-\tilde h\varphi^{\frac{4}{n-2}}\in L^r(M)$ et $r=\min(p,\frac{2n+\e}{4})>n/2$. Par le théorème \ref{gef}, on en déduit que $\varphi$ est strictement positive bornée. Par le théorème de régularité \ref{reg} et les inclusions de Sobolev, on montre que $\varphi$ appartient à $H^p_2(M)$ avec $p>n/2$.\\ 
Soient $l$ un nombre réel strictement positif et  $H$, $F$ deux fonctions réelles continues sur $\mathbb R_+$ définies par:
\begin{align*}
 H(t) &=\begin{cases} 
t^\gamma &\text{ si } 0\leq t\leq l\\
l^{q-1}(ql^{q-1}t-(q-1)l^q) &\text{ si }  t>l
      \end{cases}\\
F(t) &=\begin{cases}
t^q &\text{ si } 0\leq t\leq l\\
 ql^{q-1}t-(q-1)l^q \qquad &\text{ si }  t>l
     \end{cases}
\end{align*}

\begin{equation*}
\text{où }\gamma=2q-1,\text{ et } 1<q<\frac{n(p-1)}{p(n-2)}  
\end{equation*}
Comme $\varphi$ est une fonction positive appartenant à $ H_1(M)$, $H\circ\varphi$ et $F\circ \varphi$ sont également dans $H_1(M)$. Notons que pour tout $t\in\mathbb R_+ -\{l\}$
\begin{equation}\label{forl}
 qH(t)=F(t)F'(t),\; (F'(t))^2\leq qH'(t)\text{ et }F^2(t)\geq tH(t)
\end{equation}
Si $\varphi $ est une solution faible de l'équation \eqref{etyam} alors
\begin{equation}\label{tyfai}
\forall\psi\in H_1(M)\quad \int_M\nabla\varphi\cdot\nabla\psi \di v+\int_Mh\varphi\psi\di v=\tilde h\int_M\varphi^{N-1}\psi \di v
\end{equation}
où $N=2n/(n-2)$.\\
On choisit $\psi=\eta^2H\circ\varphi$, où $\eta$ est une fonction de classe $C^1$ à support dans la boule $B_P(2\delta)$ de rayon $2\delta$ suffisamment petit telle que $\eta=1$ sur $B_P(\delta)$. Si on substitue dans \eqref{tyfai}, on obtient
\begin{equation}\label{refref}
 \int_M \eta^2H'\circ\varphi|\nabla\varphi|^2\di v+2\int_M\eta H\circ\varphi\nabla\varphi\cdot\nabla\eta\di v=\tilde h\int_M\varphi^{N-1}\eta^2 H\circ\varphi\di v-\int_Mh\varphi\eta^2H\circ\varphi\di v
\end{equation}
On pose $f=F\circ\varphi$. On estimera les quatre intégrales ci-dessus, en utilisant la fonction $f$ et les relations \eqref{forl}. On a $\nabla f=F'\circ\varphi\nabla\varphi$ donc,  en utilisant la deuxième relation de \eqref{forl}
$$|\nabla f|^2=(F'\circ\varphi)^2|\nabla\varphi|^2\leq qH'\circ\varphi|\nabla\varphi|^2$$ 
On en déduit que la première intégrale de l'égalité \eqref{refref} est minorée par
$$\frac{1}{q}\|\eta \nabla f\|_2^2\leq\int_M \eta^2H'\circ\varphi|\nabla\varphi|^2\di v$$
La première relation de \eqref{forl} et l'inégalité de Cauchy--Schwarz impliquent que la deuxième intégrale de \eqref{refref} est minorée par:
$$2\int_M\eta H\circ\varphi\nabla\varphi\cdot\nabla\eta\di v=\frac{2}{q}\int_M\eta f\nabla f\nabla\eta\di v
\geq \frac{-2}{q}\|f\nabla\eta\|_2\|\eta\nabla f\|_2$$
Grâce à la dernière relation de \eqref{forl}, on a $\varphi H\circ\varphi\leq f^2$. Les deux intégrales de droite dans \eqref{refref} sont donc majorées par:
$$\biggl|\tilde h\int_M\varphi^{N-1}\eta^2 H\circ\varphi\di v-\int_Mh\varphi\eta^2H\circ\varphi\di v\biggr|\leq |\tilde h|\|\varphi\|^{4/(n-2)}_{N,2\delta}\|\eta f\|^2_N+\|h\|_p\|\eta f\|_{2p/(p-1)}^2$$  
où $\|\varphi\|^N_{N,r}=\int_{B_P(r)}\varphi^N\di v$. Si on regroupe ces estimées, l'égalité \eqref{refref} devient:
\begin{equation}\label{refre}
 \|\eta \nabla f\|_2^2-2\|f\nabla\eta\|_2\|\eta\nabla f\|_2\leq q(|\tilde h|\|\varphi\|^{4/(n-2)}_{N,2\delta}\|\eta f\|^2_N+\|h\|_p\|\eta f\|_{2p/(p-1)}^2)
\end{equation}
Remarquons que pour tout nombre réel positif $a,\;b,\;c\text{ et }d$, si $a^2-2ab\leq c^2+d^2$ alors $a\leq c+d+2b$. En utilisant cette remarque, l'inégalité \eqref{refre} devient:
\begin{equation}\label{inega Est}
 \|\eta \nabla f\|_2\leq \sqrt{q|\tilde h|}\|\varphi\|^{2/(n-2)}_{N,2\delta}\|\eta f\|_N+\sqrt{q\|h\|_p}\|\eta f\|_{2p/(p-1)}+2 \|f\nabla\eta\|_2
\end{equation}
 Par les inclusions de Sobolev (cf. théoème \ref{incsob}) on sait qu'il existe une constante $c>0$ qui dépend seulement de $n$ telle que  
$$\|\eta f\|_N\leq c(\|\eta \nabla f\|_2+\|f\nabla\eta\|_2+\|\eta f\|_2)$$
Le choix de $q$ ($q< N$) et l'inégalité \eqref{inega Est} permettent d'écrire 
\begin{equation*}
  (1-c\sqrt{N|\tilde h|}\|\varphi\|^{2/(n-2)}_{N,2\delta})\|\eta f\|_N\leq c\bigl(\sqrt{N\|h\|_p}\|\eta f\|_{2p/(p-1)}+3 \|f\nabla\eta\|_2+\|\eta f\|_2\bigr)
\end{equation*}
On choisit $\delta$ suffisamment petit pour que 
$$\|\varphi\|^{2/(n-2)}_{N,2\delta}\leq 1/(2c\sqrt{N|\tilde h|})$$ ensuite on fait tendre $l$ vers $+\infty$,  on en déduit qu'il existe une constante $C>0$ qui dépend  de $n,\;\delta, \|\eta\|_\infty,\;\|\nabla\eta\|_\infty,\;\|h\|_p$ et $|\tilde h|$  telle que
$$\|\varphi^q\|_{N,2\delta}\leq C(\|\varphi^q\|_2+\|\varphi^q\|_{2p/(p-1)})$$
Comme $\frac{2p}{p-1}q<N$ et que $\varphi$ est bornée dans $L^N$ on a 
$$\|\varphi\|_{qN,2\delta}\leq C$$ 
Si $(\eta_i)_{i\in I}$ est une partition de l'unité subordonnée au recouvrement $\{B_{P_i}(\delta)\}_{i\in J}$ de la variété $M$  alors
$$\|\varphi\|^{qN}_{qN}=\sum_{i\in I}\|\eta_i\varphi\|^{qN}_{qN,\delta_i}\leq C$$ 
on en déduit que $\varphi\in L^{qN}$ avec $qN>N$. En tenant compte de ce qui a été dit au début de la preuve, le théorème est démontré.
\end{proof}

\begin{proposition}\label{delta u f}
Soit $(M,g)$ une variété riemannienne compacte, si $u$ est une solution faible dans $H_1(M)$ de l'équation $\Delta u+hu=f$, où $h$ et $f$ sont deux fonctions telles que $h\in L^p(M)$ et $f\in L^q(M)$, $p>n/2$ et $q\geq 1$, alors $u\in H^{\min(p,q)}_2(M)$
\end{proposition}

\begin{proof}
Distinguons les deux cas $q\geq p$ et $q<p$.
\begin{enumerate} 
\item[$(i)$] Si $q\geq p$ . Supposons que $u\in L^{s_i}(M)$ et satisfait les hypothèses de la proposition. Alors $hu\in L^{\frac{ps_i}{p+s_i}}(M)$, donc $\Delta u\in L^{\frac{ps_i}{p+s_i}}(M)$ car $ps_i/(p+s_i)<q$. Le théorème de régularité \ref{reg} assure que $u\in H^{\frac{ps_i}{p+s_i}}_2(M)$. Ensuite, les inclusions de Sobolev $H_2^r(M)\subset L^s(M)$ si $r\leq n/2$ avec $s=nr/(n-2r)$ et $H^r_2(M)\subset C^{1-[n/r],\beta}(M)$ si $r>n/2$ permettent d'écrire 
\begin{equation*}
\begin{cases}
s_0=N\\
 u \in L^{s_{i+1}}(M)\text{ où } s_{i+1}=\frac{nps_i}{np-(p-2n)s_i}& \mbox{ si }s_i\leq \frac{np}{2p-n}\\
u \in H^p_2(M) & \mbox{ si }s_i> \frac{np}{2p-n}
\end{cases}
\end{equation*}
S'il existe $i\in \mathbb N$ tel que  $s_i>\frac{np}{2p-n}$ ce qui est équivalent à $\frac{ps_i}{p+s_i}>n/2$ alors $u\in C^{0,\beta}(M) $, ce qui implique que $\Delta u\in L^p(M)$, donc $u \in H^p_2(M)$ et la proposition est démontrée. S'il existe $i\in \mathbb N$ tel que  $s_i=\frac{np}{2p-n}$ alors $u\in L^\infty(M)$ et on conclut par le théorème de régularité que $u \in H^p_2(M)$. Supposons que pour tout $i\in \mathbb N$, $s_i<\frac{np}{2p-n}$ alors la suite  $(s_i)_{i\in\mathbb{N}}$ est croissante majorée, donc elle converge vers $s=0$ ce qui est impossible. 
\item[$(ii)$] Supposons que $q<p$ alors on doit montrer que $u\in H^q_2(M)$. Supposons que $u\in L^{s_i}(M)$ et satisfait les hypothèses de la proposition. Ceci implique que $hu\in L^{\frac{ps_i}{p+s_i}}(M)$ donc $\Delta u\in L^{r_i}(M)$ avec $r_i=\min(q,\frac{ps_i}{p+s_i})$. Par le théorème de régularité~\ref{reg}, $u\in H^{r_i}_2(M)$. Donc
\begin{equation*}
\begin{cases}
s_0=N\\
 u \in L^{s_{i+1}}(M)\text{ où } s_{i+1}=\frac{nr_i}{n-2r_i}& \mbox{ si }r_i\leq n/2\\
u \in H^q_2(M) & \mbox{ si }r_i>n/2
\end{cases}
\end{equation*}
En effet, comme $u\in H^{r_i}_2(M)$, s'il existe $i\in\mathbb N$ tel que $r_i>n/2$ alors $u$ est continue, donc $\Delta u=hu-f\in L^q(M)$ d'où $u\in H^q_2(M)$. Si $r_i=n/2$ alors $u\in L^\infty(M)$ donc $hu-f\in L^q(M)$, d'où $u\in H^q_2(M)$.\\ 
Le seul cas qui reste à étudier est bien le cas où $r_i< n/2$  pour tout $i\in\mathbb N$. Dans ce cas, s'il existe $i\in \mathbb N$ tel que $q\leq\frac{ps_i}{p+s_i}$ alors $r_i=q$ et $u\in H^q_2(M)$. Sinon pour tout $i\in\mathbb N$, $r_i=\frac{ps_i}{p+s_i}<n/2$ et  on retrouve le cas $(i)$ où la suite $(s_i)$ est croissante majorée et converge vers 0, ce qui est absurde.
\end{enumerate}
\end{proof}

\begin{proposition}\label{delta fu}
Soit $(M,g)$ une variété riemannienne compacte de dimension $n$ et soit $L:=\Delta+h$ un opérateur linéaire avec $h\in L^p(M)$ et $p>n/2$. Si la plus petite valeur propre $\lambda$ de $L$  est strictement positive alors
\begin{itemize}
 \item[i.] $L$ est coercif, autrement dit il existe $c>0$ tel que
$$\forall\psi\in H_1(M)\quad (L\psi,\psi)_{L^2}\geq c(\|\nabla\psi\|^2_2+\|\psi\|^2_2)$$
\item[ii.] pour tout $q>2n/(n+2)$,  $L: H_2^{\min(p,q)}(M)\longrightarrow L^q(M)$ est inversible  
\end{itemize}

\end{proposition}

\begin{proof} $L$ admet une plus petite valeur propre, car si $\lambda$ est une valeur propre de fonction propre $\psi$ alors il existe $C>0$ tel que
$$\lambda\|\psi\|_2^2=(L\psi,\psi)_{L^2}=\|\nabla\psi\|_2^2+\int_Mh\psi^2\di v\geq -\|h\|_p\|\psi\|_{2p/(p-1)}^2\geq -C\|h\|_p\|\psi\|_2^2$$

Donc $\lambda\geq -C\|h\|_p$. Si $\lambda$ est la plus petite valeur propre de $L$ alors
$$\lambda=\inf_{\varphi\in H_1(M)-\{0\}}\frac{E(\varphi)}{\|\varphi\|_2^2}$$
où $$E(\varphi)=(L\varphi,\varphi)_{L^2}=\int_M|\nabla\varphi|^2+h\varphi^2\di v$$ 
Alors pour tout $\varphi\in H_1(M)$
\begin{equation}\label{vp lam}
E(\varphi)\geq \lambda\|\varphi\|_2^2
\end{equation}

Supposons que $L$ ne soit pas coercif, alors il existe une suite $(\psi_i)_{i\in\mathbb N}$ dans $H_1(M)$ qui satisfait $$E(\psi_i)<\frac{1}{i}(\|\nabla\psi_i\|^2_2+vol(M)^{2/n})\mbox{ et }\|\psi_i\|_N=1$$
ce qui entraîne 
$$(1-\frac{1}{i})E(\psi_i)<\frac{vol(M)^{2/n}}{i}-\frac{1}{i}\int_Mh\psi_i^2\di v$$
Puisque $|\int_Mh\psi_i^2\di v|\leq \|h\|_{n/2}$, $\lim_{i\to +\infty}E(\psi_i)\leq 0$. D'autre part $E(\psi_i)\geq \lambda\|\psi_i\|_2^2$ avec $\lambda>0$. Ce qui est impossible.\\
Il est clair que $L$ est injective car si $L\psi=0$ alors par l'inégalité \eqref{vp lam} $,\varphi=0$.\\
Soit $f\in L^q(M)$ avec $q>2n/(n+2)$. Montrons que l'équation 
\begin{equation}\label{equation}
 \Delta\varphi+h\varphi=f
\end{equation}
admet une solution $\psi\in H^{\min(p,q)}_2(M)$. On  minimise la fonctionnelle $E$ définie au début de la preuve, pour cela on pose
\begin{equation*}
\mu=\inf\{E(\varphi)/\varphi\in H_1(M),\; \int_Mf\varphi\di v=1\}
\end{equation*}
Soit $(\psi_i)_{i\in \mathbb N}$ une suite dans $H_1(M)$ qui minimise $E$,  alors $$\lim_{i\to+\infty}E(\psi_i)=\mu\text{ et }\int_Mf\psi_i\di v=1$$
Sans perte de généralité, on peut supposer que pour tout entier naturel $i$, $E(\psi_i)\leq\mu+1$. Ce qui implique 
$$c(\|\nabla\psi_i\|^2_2+\|\psi_i\|^2_2)\leq E(\psi_i)\leq\mu+1$$
car $L$ est coercif. On en conclut que la suite $(\psi_i)_{i\in \mathbb N}$ est bornée dans $H_1(M)$. Par le théorème de Banach (voir section \ref{teb}) et le théorème de compacité de Kondrakov \ref{kon}, on en déduit qu'il existe une sous-suite $(\psi_j)_{j\in \mathbb N}$ telle que
\begin{description} 
\item[$*$] $\psi_{j}\rightharpoonup\psi$ faiblement dans $H_1(M)$
\item[$*$] $\psi_{j}\rightarrow\psi$ fortement dans $L^s(M)$ pour tout $1\leq s<N$ 
\item[$*$]  $\psi_{j}\rightarrow\psi$ presque partout.
\end{description}
En particulier la suite $(\psi_j)$ converge fortement dans $L^{q/(q-1)}(M)$ et $L^{2p/(p-1)}(M)$ car $q/(q-1)<N$ et $2p/(p-1)<N$. Par conséquent 
$$\int_Mf\psi\di v=1\text{ et } \int_Mh\psi_j^2\di v\rightarrow \int_Mh\psi^2\di v$$ 
La convergence faible dans  $H_1(M)$ et forte dans $L^2(M)$ entraînent que 
$$\lim_{j\to +\infty}\|\nabla\psi_j\|_2\geq \|\nabla\psi\|_2$$
On en conclut que $E(\psi)\leq \mu$ et donc nécessairement que $E(\psi)=\mu$. En écrivant l'équation d'Euler--Lagrange pour $\psi$, on trouve qu'elle est solution faible dans $H_1(M)$ de l'équation \eqref{equation}. Par la proposition \ref{delta u f}, on déduit que $\psi\in H^{\min(p,q)}_2(M)$.
\end{proof}

\chapter{\'Etude d'équations de type Yamabe}\label{c}

\section{Existence de solutions sans présence de symétries}\label{section sans sym}

Soit $(M,g)$ une variété riemannienne compacte $C^\infty$ de dimension $n\geq 3$. On considère l'équation suivante :
\begin{equation}\label{AF}
\Delta_g \psi+ h\psi= \tilde h \psi^{\frac{n+2}{n-2}}
\end{equation}
Où $\psi\in H_1(M)$, $h\in L^p(M)$ avec $p>n/2$ et $\tilde h$  une constante. Dorénavant, ce type d'équation s'appellera équation de type Yamabe. Dans le cas particulier $h=\frac{n-2}{4(n-1)}R_g$, l'équation \eqref{AF} est celle de Yamabe qu'on verra plus en détail dans la section \ref{yampro}. Ce type d'équation a été déjà considéré par Z.~Faget \cite{Fagt}, lorsque $h$ est continue sur $M$ et invariante par un sous groupe d'isométries. 

Pour résoudre ce type d'équations, on utilisera la méthode variationnelle, qui consiste à trouver une fonctionnelle à minimiser sur un espace bien choisi. Dans notre cas l'espace est $H_1(M)$. On montrera ensuite que le minimum de cette fonctionnelle est atteint pour une certaine fonction qui sera solution de l'équation d'Euler--Lagrange. On aura l'occasion d'appliquer cette méthode plusieurs fois.\\
On se place dans l'espace $H_1(M)$, on définit l'énergie $E$ de $\psi\in H_1(M)$ par: 
\begin{equation*}
 E(\psi)=\int_M |\nabla\psi|^2+h\psi^2\di v
\end{equation*}
Et on considère la fonctionnelle $I_g$ définie, pour tout $\psi\in H_1(M)-\{0\}$, par
\begin{equation*}
I_g(\psi)=\frac{E(\psi)}{\|\psi\|^2_N}\label{ca} 
\end{equation*}
On note
\begin{equation*}
 \mu(g)=\inf_{\psi\in H_1(M)-\{0\},\psi\geq 0}I_g(\psi)=\inf_{\|\psi\|_N=1,\psi\geq 0}E(\psi)
\end{equation*}

avec $N=\Nn$. On note  $[p]$ la partie entière d'un nombre réel $p$. Dans le cas du problème de Yamabe (i.e. $h=\frac{n-2}{4(n-1)}R_g$), $I_g$ est appelée la fonctionnelle de Yamabe, et $\mu(g)$ l'invariant conforme de Yamabe (voir section \ref{yampro}). L'un des résultats important de ce chapitre est le suivant:

\begin{theorem}\label{cg}
Soit $(M,g)$ une variété riemannienne compacte $C^\infty$ de dimension $n\geq 3$ et $p>n/2$. Si 
$$\mu(g)<K^{-2}(n,2)$$ 
alors l'équation \eqref{AF} admet une solution strictement positive $\varphi\in H^p_2(M)\subset C^{1-[n/p],\beta}(M)$,  qui minimise la fonctionnelle $I_g$ (i.e. $E(\varphi)=\mu(g)=\tilde h$ et $\|\varphi\|_N=1$). où $\beta\in ]0,1[$.
\end{theorem}

Dans la preuve de ce théorème, on aura besoin du lemme suivant dû à H.~Brezis et E.H.~Lieb~\cite{BL}
\begin{lemma}\label{BLL}
Soit $(f_i)_{i\in\mathbb N}$ une suite de fonctions dans un espace mesuré $(\Omega,\Sigma,\mu)$. 
Si $(f_i)_{i\in\mathbb N}$ est uniformément bornée dans $L^p$ avec $0<p<+\infty$ et $f_i\rightarrow f$ p.p, alors
$$\lim_{i\to +\infty}[\|f_i\|_p^p-\|f_i-f\|_p^p]=\|f\|_p^p$$ 
\end{lemma}

\begin{proof}[\textbf{Preuve du théorème \ref{cg}}]
On commence par vérifier  que $\mu(g)$ est fini. En effet, d'après l'inégalité de Hölder, on a 
\begin{equation*} 
E(\psi)\geq -\|h\|_{n/2}\|\psi\|^2_N
\end{equation*} 
on en déduit  que $\mu(g)\geq -\|h\|_{n/2}>-\infty$.\\
Soit $(\varphi_i)_{i\in \mathbb N}$ une suite minimisante:
\begin{equation}\label{AA}
E(\varphi_i)=\mu(g)+o(1),\;\|\varphi_i\|_N=1\mbox{ et }\varphi_i\geq 0 \end{equation}
En utilisant l'inégalité de Hölder encore une fois dans l'équation ci-dessus, on obtient
\begin{gather*}
\|\nabla\varphi_i\|_2^2\leq \|h\|_{n/2}+\mu(g)+o(1)\\
\|\varphi_i\|_2^2\leq (vol(M))^{2/n} 
\end{gather*}
On en déduit que $(\varphi_i)_{i\in \mathbb N}$ est bornée dans $H_1(M)$. Quitte à extraire une sous-suite, on peut 
supposer qu'il existe $\varphi\in H_1(M)$ tel que
\begin{description}
\item[$*$] $\varphi_i\rightharpoonup \varphi$ faiblement dans $H_1(M)$ par le théorème de Banach (cf. section \ref{teb}). 
\item[$*$] $\varphi_i\rightarrow \varphi$ fortement dans $L^s(M)$, pour tout $s\in[1,N[$, par l'inclusion compacte de Kondrakov (cf. théorème \ref{kon}).
\item[$*$] $\varphi_i\rightarrow \varphi$ presque partout.
\end{description}
On en conclut que: 
\begin{equation*}
\int_M |h| |\varphi_i-\varphi|^2\di v\leq \|h\|_p\|\varphi_i-\varphi\|_{2p/(p-1)}^2\rightarrow 0
\text{ fortement car }2p/(p-1)<N
\end{equation*}
On pose $\psi_i=\varphi_i-\varphi$, alors $\psi_i\rightarrow 0$ faiblement dans $H_1(M)$, fortement dans 
$L^q(M)$ pour tout $q<N$.\\ 
On a $\|\nabla\varphi_i\|_2^2=\|\nabla\psi_i\|_2^2+\|\nabla\varphi\|_2^2+2\int_M\nabla\psi_i\cdot\nabla\varphi \di v$. 
On en déduit que
\begin{equation*}
E(\varphi_i)=E(\varphi)+\|\nabla\psi_i\|_2^2+o(1)
\end{equation*}
Puisque $E(\varphi)\geq \mu(g)\|\varphi\|_N^2$ par définition de $\mu(g)$ et $E(\varphi_i)=\mu(g)+o(1)$
 par définition de la suite $(\varphi_i)_{i\in\mathbb N}$, on en déduit que
\begin{equation}\label{AE}
\mu(g)\|\varphi\|_N^2+\|\nabla\psi_i\|_2^2\leq \mu(g)+o(1)
\end{equation}
On applique le lemme \ref{BLL} à la suite $(\varphi_i)_{i\in\mathbb N}$, on trouve
\begin{align}\label{AB}
\|\psi_i\|_N^N &+\|\varphi\|_N^N+o(1)=1\\
\|\psi_i\|_N^2 &+\|\varphi\|_N^2+o(1)\geq 1 \label{AC}
\end{align}
Par le théorème \ref{INM}
\begin{equation*}
\|\psi_i\|_N^2\leq (K^2(n,2)+\varepsilon)\|\nabla\psi_i\|_2^2+o(1)
\end{equation*}
l'inégalité \eqref{AC} devient donc
\begin{equation*}
(K^2(n,2)+\varepsilon)\|\nabla\psi_i\|_2^2+\|\varphi\|_N^2+o(1)\geq 1
\end{equation*}
Si on utilise cette dernière inégalité dans \eqref{AE}, on trouve
\begin{equation*}
\mu(g)\|\varphi\|_N^2+\|\nabla\psi_i\|_2^2\leq \mu(g)[(K^2(n,2)+\varepsilon)\|\nabla\psi_i\|_2^2+\|\varphi\|_N^2]+o(1)
\end{equation*}
Finalement
\begin{equation*}
[1-\mu(g)(K^2(n,2)+\varepsilon)]\|\nabla\psi_i\|_2^2\leq o(1)
\end{equation*}
Si $\mu(g)<K^{-2}(n,2)$, on peut choisir $\varepsilon$ de sorte que le premier facteur de cette inégalité 
soit strictement positif. 
On en déduit que $(\psi_i)_{i\in\mathbb N}$ converge fortement vers 0 dans 
$H_1(M)$, $\varphi_i\rightarrow \varphi$ fortement dans $H_1(M)$ et $L^N(M)$ d'où $I_g(\varphi)=\mu(g)$.\\
On vient de mettre en évidence une solution non triviale de l'équation de type Yamabe
\begin{equation*}
\Delta \psi+ h\psi= \mu(g) \psi^{N-1}
\end{equation*}
qui satisfait $\|\varphi\|_N=1$ et $\varphi\geq 0$. Par le théorème \ref{regYS}, $\varphi\in H^p_2(M)\subset C^{1-[n/p],\beta}(M)$ et $\varphi>0$.
\end{proof}

\begin{proposition}\label{inegalite large}
 Soit $(M,g)$ une variété riemannienne compacte $C^\infty$. On a toujours:
$$\mu(g)\leq K^{-2}(n,2)$$
\end{proposition}

\begin{proof}
Soient $P$ un point fixé de $M$ et $u_\e$ une fonction radiale définie sur $M$ par
 \begin{gather*}
u_{\e}(Q)=\begin{cases}\biggl(\displaystyle\frac{\varepsilon}{r^2+\varepsilon^2}\biggr)^{\frac{n-2}{2}}-\biggl(\frac{\varepsilon}{\delta^2+\varepsilon^2}\biggr)^{\frac{n-2}{2}} &\mbox{ if }Q\in B_{P}(\delta)\\
0 &\mbox{ if }Q\in M-B_{P}(\delta)
\end{cases}
\end{gather*}
où $r=d(P,Q)$ et $B_P(\delta)$ est la boule géodésique de centre $P$ et de rayon $\delta$. Montrons que $\lim_{\e\to 0}I_g(u_\e)=K^{-2}(n,2)$, ce qui entraînera l'inégalité de la proposition car $\mu(g)$ est bien le minimum de $I_g$.\\ 
Puisque $u_\e$ est radiale 
$$\nabla u_\e=\partial_ru_\e=-(n-2)\e^{(n-2)/2}\frac{r}{(r^2+\e^2)^{n/2}}$$
En intégrant le carré de ce gradient sur $M$, on obtient:
$$\int_M|\nabla u_\e|^2\di v=(n-2)^2\omega_{n-1}\e^{n-2}\int_0^\delta\frac{r^{n+1}}{(r^2+\e^2)^n}\di r$$
En effectuant le changement de variable $t=r/\e$ on trouve
\begin{equation}\label{nabla u}
\int_M|\nabla u_\e|^2\di v=(n-2)^2\omega_{n-1}\int_0^{\delta/\e}\frac{t^{n+1}}{(t^2+1)^n}\di t
\end{equation}
D'autre part $h\in L^p(M)$ avec $p>n/2$ donc 
$$\int_Mhu_\e^2\di v\leq \|h\|_p\|u_\e\|^2_{2p/(p-1)}$$
Par le même changement de variable $t=r/\e$, on a 
$$\|u_\e\|^{2p/(p-1)}_{2p/(p-1)}\leq\int_0^\delta\biggl(\frac{\e}{r^2+\e^2}\biggr)^{\frac{p(n-2)}{p-1}}r^{n-1}\di r\leq\e^{\frac{2p-n}{p-1}}\int_0^{\delta/\e}\biggl(\frac{1}{t^2+1}\biggr)^{\frac{p(n-2)}{p-1}}t^{n-1}\di t$$
donc $\|u_\e\|^2_{2p/(p-1)}=O(\e^{2-\frac{n}{p}})$. Puisque $p>n/2$, on en déduit que 
\begin{equation}\label{hu2}
 \lim_{\e\to 0}\int_Mhu_\e^2\di v=0
\end{equation}
Il nous reste à calculer $\|u_\e\|_N^{-2}$. Lorsque on prend l'intégrale des puissances de $u_\e$ on peut négliger le terme constant dans l'expression de $u_\e$ (des détails sur les puissances de $u_\e$ sont donnés dans l'appendice \ref{nablafi}, équation \eqref{hhh}). D'où
\begin{equation}\label{uN}
\|u_\e\|^N_N=\omega_{n-1}\int_0^{\delta/\e}\frac{t^{n-1}}{(t^2+1)^n}\di t+O(\e^{n-2})
\end{equation}

Il est bien connu que la fonction
 $$v_\e: x\longmapsto \biggl(\frac{\varepsilon}{|x|^2+\varepsilon^2}\biggr)^{\frac{n-2}{2}}$$ 
est solution de l'équation $\Delta_{\mathcal E} u=n(n-2)u^{N-1}$ sur $\mathbb R^n$, où $\Delta_{\mathcal E}$ est le Laplacien euclidien sur $\mathbb R^n$. C'est aussi la fonction qui réalise la meilleure constante de l'inégalité du théorème \ref{INM} (page \pageref{INM}) sur $\mathbb R^n$. On a donc $K^2(n,2)\|\nabla v_\e\|^2_2=\|v_\e\|_N^2$. Autrement dit, si on calcule $\|\nabla v_\e\|^2_2$ et $\| v_\e\|^2_N$, en passant aux coordonnées polaires, on trouve:
\begin{equation}\label{inte}
\biggl((n-2)^2\omega_{n-1}\int_0^{+\infty}\frac{t^{n+1}}{(t^2+1)^n}\di t\biggr)\biggl(\omega_{n-1}\int_0^{+\infty}\frac{t^{n-1}}{(t^2+1)^n}\di t\biggr)^{-\frac{n-2}{n}}=K^{-2}(n,2) 
\end{equation}
En combinant \eqref{nabla u}, \eqref{hu2}, \eqref{uN} et \eqref{inte} on conclut que
$$\lim_{\e\to 0}I_g(u_\e)=\lim_{\e\to 0}(\int_M|\nabla u_\e|^2\di v+\int_Mhu_\e^2\di v)\|u_\e\|^{-2/N}_N=K^{-2}(n,2)$$
Ce qui entraîne que $\mu(M,g)\leq \lim_{\e\to 0}I_g(u_\e)=K^{-2}(n,2)$.
\end{proof}

\subsection{Application}

On considère l'équation suivante:
\begin{equation}\label{yamsi}
\Delta \psi+ \frac{R}{\rho^\alpha}\psi= \tilde R \psi^{\frac{n+2}{n-2}}
\end{equation}
où $R\in C^0(M)$, $\alpha,\;\tilde R$ sont deux nombres réels et $\rho$ la fonction distance (cf définition~\ref{distance}). On pose 
\begin{gather*}
E_\alpha(\varphi)=\int_M |\nabla\varphi|^2+\frac{R}{\rho^\alpha}\varphi^2\di v\\
I_{g,\alpha}(\varphi)=\frac{E_\alpha(\varphi)}{\|\varphi\|^2_N}\label{ca}\\
\mu_\alpha(g)=\inf_{\varphi\in H_1(M)-\{0\},\varphi\geq 0}I_{g,\alpha}(\varphi)=\inf_{\|\varphi\|_N=1,\varphi\geq 0}E_\alpha(\varphi)
\end{gather*}
\begin{proposition}\label{corollaire}
Si $0<\alpha<2$ et $\mu_\alpha(g)<K^{-2}(n,2)$ alors l'équation \eqref{yamsi} 
admet une solution $\varphi_\alpha\in C^{1-[\alpha],\beta}(M)$ strictement positive qui satisfait 
$E_\alpha(\varphi_\alpha)=\mu_\alpha(g)=\tilde R$ et $\|\varphi_\alpha\|_N=1$.
\end{proposition}
\begin{proof}
Si on pose $h:=R/\rho^\alpha\in L^p(M)$ avec 
$2>n/p>\alpha$, alors cette proposition est un corollaire immédiat du théorème \ref{cg} 
\end{proof}

\subsection*{Le cas critique $\boldsymbol{\alpha=2}$}

Ce cas correspond à l'équation non linéaire de Schrödinger  avec le potentiel de Hardy et l'exposant critique. Il a été déjà étudié sur $\mathbb R^n$ par S.~Terracini \cite{Ter} et D.~Smets \cite{Sme} qui ont montré l'existence 
et non existence de solutions de l'équation ci-dessous pour $\alpha=2$ et $\rho=|x|$ sous certaines conditions. Le théorème obtenu ici est le suivant:
\begin{theorem}
Si  $\mu_2(g)<[1+\min(R(P),0)K^2(n,2,-2)]K^{-2}(n,2)$ et\\ $1+R(P)K^2(n,2,-2)>0$ alors il existe $\varphi_2\in H_1(M)$ 
solution non triviale de l'équation \eqref{yamsi} pour $\alpha=2$.
\end{theorem}

\begin{proof}
$(a).$ On montre que $\mu_2(g)$ est fini et $\lim_{\alpha\rightarrow 2^-}\mu_\alpha(g)=\mu_2(g)$. Pour tout $\e>0$ il 
existe $\delta>0$ tel que si $Q\in B_\delta(P)$ alors $|R(Q)-R(P)|<\e$, 
de plus si $\psi\in H_1(M)$ et $\|\psi\|_N=1$ alors
\begin{equation*} 
E_2(\psi)\geq \|\nabla\psi\|_2^2-\frac{\|R\|_\infty}{\delta^2}\|\psi\|_2^2+(R(P)-\e)\int_{B_\delta(P)}\rho^{-2}\psi^2\di v 
\end{equation*}
Par le lemme \ref{aa} et l'inégalité de Hölder:
\begin{equation*}
E_2(\psi)\geq [1+(\min(R(P),0)-\e)K_\delta^2(n,2,-2)]\|\nabla\psi\|_2^2-\|R\|_\infty\delta^{-2} vol(M)^{2/n}
\end{equation*}
Si $1+R(P)K^2(n,2,-2)>0$ alors il existe $\e$ et $\delta$ tels que 
$$E_2(\psi)>-\|R\|_\infty\delta^{-2} vol(M)^{2/n}$$
Le théorème de la convergence dominée de Lebesgue, nous permet d'écrire que pour tout $\psi\in H_1(M)-\{0\}$: $\lim_{\alpha\rightarrow 2^-}I_{g,\alpha}(\psi)=I_{g,2}(\psi)$. On en déduit que $\lim_{\alpha\rightarrow 2^-}\mu_{\alpha}(g)=\mu_{2}(g)$. Il existe alors $\alpha_0$  tel que pour tout $\alpha\in [\alpha_0,2]$: $\mu_\alpha(g)<K^{-2}(n,2)$\\
$(b).$ On montre que la famille $\{\varphi_\alpha\}_{\alpha\in [\alpha_0,2[}$ est uniformément bornée dans $H_1(M)$. 
Cette famille satisfait les résultats de la proposition \ref{corollaire} donc pour tout $\alpha\in [\alpha_0,2[$
\begin{equation*}
 \|\varphi_\alpha\|_2\leq vol(M)^{1/n}
\mbox{ et }\|\nabla\fia\|_2^2+\int_{B_\delta(P)}\frac{R}{\rho^\alpha}\fia^2\di v\leq 
K^{-2}(n,2)+\delta^{-2}\|R\|_\infty\|\fia\|_2^2
\end{equation*}
Mais
$$\int_{B_\delta(P)}\frac{R}{\rho^\alpha}\fia^2\di v\geq (\min(R(P),0)-\varepsilon)K_\delta^2(n,2,-2)\|\nabla\fia\|_2^2$$ 
d'où $$[1+(\min(R(P),0)-\varepsilon)K_\delta^2(n,2,-2)]\|\nabla\fia\|_2^2\leq K^{-2}(n,2)+\delta^{-2}\|R\|_\infty vol(M)^{2/n}$$
Compte tenu de l'hypothèse sur $R(P)$, on peut choisir $\e$ suffisamment petit pour que le premier facteur de cette inégalité soit strictement positif.\\
$(c).$ Il existe une suite $(\alpha_i)_{i\in\mathbb N}$ à valeur dans $[\alpha_0,2[$ qui converge vers 2, telle que la suite de fonctions $(\fii)_{i\in\mathbb N}$ converge faiblement dans $H_1(M)$, $L^2(M,\rho^{-2})$, $L^N(M)$ et fortement dans $L^q(M)$ vers une fonction $\varphi_2\geq 0$, avec $q<N$ (voir la section \ref{hardy section} pour la définition de $L^2(M,\rho^{\gamma})$ et le théorème \ref{AD}).\\
Pour tout $\psi\in H_1(M)$
\begin{equation*}
\int_M\nabla\fii\nabla\psi \di v+\int_M\frac{R}{\rho^{\alpha_i}}\fii\psi \di v = \mu_{\alpha_i}(g)\int_M\fii^{N-1}\psi \di v
\end{equation*}
On veut passer à la limite dans cette égalité. C'est immédiat pour la première intégrale, d'après la convergence faible 
dans $H_1(M)$. Pour la seconde intégrale:
\begin{equation*}
 \biggl|\int_M\frac{R}{\rho^{\alpha_i}}\fii\psi-\frac{R}{\rho^{2}}\varphi_2\psi \di v\biggr|\leq 
\biggl|\int_M\frac{R\psi}{\rho^2} (\fii-\varphi_2)\di v\biggr|+\int_M |R\psi\fii| |\frac{1}{\rho^{\alpha_i}}
-\frac{1}{\rho^2}|\di v
\end{equation*}
La convergence faible dans $L^2(M,\rho^{-2})$ et le théorème de la convergence dominée de Lebesgue impliquent que le 
second membre converge vers 0.\\
Comme $(\fii)_{i\in\mathbb N}$ est uniformément bornée dans $L^N(M)$, $(\fii^{N-1})_{i\in\mathbb N}$ est uniformément bornée dans $L^{N/(N-1)}$. Alors 
$$\mu_{\alpha_i}(g)\int_M\fii^{N-1}\psi \di v\rightarrow \mu_2(g)\int_M\varphi_2^{N-1}\psi \di v$$
On en conclut que $\varphi_2$ est une solution faible de l'équation \eqref{yamsi} pour $\alpha=2$. Il nous reste à montrer que $\varphi_2$ 
n'est pas identiquement nulle. Le théorème \ref{INM} montre que 
\begin{equation}\label{dc}
1=\|\fii\|_N^2\leq (K^2(n,2)+\e)(\mu_{\alpha_i}(g)-\int_M\frac{R}{\rho^{\alpha_i}}\fii^2 \di v)+A\|\fii\|_2^2
\end{equation}
Ce même théorème implique encore une fois 
\begin{align*}
\int_M\frac{R}{\rho^{\alpha_i}}\fii^2 \di v &=\int_{B_\delta(P)}\frac{R}{\rho^{\alpha_i}}\fii^2 \di v+
\int_{M-B_\delta(P)}\frac{R}{\rho^{\alpha_i}}\fii^2 \di v\\
&\hspace{-1.3cm}\geq (\min(R(P),0)-\e)(K^{2}(n,2,-2)+\e')(\mu_{\alpha_i}(g)-\int_M\frac{R}{\rho^{\alpha_i}}\fii^2 \di v)-A\|\fii\|_2^2
\end{align*}
d'où
\begin{equation}\label{dd}
\int_M\frac{R}{\rho^{\alpha_i}}\fii^2 \di v\geq 
\frac{(\min(R(P),0)-\e)(K^{2}(n,2,-2)+\e')}{[1+(\min(R(P),0)-\e)(K^{2}(n,2,-2)+\e')]}\mu_{\alpha_i}(g)-A'\|\fii\|_2^2
\end{equation}
Le dénominateur ci-dessus est strictement positif, si $\e$ et $\e'$ sont suffisamment petit. Les constantes $A$ et $A'$ ne 
dépendent pas de $\alpha_i$. Des inégalités \eqref{dc} et \eqref{dd}, on tire 
\begin{equation*} 
A''\|\fii\|_2^2\geq \frac{1+(\min(R(P),0)-\e)(K^{2}(n,2,-2)+\e')-(K^2(n,2)+\e)\mu_{\alpha_i}(g)}
{1+(\min(R(P),0)-\e)(K^{2}(n,2,-2)+\e')}
\end{equation*}
Le second membre de cette expression reste strictement positif lorsque $i\to +\infty$, alors il existe $c>0$ 
tel que $\|\varphi_2\|_2^2>c$
\end{proof}

\section{Existence de solutions en présence de symétries}\label{group isome sect}

\subsection{Le groupe d'isométries et le groupe conforme}\label{iso conf}

\begin{definition}
 Soit $(M,g)$ une variété riemannienne $C^\infty$. le groupe d'isométries $I(M,g)$ et le groupe conforme $C(M,g)$ de $(M,g)$  sont définis par
\begin{gather*}
I(M,g)=\{f\in C^\infty(M,M)/ f^*g=g\}\\
C(M,g)=\{f\in C^\infty(M,M)/ f^*g=e^h g,\; h\in C^\infty(M)\}
\end{gather*}
\end{definition}

\begin{definition}\label{def gG}
Soit $G$ un sous groupe du groupe $I(M,g)$.
\begin{enumerate}
 \item On dit qu'une fonction  $f$ dans $H^q_k(M)$ est $G-$invariante  si et seulement si pour tout $\sigma \in G$, $\sigma^*f=f$ presque partout, où $k\in\mathbb N$ et $q\geq 1$. L'ensemble de ces fonctions est noté $H^q_{k,G}(M)$ si $k\geq 1$, $L_G^q(M)$ si $k=0$, et $H_{k,G}(M)$ si $q=2$.
\item Une métrique $g'$ est dite $G-$invariante si et seulement si $G\subset I(M,g')$ 
\item $[g]^G$ est la classe des métriques $G-$invariantes conforment à $g$ définie par:
$$[g]^G=\{\tilde g=e^fg/ f\in C^\infty(M),\; G\subset I(M,\gt)\}$$
\end{enumerate}
\end{definition}

Résoudre l'équation de type Yamabe \eqref{AF} en présence de symétries revient à chercher une solution $G-$invariante, strictement positive de l'équation \eqref{AF}, où $h$ est fonction $G-$invariante presque partout. E.~Hebey et  M.~Vaugon \cite{HV} ont introduit cette équation lorsque $h$ est proportionnelle à la courbure scalaire $R_g$, qui est évidemment $G-$invariante. Dans ce cas le problème a une signification géométrique que l'on précisera dans le chapitre~\ref{CHYS}. Afin de trouver des solutions à ce problème, E.~Hebey et  M.~Vaugon ont utilisé la technique des points de concentration, sans utiliser l'analogue de l'inégalité de la meilleure constante pour l'espace $H_{1,G}(M)$. Cette inégalité s'avérera fondamentale pour trouver la condition suffisante dans la résolution de l'équation de type Yamabe sans présence de symétries \eqref{AF} (cf. théorème \ref{cg}), elle a été obtenue par E.~Hebey et  M.~Vaugon~\cite{HV2}, après leurs travaux sur le problème de Yamabe équivariant, lorsqu'ils ont étudié les inclusions de Sobolev pour les espaces $G-$invariants. Ils ont obtenu les résultats suivants:

\subsection{Inégalité de la meilleure constante en présence de symétries }\label{INHV}

\begin{theorem}[Hebey--Vaugon]\label{INMHV}
 Soit $(M, g)$ une variété  riemannienne compacte de dimension $n$, $G$ un sous groupe compact du groupe $I(M,g)$. Soit $k$ la plus petite dimension des orbites de $M$ sous $G$. On pose $p^*=\frac{(n-k)p}{n-k-p}$ si $n-k-p\neq 0$. 
\begin{enumerate}
\item Si $p$ est un réel tel que $1 \leq p < n-k$ alors pour tout $q\in [1,p^*]$, l'inclusion $H_{1,G}^p(M)\subset L_G^{q}(M)$ est  continue. De plus si $q\in [1,p^*[$  elle est compacte.
\item Si $p\geq n-k$ alors pour tout $q\geq 1$, l'inclusion $H_{1,G}^p(M)\subset L_G^{q}(M)$ est  continue et compacte 
\end{enumerate}
\end{theorem}
(T.~Parker \cite{Par} avait aussi travaillé sur les inclusions de Sobolev pour les espaces $G-$invariants).\\
On note par $O_G(P)$ l'orbite du point $P$ sous l'action de $G$. La meilleure constante dans ces inclusions a été calculée par Z.~Faget \cite{Fag}.
\begin{theorem}[Z.~Faget]\label{embHV}
Sous les hypothèses du théorème précédent, si on pose 
$$A=\min\{vol(O_G(Q))/ Q\in M\text{ et }\dim O_G(Q)=k\}$$
 (si $G$ a des orbites finies alors $k = 0$ et $A = \min_{Q\in M} card O_G(Q) $) et $1 \leq p < n-k$ alors pour tout  $\e>0$, il existe $B(\e)$ tel que  
\begin{equation*}
\forall \varphi\in H_{1,G}^p(M)\quad \|\varphi\|^p_{p^*}\leq (\frac{K^p(n-k,p)}{A^{p/(n-k)}}+\varepsilon)\|\nabla\varphi\|^p_p+B(\varepsilon)\|\varphi\|^p_p
\end{equation*}
$K(n-k,p)A^{-1/(n-k)}$ est la meilleure constante.
\end{theorem}

Soit $h$ une fonction dans $L_G^p(M)$ avec $p>n/2$ et $q\in[2,\Nn]$. On considère l'équation de type Yamabe (avec un exposant $q$) suivante:
\begin{equation}\label{eygi}
\Delta_g \psi+ h\psi= \tilde h \psi^{q-1}
\end{equation}
où $\tilde h$ est une constante. Le but de cette section est de chercher des solutions $\psi>0$ et $G-$invariante dans $H^p_{2,G}$. On attachera plus d'attention au cas $q=N=\Nn$. Posons pour tout $\varphi\in H_{1,G}(M)$.
\begin{equation*}
 I_{q,g}(\varphi)=\frac{E(\varphi)}{\|\varphi\|^2_q},\qquad \mu_{q,G}(g)=\inf_{\varphi\in H_{1,G}(M)-\{0\}} I_{q,g}(\varphi)
\end{equation*}
où $E(\varphi)$ a été défini au début de la section \ref{section sans sym}.\\
Notons que si $q=N$, l'équation \eqref{eygi} et la fonctionnelle $I_{q,g}$ s'identifient à l'équation \eqref{AF} et à la fonctionnelle $I_g$ respectivement. Par contre $\mu(g)\leq \mu_{N,G}(g)$ car $\mu_{N,G}(g)$ est obtenu en prenant des fonctions tests  dans $H_{1,G}(M)\subset H^2_{1}(M)$ (voir la section \ref{section sans sym}, pour les définitions de $I_g$ et $\mu(g)$). 
\begin{proposition}\label{eysc}
Si $q\in [\frac{2p}{p-1},\Nn[$ et $\mu_{q,G}(g)>0$ alors  l'équation \eqref{eygi} admet une solution $\varphi_q\in H_{2,G}^p(M)$, $G-$invariante, strictement positive et qui minimise $I_{q,g}$,   pour $\tilde h=\mu_{q,G}(g)$.
\end{proposition}

\begin{proof}
\begin{itemize}
\item[$*$] Soit $(\varphi_i)_{i\in\mathbb{N}}$ une suite minimisante dans $H_{1,G}(M)$ 
telle que $\|\varphi_i\|_q=1$ et $\varphi_i\geq 0$ alors $(\varphi_i)_{i\in\mathbb{N}}$ est bornée dans 
$H_{1,G}(M)$, en effet $(E(\varphi_i))_{i\in\mathbb{N}}$ est une suite convergente dans $\mathbb{R}$,  on peut donc supposer qu'elle est majorée par $\mu_{q,G}(g)+1$, d'où 
\begin{equation*} 
\begin{split}
\|\varphi_i\|_2^2 &\leq vol(M)^{(q-2)/q}\|\varphi_i\|_q^2\leq vol(M)^{(q-2)/q} \\
 \|\nabla\varphi_i\|_2^2 &=E(\varphi_i)-\int_M h\varphi_i^2\di v\\
 & \leq \mu_{q,G}(g)+1+C\|h\|_p 
\end{split} 
\end{equation*}
\item[$*$] Le théorème de Banach (voir section \ref{teb}) assure l'existence d'une sous-suite $(\varphi_j)_{j\in\mathbb{N}}$ de 
$(\varphi_i)_{i\in\mathbb{N}}$, qui converge faiblement dans $H_{1,G}(M)$ vers une 
fonction $\varphi_q$, et que $$\displaystyle\liminf_{j\to+\infty}\|\nabla\varphi_j\|_{2}+\|\varphi_j\|_2\geq\|\nabla\varphi_q\|_{2}+\|\varphi_q\|_2$$
\item[$*$] Il existe une sous-suite $(\varphi_k)_{k\in\mathbb{N}}$ de $(\varphi_j)_{j\in\mathbb{N}}$, qui converge fortement dans $L^q(M)$ vers la fonction
 $\varphi_q$ si  $q\in [\frac{2p}{p-1},\Nn[$. Il en résulte que $\|\varphi_q\|_q=1$\\
\end{itemize} 
 il en résulte aussi que $$\mu_{q,G}(g)=\lim_{k\to +\infty}I_{q,g}(\varphi_k)\geq I_{q,g}(\varphi_q)$$ 
 on en déduit que $I_{q,g}(\varphi_q)=\mu_{q,G}(g)$, $\varphi_g\geq 0$ et que $\varphi_q$ est $G-$invariante presque partout. Donc $\varphi_q$ minimise la fonctionnelle $I_{q,g}$. On écrit l'équation d'Euler-Lagrange pour la fonction $\varphi_q$, on trouve:
\begin{equation}\label{eyq}
\forall\psi\in H_{1,G}(M)\qquad \int_{M}\nabla_i\varphi_q\nabla^{i}\psi+h\psi\varphi_q-\mu_{q,G}(g)
\psi\varphi_q^{q-1}\di v=0 
\end{equation}
On doit montrer que l'égalité \eqref{eyq} reste vraie pour tout $\psi\in H_1(M)$. C'est là qu'on utilise l'hypothèse $\mu_{q,G}(g)>0$ qui montre que la plus petite valeur propre $\lambda$ de l'opérateur $L:=\Delta_g+h$ est strictement positive. En effet si $\lambda\leq 0$, il existe une fonction propre $\psi\geq 0$ non identiquement nulle telle que 
$$E(\psi)=(L\psi,\psi)_{L^2}=\lambda\|\psi\|_2^2<0$$
D'autre part $E(\psi)\geq \mu_{N,G}(g) \|\psi\|_N^2>0$, ce qui est absurde. Maintenant  la proposition~\ref{delta fu} montre que $L$ est inversible. Comme $\varphi_q\in L^{N/(q-1)}(M)$ et $N/(q-1)>2n/(n+2)$, il existe une unique fonction $\tilde\varphi_q$ solution faible de l'équation
\begin{equation*}\label{argument}
L\tilde\varphi_q=\mu_{q,G}(g)\varphi_q^{q-1}
\end{equation*}
$h$ est $G-$invariante, ainsi que $\Delta_g$, donc $\sigma^*\tilde\varphi_q$ est solution de la même équation pour tout $\sigma \in G$. Par unicité 
$\sigma^*\tilde\varphi_q=\tilde\varphi_q$, $\tilde\varphi_q$ est donc $G-$invariante. D'autre part 
$$\forall\psi\in H_{1,G}(M)\qquad (L(\varphi_q-\tilde\varphi_q),\psi)_{L^2}=0$$
Si on choisit $\psi=\varphi_q-\tilde\varphi_q$ alors   $\varphi_q=\tilde\varphi_q$, car $L$ est coercif, d'après la proposition~\ref{delta fu}. Finalement $\varphi_q$ est une solution faible, non triviale de l'équation
$$\Delta_g\varphi+(h-\mu_{q,G}(g)\varphi_q^{q-2})\varphi=0$$ 
avec $(h-\mu_{q,G}(g)\varphi_q^{q-2})\in L^s(M)$, où $s=\min(p,\frac{2n}{(q-2)(n-2)})>n/2$. Par le théorème \ref{gef}, $\varphi_q$ est bornée, strictement positive, donc $\Delta\varphi_q\in L^p(M)$. Par le théorème de régularité,  $\varphi_q\in H^p_{2,G}(M)$.
\end{proof} 

On s'intéresse maintenant au cas où $q=N$ dans l'équation \eqref{eygi}. On obtient d'abord le résultat suivant:

\begin{proposition}\label{k sup 1}
 Si $k:=\inf_{Q\in M}\dim O_G(Q)\geq 1$ et $\mu_{N,G}(g)>0$ alors l'équation \eqref{eygi} admet une solution $\varphi_N\in H^p_2(M)$, qui minimise $I_{N,g}$, $G-$invariante et strictement positive pour $q=N$ et $\tilde h=\mu_{G}(M,g)$.
\end{proposition}

\begin{proof}
 D'après le théorème \ref{INMHV}, si $k\geq 1$, l'inclusion $H_{1,G}(M)\subset L_G^N(M)$ est compacte. C'est ce qui manquait pour que la preuve de la proposition \ref{eysc} soit valable pour $q=N$.  $\varphi_N$ est donc solution faible dans $H_{1,G}(M)$ de \eqref{eyq}. Pour montrer qu'elle est solution faible pour tout $\psi\in H_1(M)$, il suffit d'utiliser l'argument déjà utilisé à la fin de la preuve de la proposition \ref{eysc}, en utilisant le fait que  l'inclusion $H_{1,G}(M)\subset L_G^{2^*}(M)$ est continue, où $2^*=2(n-k)/(n-k-2)$ (cf. théorème \ref{INMHV}). Ceci entraîne qu'il existe $s>2n/(n+2)$ tel que
 $\varphi_N^{N-1}\in L^s(M)$. Le résultat de la proposition \ref{eysc} s'étend donc à $q=N$ lorsque $k\geq 1$.
\end{proof}

\begin{theorem}\label{theginv}
Soit $(M,g)$ une variété riemannienne compacte. $G$ un sous groupe de $I(M,g)$. Si $$0<\mu_{N,G}(g)<K^{-2}(n,2)(\inf_{Q\in M}\Ca O_G(Q))^{2/n}$$
alors pour  $q=N$,  l'équation \eqref{eygi}  admet une solution $\varphi\in H_{2,G}^p(M)\subset C^{1-[n/p],\beta}(M)$ strictement positive, $G-$invariante et minimisante pour la fonctionnelle $I_{N,g}$.
\end{theorem}

\begin{proof}
 On   fait tendre $q$ vers $N$ pour les solutions $\varphi_q$ de l'équation \eqref{eygi}, obtenues grâce à la proposition \eqref{eysc}. En utilisant la proposition \ref{k sup 1}, le problème est résolu  si  $k=\inf_{Q\in M}\dim O_G(Q)\geq 1$.\\
Supposons que $k=\inf_{Q\in M}\dim O_G(Q)=0$. On pose 
$$\Phi=\{\varphi_q\text{ solution de \eqref{eygi} },\;\varphi_q>0, \; \|\varphi_q\|_q=1\text{ et }\mu_{q,G}(g)=I_{q,g}(\varphi_q)/q\in[q_0,N[ \}$$ 
l'ensemble des solutions données par la proposition \ref{eysc}, avec $q_0\in ]2p/(p-1),N[$ suffisamment proche de $N$ de sorte que $\mu_{q,G}(g)$ reste strictement positive pour tout $q\in[q_0,N[$. Ce qui est possible car
$$\forall q\in [q_0,N[\quad \mu_{q,G}(g)=I_{q,g}(\varphi_q)=I_{N,g}(\varphi_q)\|\varphi_q\|_N^{-2}\geq\mu_{N,G}(g)\|\varphi_q\|_N^{-2}>0$$ 
D'autre part, pour tout $\e>0$, il existe $\varphi_\e\in H^p_{2,G}(M)$ strictement positive telle que
$$I_{N,g}(\varphi_\e)< \mu_{N,G}(g)+\e$$
Puisque 
$$\limsup_{q\to N}\mu_{q,G}(g)\leq \lim_{q\to N}I_{q,g}(\varphi_\e)=I_{N,g}(\varphi_\e)$$ 
on en déduit que 
\begin{equation}\label{mu lim}
 \limsup_{q\to N}\mu_{q,G}(g)\leq \mu_{N,G}(g)
\end{equation}

 L'ensemble $\Phi$ est borné dans $H_{1}^2(M)$, en effet:
\begin{equation*} 
\begin{split} 
\|\varphi_q\|_2 & \leq vol(M)^{1/2-1/q}\|\varphi_q\|_q\leq 1+vol(M)^{1/2-1/N}\\
\|\nabla\varphi_q\|^2_2 &=\mu_{q,G}(g)-\int_M h\varphi_q^2\di v\\
& \leq I_{q,g}(1)+\|h\|_p\|\varphi_q\|_{2p/(p-1)}^2\\
&\leq \|h\|_1 vol(M)^{-2/q}+\|h\|_p\|\varphi_q\|_{2p/(p-1)}^2\\
& \leq C\|h\|_p
\end{split} 
\end{equation*}
où $C$ est une constante strictement positive qui dépend seulement de $n$. L'ensemble $\Phi$ est donc faiblement compact dans $H_{1}^2(M)$, on en déduit qu'il existe une suite $(q_i)_{i\in \mathbb{N}}$ qui converge vers $N$ telle que \begin{description}
\item[$*$] $\varphi_{q_i}\rightharpoonup\varphi_N$ faiblement dans $H_1(M)$.
\item[$*$] $\varphi_{q_i}\rightarrow\varphi_N$ fortement dans $L^s(M)$ pour tout $1\leq s<N$.
\item[$*$]  $\varphi_{q_i}\rightarrow\varphi_N$ presque partout.
\end{description}
Donc $\varphi_N$ est nécessairement $G-$invariante presque partout.\\
Puisque $\varphi_{q_i}$ satisfait l'équation \eqref{eygi} pour $\tilde h=\mu_{q_i,G}(g)$ et $q=q_i$, alors pour tout $\psi\in H_1(M)$:
\begin{equation}\label{fai} 
\int_M\nabla^j\psi\nabla_j\varphi_{q_i}\di v+\int_Mh\psi\varphi_{q_i}\di v=\mu_{q_i,G}(g)\int_M\psi\varphi_{q_i}^{q_i-1}\di v
\end{equation}
D'autre part, l'inclusion de Sobolev  $H_1(M)\subset L^N(M)$ et l'inégalité de Hölder permettent d'écrire 
\begin{equation*}
\|\varphi_{q_i}^{q_i-1}\|_{N/(N-1)}\leq vol(M)^{\frac{N-q_i}{N-1}}\|\varphi_{q_i}\|^{q_i-1}_{N}\leq c(\|\nabla\varphi_{q_i}\|_2+\|\varphi_{q_i}\|_2)^{N-1}\leq C
 \end{equation*}
car $\Phi$ est bornée dans $H_1(M)$. Donc, à extraction de sous-suite près, $\varphi_{q_i}^{q_i-1}$ converge faiblement vers $\varphi_N^{N-1}$ dans $L^{N/(N-1)}(M)$ (voir les théorèmes des espaces de Banach dans la section \ref{teb}) et par l'inégalité \eqref{mu lim}, on peut supposer que $\mu_{q_i,G}(g)$ converge vers $\mu$. 
Par conséquent on peut passer à la limite dans \eqref{fai}, on en déduit que $\varphi_N$ est solution faible de l'équation \eqref{eygi} pour $q=N$ et $\tilde h=\mu$. Montrons que $\varphi_N$ n'est pas identiquement nulle. Puisque $\varphi_q$ est $G-$invariante presque partout, on  peut  appliquer l'inégalité de la meilleure constante en présence de symétrie du théorème \ref{embHV} :
\begin{equation*}  
\forall\varepsilon>0\quad\|\varphi_{q_i}\|^2_N\leq(K^2(n,2)[\inf_{Q\in M}\Ca O_G(Q)]^{-2/n}+\varepsilon)\|\nabla\varphi_{q_i}\|^2_2+B(\varepsilon)\|\varphi_{q_i}\|^2_2
\end{equation*}
$\varphi_{q_i}\in \Phi$ et en utilisant l'inégalité de Hölder:
 $$\|\varphi_{q_i}\|^2_N\geq vol(M)^{2/N-2/q_i}\|\varphi_{q_i}\|^2_{q_i}=vol(M)^{2/N-2/q_i}$$
 on peut donc écrire que 
\begin{equation*} 
vol(M)^{2/N-2/q_i}\leq(K^2(n,2)[\inf_{Q\in M}\Ca O_G(Q)]^{-2/n}+\e)(\mu_{q_i,G}(g)-\int_M h\varphi_{q_i}^2\di v)+B(\varepsilon)\|\varphi_{q_i}\|^2_2
\end{equation*}
Quand $i\rightarrow +\infty$, $\mu_{q_i,G}(g)\rightarrow \mu$ et $vol(M)^{2/N-2/q_i}\rightarrow 1$ donc
\begin{gather*} 
1\leq(K^2(n,2)[\inf_{Q\in M}\Ca O_G(Q)]^{-2/n}+\e)(\mu-\int_M h\varphi_N^2\di v)+B(\varepsilon)\|\varphi_N\|^2_2
\end{gather*}
Comme  $\mu<\mu_{N,G}(g)<K^{-2}(n,2)(\inf_{Q\in M}\Ca O_G(Q))^{2/n}$, on peut même supposer qu'il existe $\e_0>0$ tel que 
$$(K^2(n,2)[\inf_{Q\in M}\Ca O_G(Q)]^{-2/n}+\e_0)\mu< 1-\e_0 $$
cela entraîne l'existence d'une constante $C(\e_0)>0$ telle que
$$ B(\e_0)\|\varphi_N\|^2_2+C(\e_0)\|h\|_p\|\varphi_N\|^2_{\frac{2p}{p-1}}\geq\e_0$$
alors $\varphi_N$ n'est pas identiquement nulle. On vient donc de montrer que $\varphi_N$ est une solution faible positive, non identiquement nulle et $G-$invariante presque partout de l'équation 
\begin{equation}\label{eygih}
\Delta_g \varphi_N+ h\varphi_N= \mu \varphi_N^{N-1}
\end{equation}
Par le théorème \ref{regYS}, $\varphi_N\in H^p_{2,G}(M)$ est strictement positive. Il reste à montrer que $\varphi_N$ est minimisante pour la fonctionnelle $I_{N,g}=I_g$ et que $\mu=\mu_N(g)$. On revient pour celà à la suite $(\varphi_{q_i})$ qui converge fortement vers $\varphi_N$ dans $L^s$ pour tout $1\leq s<N$. En utilisant l'inégalité de Hölder et le fait que $\|\varphi_{q_i}\|_{q_i}=1$, on a l'inégalité suivante:
\begin{equation*}
 \int_M\varphi_{q_i}^{N-1}\varphi_N\di v\leq \|\varphi_N\|_{q_i/(q_i-N+1)}
\end{equation*}
 En passant à la limite dans cette inégalité et grâce au fait que $\varphi_{q_i}\rightarrow \varphi_N$ fortement dans $L^{N-1}$ et que $\varphi_N$ est continue sur $M$ (i.e. $\varphi_N\in H^p_2(M)$), on en déduit que $\|\varphi_N\|_N\leq 1$. D'autre part, si on multiplie l'équation \eqref{eygih} par $\varphi_N$ et on intégre sur $M$, on trouve que 
$$\mu\|\varphi_N\|^{N-2}_N=I_{N,g}(\varphi_N)\geq \mu_{N,G}(g)$$
D'où $\mu\geq \mu_{N,G}(g)$. En combinant avec l'inégalité \eqref{mu lim}, on conclut que $\mu=\mu_{N,G}(g)$ et $\|\varphi_N\|_N=1$.
\end{proof}

\paragraph{Remarque} La méthode que l'on vient d'utiliser dans la preuve de ce théorème n'est  pas  valable  dans le cas où $\mu_{q,G}(g)\leq 0$, car l'opérateur $L=\Delta_g+h$ n'est plus inversible. On verra dans la section \ref{existence g invar} que si la fonction $h$ est proportionnelle à la courbure scalaire $R_g$ de $g$, alors on peut s'en tirer grâce au théorème d'unicité des solutions \ref{unique}.\\
Si on reprend la même démarche utilisée pour montrer le théorème \ref{cg} afin de démontrer le théorème \ref{theginv}, on montre qu'il existe $\varphi_N$  solution faible dans $H_{1,G}(M)$ de l'équation \eqref{eygi}. Plus précisément $\varphi_N$ est solution de l'équation \eqref{eyq}, pour tout $\psi \in H_{1,G}(M)$ et pour $q=N$. Pour  que $\varphi_N$ soit une solution  de l'équation \eqref{eyq}, pour tout $\psi \in H^2_{1}(M)$ et pour $q=N$, il suffit de montrer que l'équation $Lu=\mu_{N,G}(g)\varphi_N^{N-1}$ admet une unique solution faible $u=\tilde\varphi_N\in H_{1,G}(M)$ puis utiliser le même argument que celui de la fin de la preuve de la proposition \ref{eysc} (voir page \pageref{argument}). Malheureusement, on ne peut pas conclure qu'il existe une telle solution $\tilde\varphi_N$, car la proposition \ref{delta fu} assure l'existence d'une telle fonction, si $f\in L^q(M)$ avec $q>2n/(n+2)$, or $\varphi_N^{N-1}\in L^{2n/(n+2)}(M)$.\\
Dans le cas positif (i.e. $\mu(g)>0$), le théorème \ref{cg} est une conséquence du théorème \ref{theginv}, en prenant $G=\{\mathrm{ id }\}$.

\begin{proposition}\label{ine largee}
 Soit $(M,g)$ une variété riemannienne compacte $C^\infty$. $G$ un sous groupe de $I(M,g)$. On a toujours:
$$\mu_G(g)\leq K^{-2}(n,2)(\inf_{Q\in M}\Ca O_G(Q))^{2/n}$$ 
\end{proposition}

\begin{proof}
L'inégalité est triviale si $\inf_{Q\in M}\Ca O_G(Q)=+\infty$. Supposons qu'il existe une orbite minimale finie et soit $P$ un point de cette orbite. Autrement dit 
$$\inf_{Q\in M}\Ca O_G(Q)=\Ca O_G(P)<+\infty$$ 
$O_G(P)=\{P_i\}_{1\leq i\leq k}$, $P=P_1$ et $k=\Ca O_G(P)$. Soit $u_\e$ la fonction définie dans la preuve de la proposition \ref{inegalite large}, que l'on note $u_{\e,P}$ car elle dépend du point $P$ qu'on avait fixé arbitrairement. Soit donc $u_{\e,P_i}$ les fonctions obtenues en remplaçant $P$ par $P_i$ dans l'expression qui définit $u_{\e, P}$. Enfin, on pose
$$U_\e=\sum_{i=1}^k u_{\e,P_i}$$
D'autre part on choisit $\delta$ suffisamment petit tel que pour tout $\sigma\in G-\{\mathrm{id}\}$ 
$$B_P(\delta)\cap B_{\sigma(P)}(\delta)=\emptyset$$ 
Puisque $u_{\e,P_i}$ est radiale (i.e. pour tout $\sigma\in I(M,g)$, $\sigma^*u_{\e,P_i}=u_{\e,\sigma^{-1}(P_i)}$), on en déduit par cette construction que la fonction $U_\e$ est $G-$invariante, à support compact et que  pour tout $1\leq i\leq k$:
$$E(U_{\e})=\sum_{i=1}^kE(u_{\e,P_i})=kE(u_{\e,P})\text{ et }\|U_{\e}\|^N_N=k\|u_{\e,P_i}\|^N_N$$ 
Finalement 
$$I_g(U_\e)=k^{2/n}I_g(u_{\e,P})$$
La proposition \ref{inegalite large} montre que $\lim_{\e\to 0}I_g(U_\e)=k^{2/n}K^{-2}(n,2)$
\end{proof}

\chapter{Le problème de Yamabe avec singularités}\label{CHYS}

Dans ce chapitre on interprétera géométriquement les résultats obtenus dans le chapitre~\ref{c}. On donnera une signification géométrique aux équations de type Yamabe qu'on a déjà résolues. On commence par un rappel historique sur le problème de Yamabe.

\section{Le problème de Yamabe}\label{yampro}

Soit $(M,g)$ une variété riemannienne compacte $C^\infty$ de dimension $n\geq 3$, $R_g$ désigne la courbure scalaire de $g$. Le problème de Yamabe est le suivant:\\

\begin{problem}
Parmi les métriques conformes à $g$, existe-t-il une métrique à courbure scalaire constante? 
\end{problem}

Yamabe \cite{Yam} avait posé ce problème dans le but de résoudre la conjecture de Poincaré. Si on pose $\tilde g=\varphi^{4/(n-2)}g$ une métrique conforme à $g$, où $\varphi>0$ est une fonction $C^\infty$, alors les courbures scalaires $R_g$, $R_{\gt}$ sont reliées par l'équation suivante: 
\begin{equation}\label{yamabe}
\frac{4(n-1)}{n-2}\Delta_g \varphi+ R_g\varphi=  R_{\gt} \varphi^{N-1}
\end{equation}
avec $N=\Nn$.\\ 
Pour résoudre ce problème, il suffit de chercher une fonction $C^\infty$, strictement positive $\varphi$ solution de l'équation aux dérivées partielles non linéaire ci-dessus. L'équation \eqref{yamabe} est appelée l'équation de Yamabe. On utilise la méthode variationnelle pour résoudre cette équation. H.~Yamabe a posé la fonctionnelle suivante, définie pour tout $\psi \in H_1(M)-\{0\}$ par
\begin{equation}\label{fonctionnel}
I_g(\psi)=\frac{E(\psi)}{\|\psi\|_N^2}=\frac{\displaystyle\int_M |\nabla\psi|^2+
\frac{n-2}{4(n-1)}R_g\psi^2\di v}{\|\psi\|_N^2}
\end{equation}
ensuite, il a considéré le minimum de $I_g$ et a défini l'invariant conforme suivant:
$$\mu(g)=\inf_{\psi\in H_1(M)-\{0\}}I_g(\psi)$$
La difficulté majeure dans la recherche des solutions est le fait que l'inclusion de Sobolev $H_1(M)\subset L^q(M)$ est seulement continue pour $q=N$. Par contre cette inclusion est compacte si $1\leq q< N$. Yamabe a donc commencé par résoudre une "sous-équation":

\begin{equation}\label{EY}
\frac{4(n-1)}{n-2}\Delta_g \varphi+ R_g\varphi=\mu_q(g)\varphi^{q-1}
\end{equation}
où $q\in[2,N[,\; N=2n/(n-2)$ et $\mu_q(g)\in\mathbb R$, ensuite a fait tendre $q$ vers $N$. H.~Yamabe a affirmé que l'ensemble $\{\varphi_q>0\text{ solution de } \eqref{EY}, q\in [2,2n/(n-2)[\}$ est uniformément borné dans $C^0(M)$. Or N.~Trudinger \cite{Trud} a montré que c'est seulement vrai lorsque $\mu_q(g)\leq~0$. Finalement, H.~Yamabe a seulement réussi à résoudre le problème dans le cas négatif et nul de $\mu(g)$. Le cas positif est resté ouvert jusqu'à ce que T.~Aubin \cite{Aub} montre qu'il suffit de prouver la conjecture suivante pour  résoudre le problème dans tout les cas. 
\begin{conjecture}[T.~Aubin \cite{Aub}]\label{Aubincon}
Si $(M,g)$ est une variété riemannienne compacte $C^\infty$ de dimension $n$ et non conformément difféomorphe à $(S_n,g_{can})$ alors 
 \begin{equation}\label{cau}
 \mu(M,g)<\mu(S_n,g_{can})
\end{equation}
où $\mu(M,g)=\inf\{I_g(\psi),\; \psi\in H_1(M)-\{0\}\}$
\end{conjecture}
Dans la suite, on écrira $\mu(g)$ en place de $\mu(M,g)$.\\
T.~Aubin a montré que cette inégalité est vraie pour les variétés de dimension $n\geq 6$, non conformément plates et pour les variétés conformément plates de groupe fondamental fini, non trivial. Le cas des variétés conformément plates et des dimensions 3,4 et 5 a été résolu par Schoen \cite{Schoen}, en admettant le théorème de la masse positive. Finalement, la conjecture ci-dessus est toujours vraie. Grâce essentiellement aux travaux de Yamabe \cite{Yam}, T.~Aubin~\cite{Aub} et Schoen~\cite{Schoen}, le problème de Yamabe est complètement résolu dans le cas des variétés riemanniennes compactes $C^\infty$ (voir aussi \cite{Bah},\cite{BB}, \cite{BC}  pour résolution avec une méthode topologique).

\begin{theorem}[Aubin--Schoen]
 Soit $M$ une variétés compacte $C^\infty$, de dimension $n\geq~3$. pour toute métrique riemannienne $g$ de classe $C^\infty$, il existe une métrique conforme $\gt=\varphi^{4/(n-2)}g$ de courbure scalaire constante $R_{\gt}$, où $\varphi$ est une fonction $C^\infty$, strictement positive, qui minimise la fonctionnelle de Yamabe $I_g$.
\end{theorem}

On s'intéresse maintenant au problème de Yamabe avec singularités.

\section{Choix de la métrique}\label{hypG}

Soit $M$ une variété compacte $C^\infty$ de dimension $n\geq 3$ et $g$ une métrique riemannienne sur $M$.\\

\textbf{Hypothèse $\boldsymbol{(H)}$:} \emph{$g$ est une métrique dans l'espace de Sobolev $H_2^p(M,T^*M\otimes T^*M)$ avec $p>n$. Il existe un point $P_0\in M$ et $\delta>0$ tels que $g$ est $C^\infty$ sur la boule $B_{P_0}(\delta)$.}\\

Les métriques que l'on considère sont dans l'espace $H_2^p(M,T^*M\otimes T^*M)$, défini dans la section \ref{sobddd}. On a choisi cet espace de métriques pour donner un sens aux courbures, qui sont donc dans $L^p$. (On peut supposer $g$ de classe $C^2$ dans la boule $B_{P_0}(\delta)$ au lieu de $C^\infty$, mais ce n'est pas un point important).\\ 
En fait, l'objectif de cette partie est surtout d'étudier le problème de Yamabe dans le cas où la métrique $g$ a un nombre fini de points de singularités et est $C^\infty$ en dehors de ces points, l'hypothèse $(H)$ généralise ces conditions et précise la notion de "singularité".\\
Par les inclusions de Sobolev \ref{incsob}, $H_2^p(M,T^*M\otimes T^*M)\subset C^{1,\beta}(M,T^*M\otimes T^*M)$, pour un certain $\beta\in ]0,1[$. Donc les métriques qui satisfont l'hypothèse $(H)$ sont  de classe $C^{1,\beta}$. Les Christoffels sont dans $ C^\beta$ et les courbures de Riemann, Ricci et scalaire sont dans $L^p$ car elles font appel à la dérivée seconde de la métrique $g$ qui est seulement dans $L^p$. Comme exemple de métrique qui satisfait l'hypothèse $(H)$, on peut considérer $g=(1+\rho^{2-\alpha})^m g_0$, où $g_0$ est une métrique $C^\infty$, $\alpha\in]0,1[$ et $\rho$ est définie dans \ref{distance}. Les dérivées secondes de $g$ ont alors des singularités du type $\rho^{-\alpha}$. \\
Dans la suite, beaucoup de résultats seront vrais pour toute métrique dans $H^p_2(M,T^*M\otimes T^*M)$, avec  $p>n/2$ (c'est la valeur minimale de $p$ qui donne un sens à la fonctionnelle de Yamabe. Le cas $p=n/2$ est un cas critique, il est hors de considération). L'hypothèse $(H)$ impose en plus que la métrique est $C^\infty$ dans une certaine boule et que $p>n$. On rajoute la condition $p>n$ pour que les Christoffels de la métrique $g\in H^p_2(M,T^*M\otimes T^*M) $ soient continus. L'hypothèse $(H)$ est suffisante pour montrer la conjecture \ref{Aubincon} (cf. théorème \ref{conj aub}) et pour construire la fonction de Green du Laplacien conforme (cf. section \ref{glc}).\\
On considère le problème suivant:
\begin{problem}\label{yam sing}
 Soit $g$ une métrique qui satisfait l'hypothèse $(H)$. Existe-t-il une métrique $\tilde g$ conforme à $g$ pour laquelle la courbure scalaire $R_{\gt}$ est constante (même aux points où $R_g$ n'est pas régulière)?
\end{problem}

Il est clair que si la métrique initiale $g$ est de classe $C^\infty$, alors le problème ci-dessus n'est autre que le problème de Yamabe \ref{yampro} qui a été  déjà complètement résolu. On montrera plus loin que la réponse à ce problème est positive. La proposition suivante, permet de préciser ce que l'on entend par changement de  métrique conforme lorsque les métriques sont dans $H^p_2$.
\begin{proposition}
Soit $g$ une métrique dans $H^p_2$ et $\psi\in H^p_2(M)$, strictement positive. Si $p>n/2$ alors la métrique $\gt=\psi^{\frac{4}{n-2}}g$ est bien définie, et elle est dans le même espace que $g$.
\end{proposition}

\begin{proof}
 Cette proposition découle du fait que $H^p_2(M)$ est une algèbre, pour tout $p>n/2$ (cf. proposition \ref{algebre}, page \pageref{algebre}).
\end{proof}

\section{Le Laplacien conforme}

\begin{definition}Le Laplacien conforme d'une variété riemannienne $(M,g)$ est l'opérateur $L_g$ défini par :  
$$L_g=\Delta_g+\an R_g$$ 
\end{definition}

\subsection{L'invariance conforme faible }

Il est bien connu que le Laplacien conforme lorsque $g$ est $C^\infty$, est conformément invariant, c'est à dire qu'il vérifie \eqref{laplinvfai} fortement. On montre qu'on a toujours la même propriété lorsque la métrique est dans $H^p_2(M,T^*M\otimes T^*M)$.
\begin{proposition}\label{invcon}
Soient  $M$  une variété compacte $C^{\infty}$ de dimension $n\geq 3$ et $g\in H_2^p(M,T^*M\otimes T^*M)$ est une métrique riemannienne sur $M$, avec $p>n/2$. 
Si $\gt=\psi^{\frac{4}{n-2}}g$ est une métrique conforme à $g$, avec $\psi\in H_2^p(M)$ et $\psi>0$, alors $L$ est faiblement conformément invariant, autrement dit
\begin{equation}\label{laplinvfai}
 \forall u\in H_1(M)\qquad \psi^{\frac{n+2}{n-2}}L_{\gt}(u)=L_g(\psi u)\quad faiblement
\end{equation}
De plus si $\mu(g)>0$ alors le Laplacien conforme $L_g=\Delta_g+\frac{n-2}{4(n-1)}R_g$ est inversible et coercif.
\end{proposition}

\begin{proof}
Rappelons que $\di v_{\gt}=\psi^{\Nn}\di v$ et que
$$\forall u,w\in L^2(M)\quad (u,w)_{g,L^2}=\int_Muw\di v_g$$ est le produit scalaire sur l'espace $L^2(M)$ muni de la métrique $g$.\\ 
Pour tout $u,w\in H_1(M)$: 
\begin{equation*}
\begin{split}
(\psi^{\Nn}L_{\gt} u,w)_{g,L^2} &=(L_{\gt} u,w)_{\gt,L^2}\\
&=\int_M \gt(\nabla u ,\nabla w)+\an R_{\gt} uw \di v_{\gt}\\
                              &=\int_M \psi^2g(\nabla u ,\nabla w)+\an R_{\gt}\psi^{\frac{n+2}{n-2}} (uw\psi)\di v_{g}		      
\end{split}
\end{equation*}
D'autre part, on sait que les deux courbures scalaires $R_g$ et $R_{\gt}$ sont reliées par l'équation de Yamabe \eqref{yamabe}, ce qui est équivalent à 
$$L_g\psi=\an R_{\gt}\psi^{\frac{n+2}{n-2}}\quad faiblement$$
ce que l'on écrit
$$(L_g\psi,uw\psi)_{g,L^2}=\an (R_{\gt}\psi^{\frac{n+2}{n-2}},uw\psi)_{g,L^2}$$
où il y a un abus de notation car $uw\psi$ n'appartient pas forcément à $ L^2(M)$. Par contre $L_g\psi\in L^p(M)\subset L^{n/2}(M)$ et $uw\psi\in L^{n/(n-2)}(M)$, le produit est donc bien défini. Par conséquent
\begin{equation*}
\begin{split}
(\psi^{\Nn}L_{\gt} u,w)_{g,L^2} &=\int_M \psi^2g(\nabla u, \nabla w)+g(\nabla\psi ,\nabla (uw\psi))+\an R_g \psi(uw\psi)\di v_{g}\\
			      & =\int_M g(\nabla (\psi u), \nabla (w\psi))+\an R_g (\psi u)(w\psi)\di v_{g}\\
			      & =(\psi L_g(\psi u),w)_{g,L^2} 
\end{split}
\end{equation*}
On a utilisé le fait que $u\psi$ et $w\psi$ appartiennent à $H_1(M)$, car on a les inclusions 
$$H_2^p(M)\subset C^{1-[n/p],\beta}(M),\; H^p_1(M)\subset L^{\frac{pn}{n-p}}(M)\text{ et } H_1(M)\subset L^{\frac{2n}{n-2}}(M)$$
Maintenant, montrons que $L_g$ est inversible et coercif. Soit $\lambda$ la plus petite valeur propre de $L_g$, de fonction propre $\varphi\in H_1(M)$ positive, non identiquement nulle, alors 
$$\lambda\|\varphi\|_2^2=(L_g\varphi,\varphi)_{g,L^2}=I_g(\varphi)\|\varphi\|_N^2\geq \mu(g)\|\varphi\|_N^2>0$$
d'où $\lambda>0$. Il suffit donc d'appliquer la proposition \ref{delta fu}.

\end{proof}

\section{L'invariant conforme de Yamabe}\label{invariance de mu}

Dans le cas des métriques de classe $C^\infty$, $\mu(g)$ est un invariant conforme, ce qui signifie que si $g$ et $\gt$ sont deux métriques conformes de classe $C^\infty$ alors $$\mu(g)=\mu(\gt)$$ (voir la section \ref{yampro} pour la définition). La proposition suivante montre qu'on peut étendre  cette propriété à des métriques dans $H^p_2$. Elle nous permettra aussi de prendre une métrique quelconque dans la classe conforme $[g]$ comme métrique initiale, tout en gardant la valeur de $\mu(g)$ inchangée. 
\begin{proposition}\label{invconforme}
 Soit $M$ une variété compacte $C^\infty$, de dimension $n$. Soit $g$ et $\gt=\psi^{\frac{4}{n-2}}g$ deux métriques dans $H^p_2$, avec $\psi\in H^p_2(M)$, strictement positive. Si $p>n/2$ alors 
$$\mu(g)=\mu(\gt)$$ 
\end{proposition}

\begin{proof}
 Soient $u\in H_1(M)$ une fonction test et  $I_g$ la fonctionnelle de Yamabe \eqref{fonctionnel}. Remarquons que $E(u)=(L_g(u),u)_{g,L^2}$. Donc 
$$I_{\gt}(u)=(L_{\gt}(u),u)_{\gt,L^2}\|u\psi\|_N^{-2}$$ 
De  la proposition \ref{invcon}, on en déduit que
$$I_{\gt}(u)=(L_g(\psi u), \psi u)_{g,L^2}\|u\psi\|_N^{-2}$$ 
Finalement 
\begin{equation}\label{IgIg}
 I_{\gt}(u)=I_g(\psi u)
\end{equation}
ce qui implique que $\mu(g)=\mu(\gt)$, et que cet invariant dépend seulement de la classe conforme $[g]$ et de la variété $M$.
\end{proof}

\section[Fonction de Green]{La fonction de Green du Laplacien conforme}\label{glc}

\begin{definition}Soit $(M,g)$ une variété riemannienne compacte et $P$ un point de $M$. On appelle fonction de Green au point $P$ d'un opérateur linéaire $L$, la fonction $G_P$ qui vérifie au sens des distributions 
$$LG_P=\delta_{P}   (\Longleftrightarrow \forall f\in C^{\infty}(M)\quad \langle G_P,Lf\rangle =f(P))$$
\end{definition}
La fonction de Green peut être vue comme l'inverse de l'opérateur $L$, lorsque ce dernier est inversible.
La proposition \ref{grgr} montre l'existence d'une telle fonction pour un opérateur du type $L=\Delta+h$ avec $h>0$ continue. Malheureusement, la méthode utilisée pour construire cette fonction de Green n'est pas valable lorsque la fonction $h$ est dans $L^p(M)$. Ce cas se présente pour le Laplacien conforme $L_g$, car $R_g\in L^p(M)$. Mais, grâce à la proposition~\ref{green conf}, on pourra s'en tirer, et obtenir le corollaire \ref{green coroll}. Pour montrer son existence lorsque $h$ est continue, on aura besoin du résultat suivant dû à G.~Giraud \cite{Gir} (On peut aussi consulter \cite{Aubin}, page 108).
\begin{proposition}\label{giraut}
Soit $\Omega$ un ouvert d'une variété riemannienne compacte $(M,g)$. $\varphi$, $\psi$ deux fonctions continues 
sur $\Omega\times\Omega-\{(x,x)\in\Omega\times\Omega\}$ qui vérifient:
$$|\varphi(P,Q)|\leq c(d(P,Q))^{\alpha-n}\mbox{ et }|\psi(P,Q)|\leq c (d(P,Q))^{\beta-n}$$
pour tout $(P,Q)\in\Omega\times\Omega-\{(x,x)\in\Omega\times\Omega\} $, où $\alpha,\;\beta\in ]0,n[$.\\ alors la fonction $\chi$ définie par:
$$\chi(P,Q)=\int_\Omega \varphi(P,R)\psi(R,Q)\di v(R)$$
est continue sur $\Omega\times\Omega-\{(x,x)\in\Omega\times\Omega\}$ et est vérifie:
\begin{equation*} 
|\chi(P,Q)|\leq \begin{cases}c(d(P,Q))^{\alpha+\beta-n} &\mbox{ si }\alpha+\beta<n\\
c(1+\log d(P,Q))& \mbox{ si }\alpha+\beta=n\\
c &\mbox{ si }\alpha+\beta>n
\end{cases}
\end{equation*}
dans le dernier cas la fonction $\chi$ est continue sur $\Omega\times\Omega$.
\end{proposition}

\begin{proposition}\label{grgr}
Soit $M$ une variété compacte $C^\infty$ de dimension $n\geq 3$, $h$ une fonction continue, strictement positive et $P$ un point de $M$. $g$ une métrique qui satisfait l'hypothèse~$(H)$ (cf. section \ref{hypG}). Il existe une unique fonction de Green $G_{P}$ de l'opérateur $L=\Delta_g+h$ qui satisfait au sens des distributions $LG_{P}=\delta_{P}$ 
et
\begin{itemize} 
\item[$(i)$] $G_{P}$ est  $C^\infty$ sur $B_{P_0}(\delta)-\{P\}$
\item[$(ii)$] $G_{P}\in C^2(M-\{P\})$ 
\item[$(iii)$] Il existe $c>0$ tel que pour tout $Q\in M-\{P\}$,  $|G_{P}(Q)|\leq c d(P,Q)^{2-n} $
\end{itemize}
\end{proposition}
\begin{proof}
L'unicité de $G_{P}$ est due au fait que  $L$ est inversible. En effet, si $\lambda$ est une valeur propre de $L$ et $\varphi$ une fonction propre, non identiquement nulle, associée à $\lambda$ alors 
$$\lambda\|\varphi\|^2_2=(L\varphi,\varphi)_{L^2}=E(\varphi)>0$$
D'où $\lambda>0$. Pour conclure, il suffit d'appliquer la proposition \ref{delta fu}. En ce qui concerne l'existence de cette fonction, on reprend la construction de T.~Aubin \cite{Aubin} pour le Laplacien, dans le cas des métriques $C^\infty$. On choisit $f(r)$ une fonction radiale décroissante $C^{\infty}$ positive, égale à $1$ pour $r<\delta/2$ et nulle pour $r\geq\delta(M)$, le rayon d'injectivité de $M$. On définit les fonctions suivantes:

\begin{gather*} 
H(P,Q)=\frac{f(r)}{(n-2)\omega_{n-1}}r^{2-n}\mbox{ avec }r=d(P,Q)\\
\Gamma^1(P,Q)=-L_QH(P,Q) \\
\forall i\in\mathbb{N^*}\qquad\Gamma^{i+1}(P,Q)=\int_M\Gamma^{i}(P,S)\Gamma^1(S,Q)\di v(S)
\end{gather*}
où $L_QH(P,Q)$ signifie qu'on applique l'opérateur $L$ à la fonction $H(P,Q)$ par rapport à $Q$.\\ 
On observe que  $\Gamma^1$ est continue sur $M\times M-\{(Q,Q)\in M\times M\}$, et il existe $c>0$ tel que pour tout $P,Q\in M$:
\begin{equation*}
|\Gamma^1(P,Q)| \leq c d(P,Q)^{2-n}
\end{equation*}
En utilisant la proposition \ref{giraut}, on montre les inégalités suivantes:
\begin{equation*}
\forall i\geq 1\qquad|\Gamma^i(P,Q)| \leq \begin{cases}& cd(P,Q)^{2i-n}  \hspace{2.75cm}\text{ si }2i<n\\
& c(1+\log d(P,Q))  \hspace{2cm}\text{ si } 2i=n\\
& c   \hspace{4.7cm}\text{ si } 2i>n
\end{cases}
\end{equation*}
La fonction de Green de $L$ s'écrit 
\begin{equation}\label{greeen}
G_{P}(Q)=H(P,Q)+\sum_{i=1}^{k}\int_M\Gamma^{i}(P,S)H(S,Q)\di v(S)+F_{P}(Q)
\end{equation} 
où $F_{P}$ est une fonction que l'on détermine dans les lignes qui suivent. On prend $k=[n/2]$ alors $\Gamma^{k+1}(P,\cdot) $ est continue (cf. proposition \ref{giraut}). On veut  $L_QG_P(Q)=0$ pour $Q\neq P$. On a l'identité
$$ \psi(Q)=\Delta_g\int_M H(P,Q)\psi(P)\di v(P)-\int_M \Delta_QH(P,Q)\psi(P)\di v(P)$$
(La preuve est donnée dans \cite{Aubin}, page 106). D'où
$$\psi(Q)=L\int_M H(P,Q)\psi(P)\di v(P)-\int_M L_QH(P,Q)\psi(P)\di v(P)$$
En utilisant cette dernière identité, on trouve que
$$L_QG_P(Q)=-\Gamma^{k+1}(P,Q)+L_QF_P(Q)$$
Puisque $L$ est inversible, il suffit de  poser $F_P$ comme l'unique solution de l'équation $$LF_{P}=\Gamma^{k+1}(P,\cdot)$$ 
Par le théorème de régularité \ref{reg}, $F_P$ est de classe $C^2$. \\
$(i)$ Comme  $L_gG_{P}=0$ sur $B_{P_0}(\delta)-\{P\}$ et que la métrique est $C^\infty$ sur $B_{P_0}(\delta)$, le théorème de régularité affirme que $G_{P}$ est $C^\infty$ sur $B_{P_0}(\delta)-\{P\}$, avec $P\in M$ et $B_{P_0}(\delta)-\{P\}=B_{P_0}(\delta)$ si $P\notin B_{P_0}(\delta)$.\\
$(ii)$ On a aussi $LG_{P}=0$ sur $M-\{P\}$. On conclut par le théorème de régularité que $G_P$ est $C^2$ sur $M-\{P\}$.\\
$(iii)$ En observant l'expression \eqref{greeen} qui définit $G_P$, on remarque que le terme dominant, au voisinage de $P$, est bien $H(P,Q)$, donc pour tout $P\neq Q$, $$|G_P(Q)|\leq cd(P,Q)^{2-n}$$
\end{proof}
 
\begin{proposition}\label{green conf}
 Soit $g$  une métrique dans $H^p_2(M,T^*M\otimes T^*M)$, $\gt=\psi^{\frac{4}{n-2}}g$ une métrique conforme à $g$, avec $\psi\in H^p_2(M)$, strictement positive et $p>n/2$. On suppose que le Laplacien conforme $L_{\gt}$ admet une fonction de Green $\tilde G_{P}$, alors $L_g$ admet aussi  une fonction de Green notée $G_P$ et elle donnée par
$$\forall Q\in M-\{P\}\qquad G_P(Q)=\psi(P)\psi(Q)\tilde G_P(Q)$$ 
\end{proposition}

\begin{proof}

Pour toute fonction $\varphi\in C^{\infty}(M)$:
\begin{equation*}
\begin{split}
\langle \psi(P)\psi\tilde G_P,L_{g}\varphi\rangle _g &=\psi(P)\int_M\tilde G_P\psi L_{g}[\psi(\frac{\varphi}{\psi})]\di v_g\\
							  &=\psi(P)\int_M \tilde G_PL_{\gt}\frac{\varphi}{\psi}
								\di v_{\gt}\\
							  &=\psi(P)\langle\tilde G_P,L_{\gt}\frac{\varphi}{\psi}\rangle  _{\gt}\\
							  &=\varphi(P)
\end{split}
\end{equation*}
La deuxième égalité ci-dessus vient de l'invariance conforme faible du Laplacien conforme (cf. proposition \ref{invcon}). La troisième inégalité est réalisée car pour tout $Q\in M-\{P\}$
$$|\tilde G_P(Q)|\leq  cd(P,Q)^{2-n}$$
donc $G_P\in L^s(M)$, pour tout $1\leq s<n/(n-2)$ et $L_{\gt}\frac{\varphi}{\psi}\in L^p(M)$ avec $p>n/2$. On peut donc choisir $s$ pour que $\langle\tilde G_P,L_{\gt}\frac{\varphi}{\psi}\rangle  _{\gt}$ soit fini.  
\end{proof}

\begin{proposition}\label{green coroll}
 Soit $M$ une variété compacte $C^\infty$ de dimension $n\geq 3$. $g$  une métrique riemannienne qui satisfait l'hypothèse $(H)$. Si $\mu(g)>0$, alors le Laplacien conforme $L_g$ admet une fonction de Green $G_{P_0}$, qui satisfait au sens des distributions $LG_{P_0}=\delta_{P_0}$ et
\begin{itemize} 
\item[$(i)$] $G_{P_0}$ est  $C^\infty$ sur $B_{P_0}(\delta)-\{P_0\}$
\item[$(ii)$] $G_{P_0}\in H^p_2(M-B_{P_0}(r))$ pour tout $r>0$. 
\item[$(iii)$] Il existe $c>0$ tel que pour tout $Q\in B_{P_0}(\delta)-\{P_0\}$,  $|G_{P_0}(Q)|\leq c d(P_0,Q)^{2-n} $
\end{itemize}
\end{proposition}

\begin{proof}
Puisque $\mu(g)>0$,  $L_g$ est nécessairement inversible. On en déduit que si $L_g$ admet une fonction de Green, celle-ci est unique. La proposition \ref{eysc} permet de montrer que l'équation
\begin{equation}\label{eygiii}
\Delta_g \psi+ \frac{n-2}{4(n-1)}R_g\psi= \mu_{q,G}(g) \psi^{q-1}
\end{equation}
admet une solution $\psi\in H^p_2(M)$, strictement positive (pour $q<N$ suffisamment proche de $N$ et $G=\{\mathrm{id}\}$). De plus, puisque la métrique $g$ est $C^\infty$ dans $B_{P_0}(\delta)$, les théorèmes de régularité montrent que $\psi$ est également $C^\infty$ dans cette même boule. La métrique $\gt:=\psi^{\frac{4}{n-2}}g$ satisfait donc l'hypothèse $(H)$. D'après l'équation de Yamabe \eqref{yamabe} (cf. page~\pageref{yamabe}), la courbure scalaire de la métrique $\gt$  est $$R_{\gt}=\frac{4(n-1)}{n-2}\mu_{q,G}(g) \psi^{q-N}$$ 
Par conséquent, $R_{\gt}$ est continue et strictement positive car $\mu_{q,G}(g)>0$. On est maintenant en mesure d'utiliser la proposition \ref{grgr}, qui assure l'existence d'une fonction de Green $\tilde G_{P_0}$ du Laplacien conforme $L_{\gt}$ pour la variété $M$ muni de la métrique $\gt$. Par la proposition \ref{green conf}, on conclut que $G_{P_0}=\psi(P_0)\psi\tilde G_{P_0}$ est la fonction de Green du Laplacien $L_g$. Comme les métriques $g$ et $\gt$ sont  $C^\infty$ sur $B_{P_0}(\delta)$ et que $\tilde G_{P_0}$ satisfait les propriétés de la proposition~\ref{grgr}, les propriétés énoncées pour $G_{P_0}$ sont vérifiées.
\end{proof}

On dit la fonction de Green $G_P$ est normalisée si 
$$\lim_{r\to 0}r^{2-n}G_P(Q)=1$$
Autrement dit, si $G_P$ est normalisée alors $$L_g G_P=(n-2)\omega_{n-1}\delta_P$$
où $r=d(P,Q)$ et $\omega_{n-1}$ est le volume de la sphère $S_{n-1}$. Lorsque il s'agit de la fonction de Green $G_{P_0}$ du Laplacien conforme $L_g$, on peut toujours la normaliser car elle est d'ordre $r^{2-n}$. On gardera la même notation pour la fonction de Green normalisée.

\section{La métrique de Cao--Günther}\label{metric cg}

Dans l'article \cite{LP} sur le problème de Yamabe, J.M.~Lee et T.~Parker ont montré que sur une variété riemannienne $(M,g)$, il existe un système de coordonnées normale $\{(U_i,x_i)\}_{i\in I}$ et une métrique   $g'$ conforme à $g$ tels que $\det g'=1+O(|x|^m)$ avec $m$ aussi grand que l'on veut. J.~Cao \cite{Cao} et M.~Günther \cite{Gun} ont montré (indépendamment) qu'on peut avoir, en fait, $\det g'=1$.
 
\begin{definition}
 Soit $(M,g)$ une variété riemannienne compacte. $\gt$ est une métrique de Cao--Günther, si elle est conforme à $g$ et s'il existe un système de coordonnées dans lequel $\det \gt=1$.  
\end{definition}

\begin{theorem}[Cao--Günther]\label{caogun}
Soient $M$ une variété de dimension $n$ et de classe $C^{a+2,\beta}$ avec $a\in\mathbb N$, $\beta\in]0,1[$. $g$ une métrique riemannienne de classe $C^{a+1,\beta}$, et $P$ un point de $M$. Alors il existe une fonction $\varphi$ strictement positive, de classe $C^{a+1,\beta'}$, avec $\beta'\in ]0,\beta[$ telle que  $\det(\varphi g)=1$, dans un système de coordonnées normales pour la métrique $\varphi g$ d'origine $P$. 
\end{theorem}
 
On remarque que si la métrique $g\in H_2^p(M,T^*M\otimes T^*M)$ avec $p>n$, alors elle est de classe $C^{1,\beta}$, la variété $(M,g)$ admet une métrique de Cao--Günther. Il n'est donc pas  utile de supposer que la métrique $g$ est $C^\infty$ dans une boule pour l'existence de telles coordonnées.

\section{Le théorème de la masse positive}

Dans cette section, on rappelle les résultats obtenus au sujet de la masse positive. 
\begin{definition}
 Une variété  riemannienne $M$ muni d'une métrique $C^\infty$, $g$ est dite asymptotiquement plate d'ordre $\tau>0$, s'il existe  une  décomposition $M=M_0\cup M_\infty$ (avec $M_0$ compacte) et un difféomorphisme $M_\infty\rightarrow \mathbb R^n-B_R(O)$ pour un certain $R>0$ tels que:
\begin{equation}\label{ogog}
 g_{ij}=\delta_{ij}+O(\rho^{-\tau}),\quad \partial_k g_{ij}=O(\rho^{-\tau-1}),\quad \partial_{kl}g_{ij}=O(\rho^{-\tau-2})
\end{equation}
 quand $\rho=|z|\to +\infty$ dans les coordonnées $\{z^i\}$ induites sur $M_\infty$. Les  $\{z^i\}$ sont appelés les coordonnées asymptotiques.
\end{definition}
On écrit $g_{ij}=\delta_{ij}+O''(\rho^{-\tau})$ si $g_{ij}$ satisfait \eqref{ogog}. D'une façon analogue, on peut définir $O''$ pour tout fonction. 
\begin{definition}

Etant donné une variété riemannienne asymptotiquement plate $(M,g)$ avec des coordonnées asymptotiques $\{z^i\}$, on définit la  masse de la façon suivante:
$$m(g)=\lim_{\rho\to +\infty}\omega_{n-1}^{-1}\int_{\partial B_P(\rho)}\partial_\rho(g_{\rho\rho}-g_{ii})+\rho^{-1}(ng_{\rho\rho}-g_{ii})d\sigma_\rho$$
\end{definition}
Cette  définition de la  masse dépend des coordonnées asymptotiques. R.~Bartnik \cite{Bar} a montré que si $(M,g)$  asymptotiquement plate d'ordre $\tau>(n-2)/2$, alors $m(g)$ est bien définie et dépend seulement de la métrique $g$.\\
Le théorème de la masse positive s'énonce comme suit:
\begin{theorem}\label{msp1}
 Soit $(M,g)$ une variété riemannienne de dimension $n\geq 3$, asymptotiquement plate d'ordre $\tau>(n-2)/2$, de courbure scalaire positive. La masse $m(g)$ est toujours positive ou nulle. De plus $m(g)=0$ si et seulement si  $(M,g)$ est isométrique à l'espace euclidien $(\mathbb R^n,\mathcal E)$ muni de sa métrique canonique.  
\end{theorem}
Beaucoup de mathématiciens ont contribué à la preuve de ce théorème, essentiellement T.~Aubin \cite{Aub2,Aub6} R.~Schoen et S.T.~Yau \cite{SY,SY2,SY3}, E.~Witten \cite{Wit}. \\
Récemment T.~Aubin \cite{Aub2} a montré que:
\begin{theorem}\label{aubm}
Si $g$ est une métrique de Cao--Günther, $L_g$ est inversible et si au voisinage de $P_0\in M$ la fonction de Green normalisée $G_{P_0}$ de $L_g$ s'écrit
$$G_{P_0}(Q)=r^{2-n}+A+O(r)$$ 
avec $r=d(P_0,Q)$, alors $A>0$ sauf si $(M,g)$ est conformément difféomorphe à la sphère $(S_n,g_{can})$, auquel cas $A=0$.
\end{theorem}

On utilisera les deux théorèmes \ref{msp1}, \ref{aubm}, sous réserve de leur validité.

\section{Théorème d'existence de solutions sans présence de symétries}

\begin{theorem}\label{conj aub}
 Soit $M$ une variété compacte $C^\infty$ de dimension $n\geq 3$, $g$ une métrique riemannienne qui satisfait l'hypothèse $(H)$. Si $(M,g)$ n'est pas conformément difféomorphe à la sphère $(S_n, g_{can})$,  alors $\mu(g)<K^{-2}(n,2)$.
\end{theorem}
On montre ce théorème sous réserve de la validité du théorème \ref{aubm}.\\
Ce théorème affirme que la conjecture de T.~Aubin \ref{Aubincon} reste vraie pour des métriques qui satisfont l'hypothèse $(H)$ (pas nécessairement $C^\infty$ partout).  \\
Pour montrer ce théorème, on se base sur les travaux de T.~Aubin et R.~Schoen dans le cas où $g$ est $C^\infty$. La stratégie est la suivante: on construit des fonctions test pour la fonctionnelle $I_g$, à support dans des petites boules géodésiques. Puisque le problème est local et que la métrique $g$ est $C^\infty$ sur la boule $B_{P_0}(\delta)$, alors la preuve du théorème ci-dessus est  identique à celle dont la métrique $g$ est $C^\infty$ sur $M$ (c'est pour cette raison qu'on a supposé que la métrique est $C^\infty$ dans la boule $B_{P_0}(\delta)$). On  prendra donc les fonctions test de T.~Aubin et R.~Schoen à support dans $B_{P_0}(\delta)$.\\

\begin{proof}[\textbf{Preuve du théorème \ref{conj aub}}]
Si $\mu(g)\leq 0$ alors l'inégalité est triviale. \`A partir de maintenant jusqu'à la fin de la preuve, on suppose que $\mu(g)>0$. Quitte à considérer une métrique conforme, on peut supposer que $g$ est la métrique de Cao--Günther donnée par le théorème \ref{caogun}. En effet, $\mu(g)$ est un invariant conforme d'après la proposition \ref{invconforme}.\\
 Deux cas se présentent:\\
$(a)$ Soit $(M,g)$ n'est pas conformément plate en $P_0$ et $n\geq 6$. Dans ce cas, on pose $\varphi_\e=\eta v_\e$, $\eta$ une fonction cut-off de support dans $B_{P_0}(2\e)$, $\eta=1$ sur $B_{P_0}(\e)$, $2\e<\delta$ et
$$v_\e(Q)=\biggl(\frac{\e}{r^2+\e^2}\biggr)^{\frac{n-2}{2}}\quad r=d(P_0,Q)$$
Comme $supp\varphi\subset B_{P_0}(\delta)$ et que la métrique $g$ est 
de classe $C^\infty$ sur cette boule, on obtient le lemme suivant (cf. T.~Aubin \cite{Aub}): 
\begin{lemma}\label{lemme aub yam}
\begin{equation*}
\mu(g)\leq I_g(\varphi_\e)\leq \begin{cases}& K^{-2}(n,2)-c|W(P_0)|^2\e^4+o(\e^4)\text{ si }n>6\\
& K^{-2}(n,2)-c|W(P_0)|^2\e^4\log \frac{1}{\e}+O(\e^4)\text{ si }n=6
\end{cases}
\end{equation*}
où $|W(P_0)|$ est la norme du tenseur de Weyl au point $P_0$.
\end{lemma}
J.M.~Lee et T.~Parker  ont donné une preuve simple de ce lemme, en utilisant les coordonnées géodésiques conformes en $P_0$ (cf. \cite{LP}). 
Par hypothèse la métrique n'est pas conformément plate au voisinage de $P_0$ et $n\geq 6$ donc 
$|W(P_0)|\neq 0$ d'où $\mu(g)<K^{-2}(n,2)$.\\
$(b)$ Soit $(M,g)$ est conformément plate en $P_0$ ou $n=3,\; 4\text{ ou }5$: Puisque $\mu(g)$ est un invariant conforme, quitte à considérer une métrique conforme à $g$, on peut supposer que la métrique est celle de Cao--Günther et que la fonction de Green normalisée $G_{P_0}$, construite dans la proposition \ref{green coroll}, s'écrit:
$$G_{P_0}(Q)=r^{2-n}+A+O(r)$$
au voisinage de $P_0$, avec $r=d(P_0,Q)$ (cf. l'article de J.M.~Lee et T.~Parker \cite{LP} pour la preuve de ce développement limité).\\
Si la métrique $g$ satisfait l' hypothèse $(H)$ et que $(M,g)$ n'est pas conformément 
difféomorphe à la sphère $(S_n,g_{can})$, par le théorème \ref{aubm}, nous savons que $A>0$. 
Considérons alors $\varphi_\e$, la fonction test introduite par R.~Schoen \cite{Schoen}, définie pour tout $Q\in M$ par:
\begin{equation*}
\varphi_\e(Q)= \begin{cases}& v_\e(Q)\text{ si }Q\in B_{P_0}(\rho_0)\\
& \e_0[G_{P_0}-\eta (G_{P_0}-r^{2-n}-A)](Q) \text{ si }Q\in B_{P_0}(2\rho_0)-B_{P_0}(\rho_0)\\
& \e_0 G_{P_0}(Q)\text{ si }Q\in M-B_{P_0}(2\rho_0)
\end{cases}
\end{equation*}
avec $2\rho_0<\delta$, $(\frac{\e}{\rho_0^2+\e^2})^{(n-2)/2}=\e_0(\rho_0^{2-n}+A)$ et 
$\eta$ une fonction réelle positive  $C^{\infty}$, décroissante sur $\mathbb{R}_+$, à support dans $]-2\rho_0,2\rho_0[$, identiquement égale à $1$ sur $[0,\rho_0]$, dont le gradient vérifie
$|\nabla\eta(r)|\leq\rho_0^{-1}$. Puisque la métrique $g$ 
est $C^\infty$ sur $B_{P_0}(2\rho_0)\subset B_{P_0}(\delta)$ et que $G_{P_0}\in H^p_2(M-B_{P_0}(\rho_0))$ (voir le corollaire \ref{green coroll}), alors on a l'estimée suivante de $\mu(g)$, obtenue par R.~Schoen~\cite{Schoen}:
\begin{lemma}
 $$\mu(g)\leq I_g(\varphi_\e)\leq K^{-2}(n,2)+c\e_0^2(c\rho_0-A)$$
\end{lemma}
Comme $A>0$  alors on peut choisir  $\rho_0$ suffisamment petit ($c\rho_0<A$) pour que $\mu(g)<K^{-2}(n,2)$.\\
\end{proof}

On est maintenant en mesure d'énoncer le théorème qui résout le problème \ref{yam sing} pour les métriques qui satisfont l'hypothèse $(H)$. 

\begin{theorem}\label{inegg}
Soit $M$ une variété compacte $C^\infty$ de dimension $n\geq 3$, $g$ une métrique riemannienne qui satisfait l'hypothèse $(H)$, alors il existe une métrique $\tilde g$ conforme à $g$ ayant une courbure scalaire $R_{\tilde g}$ constante, solution du problème \ref{yam sing}.
\end{theorem}

Ce théorème affirme qu'il existe toujours des solutions pour l'équation de type Yamabe~\eqref{AF} (page \pageref{AF}) et que l'hypothèse du  théorème \ref{cg} est toujours satisfaite avec $h=\frac{n-2}{4(n-1)}R_g$. 

\paragraph{\textbf{Remarque}}Dans l'énoncé du théorème \ref{cg}, la métrique $g$ est supposée être de classe $C^\infty$.  Ce théorème reste vrai si l'on suppose que la métrique est dans $H^p_2$ avec $p>n$. Pour le voir, il suffit de remarquer que si $g\in H^p_2$, il existe une solution faible pour l'équation \eqref{AF} (preuve identique).   La seule chose qui peut changer est la régularité de la solution faible. Dans ce cas, on aura la même régularité car  les coefficients de $\Delta_g$ sont continus.

\begin{proof}

Si $(M,g)$ est conformément difféomorphe à la sphère $S_n$, munie de la métrique canonique $g_{can}$, alors il n'y a rien à montrer car $(S_n,g_{can})$ est à courbure scalaire constante. Sinon $(M_n,g)$ n'est pas conformément difféomorphe à $(S_n,g_{can})$. Au quel cas, on a l'inégalité
$$\mu(g)<K^{-2}(n,2)$$ 
par le théorème \ref{conj aub}. Le théorème \ref{cg} nous fournit une solution $\psi\in H^p_2(M)$, strictement positive, de l'équation \eqref{AF}, où $h=\frac{n-2}{4(n-1)}R_g$ et $\tilde h=\mu(g)$. D'après l'équation \eqref{yamabe}, la métrique $\gt=\psi^{\frac{4}{n-2}}g$ est à courbure scalaire constante $R_{\gt}=\frac{4(n-1)}{n-2}\mu(g)$.
\end{proof}

\section{Unicité des solutions}

Pour le problème de Yamabe classique (i.e. la métrique $g$ est $C^\infty$), on sait qu'on a unicité des solutions à une constante multiplicative près dans le cas où l'invariant conforme de Yamabe $\mu(g)$ est négatif ou nul. Le théorème suivant, montre qu'on a toujours les mêmes résultats lorsque la métrique est seulement de classe $H^p_2$, avec $p>n$.
\begin{theorem}\label{unique}
 Soit $g$ une métrique dans  $H^p_2(M,T^*M\otimes T^*M)$, avec $p>n$. Si $\mu(g)\leq 0$, alors les solutions de l'équation \eqref{yamabe} sont uniques à une constante multiplicative près.
\end{theorem}

\begin{proof}
 Soit $\varphi_1$ et $\varphi_2$ deux solutions strictement positives de l'équation \eqref{yamabe}. Les métriques $g_i=\varphi_i^{\frac{4}{n-2}}g$ sont à courbures scalaires constantes $R_i$, où $i=1$ ou $2$. On pose $\psi=\frac{\varphi_1}{\varphi_2}$, donc $g_1=\psi^{\frac{4}{n-2}}g_2$. Ce qui entraîne que $\psi$ satisfait
\begin{equation}\label{unieq}
\Delta_{g_2} \psi+ \frac{n-2}{4(n-1)}R_{2}\psi=  \frac{n-2}{4(n-1)}R_{1} \psi^{\frac{n+2}{n-2}} 
\end{equation}

Par le théorème de régularité \ref{reg}, on en déduit que $\psi$ est de classe $C^{2,\beta}$ car les coefficients du Laplacien sont $C^0$. En effet, dans une carte locale:
$$\Delta_g\psi=-\nabla_i\nabla^i\psi=-g^{ij}(\partial_{ij}\psi-\Gamma^k_{ij}\partial_k\psi)$$
et les Christoffels sont donnés par 
$$\Gamma^k_{ij}=g^{kl}(\partial_i g_{lj}+\partial_jg_{il}-\partial_l g_{ij})$$
Ils sont dans $H^p_1$, et continus si $p>n$. D'autre part, remarquons que $R_1$ et $R_2$ sont forcément de même signe. Pour le voir, il suffit d'intégrer l'équation \eqref{unieq} sur $M$, avec l'élément de volume de $g_2$, et utiliser le fait que l'intégrale du Laplacien d'une fonction $C^2$ est toujours nulle.  \\
Si $\mu(g)<0$, alors $R_i<0$ pour $i=1$ et 2. Supposons que $\psi$ atteint son maximum en $Q_1\in M$ et son minimum en $Q_2\in M$ alors $\Delta_{g_2}\psi(Q_1)\geq 0$ et  $\Delta_{g_2}\psi(Q_2)\leq 0$. Par conséquent, si on évalue l'équation \eqref{unieq} au point $Q_1$ et $Q_2$, on obtient les deux inégalités suivantes:
$$\psi^{\frac{4}{n-2}}(Q_1)\leq \frac{R_2}{R_1}\mbox{ et }\psi^{\frac{4}{n-2}}(Q_2)\geq \frac{R_2}{R_1}$$ 
de là on tire que $\psi=\frac{R_2}{R_1}$ et que $\varphi_1$ et $\varphi_2$ sont proportionnelles.\\
Si $\mu(g)=0$ alors $R_1=R_2=0$ et l'équation \eqref{unieq} est réduite à $\Delta_{g_2}\psi=0$, d'où $\psi$ est constante.
\end{proof}

\section{Application}

Prenons le cas particulier d'une métrique 
\begin{equation*}
g_\alpha=(1+\rho_{P_0}^{2-\alpha})^m g_0 
\end{equation*}
 où $g_0$ est une métrique riemannienne $C^\infty$, $\alpha\in]0,1[$ et $\rho_{P_0}$ la fonction distance donnée par la définition \ref{distance} (page \pageref{distance}). Les dérivées secondes de $g_\alpha$ ont des singularités du type $\rho^{-\alpha}$, ce qui entraîne qu'il existe une fonction continue $R_0$ telle que la courbure scalaire de $g$ soit de la forme $R_{g_\alpha}=\frac{R_0}{\rho^\alpha}$. Cette fonction est dans $L^p(M)$, si $p<\frac{n}{\alpha}$. Comme $\alpha\in]0,1[$ alors on peut trouver un $p>n$ car $\frac{n}{\alpha}>n$. Pour cette valeur de $p$, on conclut que la métrique $g_\alpha$ est dans $H^p_2(M,T^*M\otimes T^*M)$ et elle satisfait l'hypothèse $(H)$ car la fonction $\rho_{P_0}$ est $C^\infty$ sur  $B_{P_0}(\delta(M))-\{P_0\}$, avec $\delta(M)$ le rayon d'injectivité de $M$.\\
Soit $\varphi$ une fonction strictement positive dans $H^p_2(M)$ et $\gt=\varphi^{\frac{4}{n-2}}g_{\alpha}$ une métrique conforme à $g_\alpha$. Si on veut que $\gt$ soit une métrique qui résout le problème \ref{yam sing} (cf. page \pageref{yam sing}) alors il suffit que $\varphi$ soit solution de l'équation de type Yamabe \eqref{yamsi}. D'après le théorème \ref{inegg} et la proposition \ref{corollaire}, une telle solution existe toujours.

\section{Le problème de Yamabe équivariant}\label{Equi Yam pro}

\subsection{Le problème de Hebey--Vaugon }\label{HVprob}

Soit $(M,g)$ une variété riemannienne compacte $C^\infty$ de dimension $n$. $G$ un sous groupe du groupe d'isométries $I(M,g)$. E.~Hebey et  M.~Vaugon \cite{HV} ont considéré le problème suivant:

\begin{problem}\label{HVProbleme}
 Existe-t-il une métrique $g_0$, $G-$invariante qui minimise la fonctionnelle
$$J(g')=\frac{\int_MR_{g'}\di v(g')}{(\int_M\di v(g'))^{\frac{n-2}{n}}}$$
où $g'$ appartient à la classe $G-$conforme de $g$:
$$[g]^G:=\{\tilde g=e^fg/ f\in C^\infty(M),\; \sigma^*\tilde g=\tilde g\quad \forall \sigma\in G\}$$
\end{problem}
(Les définitions sont données dans la section \ref{group isome sect}).\\
Ils ont démontré que ce problème à toujours des solutions, sous réserve de la validité du théorème de la masse positive \ref{msp1}. La résolution de ce problème a deux conséquences. La première est l'existence d'une métrique $g_0$, $I(M,g)-$invariante et conforme à $g$, telle que la courbure scalaire de $g_0$ est constante. En effet, si $g_0=\varphi^{\frac{4}{n-2}}g$ est une métrique qui minimise $J$, alors $\varphi$ est $I(M,g)-$invariante, solution de l'équation d'Euler--Lagrange de $J$. Cette équation est bien celle de Yamabe \eqref{yamabe}, avec $R_{\gt}=R_{g_0}$  une constante qui joue le rôle de la courbure scalaire de $g_0$.  La deuxième conséquence est  que la conjecture de A.~Lichnerowicz~\cite{Lic} ci-dessous est vraie. Par les travaux de J.~Lelong-Ferrand~\cite{Lel} et  M.~Obata~\cite{Oba}, on sait que si $(M,g)$ n'est pas conformément difféomorphe à $(S_n,g_{can})$ alors le groupe conforme $C(M,g)$ est compact et il existe une métrique  $g'$ conforme à $g$ telle que $I(M,g')=C(M,g)$. 
\begin{conjecture}[\textbf{A.~Lichnerowicz \cite{Lic}}]
Pour tout variété riemannienne $(M,g)$, compacte $C^\infty$, de dimension $n$ et qui n'est pas conformément difféomorphe à $(S_n,g_{can})$, il existe une métrique $\gt$ conforme à $g$ de courbure scalaire $R_{\tilde g}$ constante et pour laquelle $I(M,\tilde g)=C(M,g)$.
\end{conjecture}

On a déjà remarqué que les métriques qui résolvent le problème de Hebey--Vaugon \ref{HVProbleme} sont nécessairement solutions de l'équation de Yamabe \eqref{yamabe}. Par conséquent, le problème de Yamabe classique, décrit à la section \ref{yampro}, correspond au cas particulier $G=\{\mathrm{id}\}$ du problème \ref{HVProbleme}.\\

Au début de ce chapitre, on a rappelé le problème de Yamabe, ensuite on a montré que les équations de type Yamabe \eqref{AF} admettent toujours des solutions, si la fonction $h$ est proportionnelle à la courbure scalaire $R_g$ (cf. théorème \ref{inegg}). On essaye de faire  le même travail lorsque un sous groupe $G$ du groupe d'isométries agit sur $M$. Les métriques ne seront pas nécessairement $C^\infty$, mais elles vérifient  l'hypothèse $(H)$ (cf. section \ref{hypG}).

\begin{problem}\label{YMsing} 
Supposons que la métrique $g\in H^p_2(M,T^*M\otimes T^*M)$. Existe-t-il une métrique $\tilde g$ dans la classe conforme  $G-$invariante de $g$ qui minimise la fonctionnelle $J$ et pour laquelle la courbure scalaire $R_{\gt}$ est constante partout?
\end{problem}

Si la métrique $g$ est $C^\infty$ alors ce problème est exactement le problème de Hebey--Vaugon~\ref{HVProbleme}. Si la métrique  $\tilde g$ minimise la fonctionnelle $J$ définie au début de la section \ref{Equi Yam pro}, alors la courbure scalaire de $\gt$ est automatiquement constante. Plus précisément, si $\tilde g=\psi^{4/(n-2)}g$, avec $\psi\in H_2^p(M)$, strictement positive et $G-$invariante, alors $\psi$ est solution de l'équation de Yamabe \eqref{yamabe}.

\subsection{L'invariant de Yamabe $\boldsymbol{G-}$conforme}

Soit $I_g$ la fonctionnelle de Yamabe définie par \eqref{fonctionnel}  (page \pageref{fonctionnel}). Pour ce problème, on considère  seulement des fonctions test dans $H_{1,G}(M)$, l'espace des fonctions dans $H_1(M)$, $G-$invariante.  

\begin{definition}
L'invariant $G-$conforme de Yamabe $\mu_G(g)$ est défini par: 
 \begin{equation*}
\mu_G(g)=\inf_{\psi\in H_{1,G}(M)-\{0\}}I_g(\psi)
\end{equation*}
\end{definition}

La proposition suivante justifie la terminologie employée.
\begin{proposition}\label{mu invariance}
Soit $M$ une variété compacte $C^\infty$. $g\in H^p_2(M,T^*M\otimes T^*M)$ une métrique riemannienne, avec $p>n/2$. Alors 
\begin{enumerate}
\item $\mu_G(g)=\frac{n-2}{4(n-1)}\inf_{g'\in [g]^G} J(g')$
 \item Si $\gt\in [g]^G$ alors  $\mu_G(\gt)=\mu_G(g)$.
\end{enumerate}
\end{proposition}

\begin{proof}
Pour tout $g'\in [g]^G$, il existe $\psi\in H^p_{2,G}(M)$, strictement positive telle que $g'=\psi^{\frac{4}{n-2}}g$. Par l'équation de Yamabe \eqref{yamabe}:
$$R_{g'}=\psi^{-\frac{n+2}{n-2}}(4\frac{n-1}{n-2}\Delta_g\psi+R_g\psi)$$
En intégrant cette équation sur $M$ par rapport à l'élément de volume $\di v_{g'}$, on obtient
$$\int_MR_{g'}\di v_{g'}=\int_M \psi(4\frac{n-1}{n-2}\Delta_g\psi+R_g\psi)\di v_g=4\frac{n-1}{n-2}E(\psi)$$
D'autre part
$$\int_M\di v_{g'}=\|\psi\|_N^N$$
On en déduit que 
\begin{equation}\label{relation J I}
 J(g')=4\frac{n-1}{n-2}I_g(\psi)
\end{equation}
En prenant la borne inférieure, on obtient la première propriété. Pour la seconde propriété, il suffit de reprendre la preuve de la proposition \ref{invconforme}. En effet, si $g'$ est la métrique considérée ci-dessus alors, d'après l'équation \eqref{IgIg}
$$\forall \varphi\in H_{1,G}(M)\qquad I_{g'}(\varphi)=I_g(\psi\varphi)$$
\end{proof}

\section{Théorème d'existence de solutions en présence de symétries}\label{existence g invar}

\begin{theorem}\label{HVGINVD}
Soit $M$ une variété compacte $C^\infty$ de dimension $n\geq 3$. $g$ une métrique  riemannienne qui appartient à $H_2^p(M,T^*M\otimes T^*M)$, avec $p>n$. Si 
$$\mu_G(g)<\frac{1}{4}n(n-2)\omega_{n}^{2/n}(\Ca O_G(P))^{2/n}$$ 
alors l'équation \eqref{yamabe} admet une solution strictement positive $\varphi\in H_{2,G}^p(M)\subset C^{1-[n/p],\beta}(M)$, $G-$invariante. De plus la métrique $\gt=\varphi^{\frac{4}{n-2}}g$ est solution du problème \ref{YMsing} et de courbure scalaire constante $R_{\gt}=\frac{4(n-1)}{n-2}\mu_G(g)$.
\end{theorem}

\begin{proof}
Si $\mu_G(g)\leq 0$, d'après le théorème \ref{unique}, les solutions de l'équation de Yamabe  sont proportionnelles. Si $\varphi$ est une solution de \eqref{yamabe} alors pour tout $\sigma\in G$, $\sigma^*\varphi$ est également une solution. Il existe donc une constante $c>0$ telle que $\sigma^*\varphi=c\varphi$. D'autre part, $\|\sigma^*\varphi\|_2=\|\varphi\|_2$. On en déduit que $c=1$ et que $\varphi$ est $G-$invariante.\\
Supposons que $\mu_G(g)>0$. Notons que 
$$K^{-2}(n,2)=\frac{1}{4}n(n-2)\omega_{n}^{2/n}$$ 
l'expression de $K(n,q)$ est donnée dans le théorème \ref{INM} (page \eqref{INM}). Il suffit d'appliquer le théorème~\ref{theginv} pour $h=\frac{n-2}{4(n-1)}R_g$, qui entraîne que l'équation \eqref{yamabe} admet une solution $\varphi\in H^p_{2,G}(M)$, strictement positive et minimisante pour la fonctionnelle $I_g$. D'après la relation \eqref{relation J I}, la métrique $\varphi^{\frac{4}{n-2}}g$ minimise la fonctionnelle $J'$.
\end{proof}

\paragraph{\textbf{Remarque}}
D'après le théorème \ref{HVGINVD}, la condition suffisante, pour trouver  une solution $G-$invariante de l'équation de Yamabe \eqref{yamabe} est que l'inégalité 
$$\mu_G(g)<n(n-1)\omega_{n-1}^{2/n}(\Ca O_G(P))^{2/n}$$ 
soit toujours vraie.\\ 
On a vu que dans le cas particulier où $G=\{\mathrm{id}\}$, cette inégalité est vraie pour toute variété compacte $(M,g)$, non conformément difféomorphe à $(S_n,g_{can})$, munie d'une métrique $g$ qui satisfait l'hypothèse $(H)$ (cf. théorème \ref{conj aub}). \\
Dans le cas où $G$ est un sous groupe quelconque de $I(M,g)$ lorsque $(M,g)$ est une variété riemannienne compacte $C^\infty$, E.~Hebey et M.~Vaugon \cite{HV} ont annoncé cette inégalité  sous forme de conjecture (cf. conjecture \ref{HVcon}). Ils l'ont démontrée dans certains cas (cf. théorème \ref{HV theorem}). Dans le chapitre \ref{CHYSr}, on démontre qu'elle est  vraie dans de nouveaux cas (par contre, vu la complexité de la preuve et des arguments utilisés, on n'est pas encore en mesure d'adapter la preuve, dans le cas où la métrique est seulement dans $H^p_2$).

\chapter{Calculs techniques sur la courbure scalaire}\label{CHAub}

Dans tout ce chapitre, on suppose que $M$ est une variété compacte $C^\infty$, de dimension $n\geq 3$, $g$ 
est une métrique riemannienne $C^\infty$, munie de sa connexion riemannienne, notée $\nabla_g$. On note par $\nabla^\beta$ la dérivée covariante $\nabla^{\beta_1}\cdots\nabla^{\beta_i}$, où $\beta\in \intl 1, n\intr^i$ sont des  multi-indices, et $|\beta|=i$. On note par $\intl 1, n\intr$ l'ensemble des entiers naturels entre  $1$ et $n$.

\begin{definition}\label{omega defintione}
Soit $(M,g)$ une variété riemannienne et $W_g$ le tenseur de Weyl associé à $g$. On définit l'entier $\omega$ au point $P$ par
$$\omega(P)=\inf \{|\beta|\in \mathbb N/\|\nabla^\beta W_g(P)\|\neq 0\}$$
(et si $\|\nabla^\beta W_g(P)\|= 0$ pour tout multi-indices $\beta$, alors $\omega(P)=+\infty$).
\end{definition}
Pour des raisons de simplicité, on omet $P$ dans $\omega (P)$. On a les propriétés suivantes:
\begin{proprietes}\label{omega invariant}
Soit $\gt$ une métrique conforme à $g$. On note $\tilde\omega$ l'entier  défini ci-dessus associé à la métrique $\gt$. Alors 
$$\omega=\tilde\omega$$ 
$\omega$ est conformément invariant.
\end{proprietes}

\begin{proof}
 Si $\gt=\varphi^\N g$, alors $W_{\gt}=\varphi^\N W_g$ (cf. remarque après la  définition \ref{Weyl def}), avec $\varphi$ une fonction $C^\infty$, strictement positive. Par conséquent
$$\forall i<\omega\quad \nabla^i W_g(P)=0\Longleftrightarrow \forall i<\tilde\omega\quad \nabla^i W_{\gt}(P)=0 $$ 
\end{proof}

\section{Calculs  sur l'intégrale de la courbure scalaire }\label{aubin computations}

Cette section est consacrée au calcul de l'intégrale de la courbure scalaire  sur une sphère de rayon $r$ assez petit. Ces calculs ont été effectués  par T.~Aubin \cite{Aub5, Aub4}, que l'on reprendra, avec des preuves détaillées. Notons par $S(r)$ la sphère de dimension $n-1$ et de rayon $r$, et par $\di\sigma_r$ l'élément de volume sur $S(r)$. On note par $\bar\int$ la valeur moyenne 
$$\bar \int_M \varphi\di v=\frac{1}{vol(M)}\int_M\varphi\di v$$ 
L' intégrale de la courbure scalaire que l'on calculera joue un rôle important dans  la fonctionnelle de Yamabe \eqref{yamm}. On verra que si elle est négative alors la conjecture de Hebey--Vaugon  \ref{HVcon} est démontrée.  Mais dans certains cas, elle est positive, ce qui complique la preuve de la conjecture \ref{HVcon}. Notons qu'on a déjà démontré que l'inégalité large suivante est toujours vraie  
$$\mu_G(g)\leq\frac{n(n-2)}{4}\omega_n^{2/n}(\Ca O_G(P))^{2/n}$$
pour tout variété compacte $(M,g)$, de dimension $n\geq 3$ (cf. proposition \ref{ine largee}, page \pageref{ine largee} ), même dans le cas où on met $h$ une fonction quelconque à la place  de $R_g$. On constate qu'il y a certaines informations contenues dans $R_g$ qu'il faut absolument utiliser  pour démontrer la conjecture.

\begin{definition}\label{mu defin}
Soit $P$ un point fixé de $M$. On note $\mu(P)$ l'entier naturel, défini comme suit:  $|\nabla_\beta R_g(P)|=0$ pour tout $|\beta|<\mu(P)$ et il existe $\beta\in \intl 1, n\intr^{\mu(P)}$ tel que $|\nabla_\beta R_g(P)|\neq 0$. Dans un système de coordonnées normales $\{x^i\}$ d'origine $P$
$$R_g(Q)=\bar R +O(r^{\mu(P)+1})$$ 
où   $\bar R=r^{\mu(P)}\sum_{|\beta|= \mu(P)}\nabla_\beta R_g(P)\xi^\beta$ est un polynôme homogène de degré  $\mu(P)$, qui représente la partie principale de $R_g$, $r=d(P,Q)=|x|$ et $\xi^i=\frac{x^i}{r}$.
\end{definition}
Pour des raisons de simplicité, on omet $P$ dans $\mu(P)$.\\
Le lemme \ref{lemme HV ex}, énoncé dans le chapitre suivant, et le développement limité de la métrique donnent: 
\begin{lemma}\label{LLLL}
 On a toujours $\mu\geq \omega$, $g_{ij}=\delta_{ij}+O(r^{\omega+2})$ et $\bar\int_{S(r)}R_g\di \sigma_r=O(r^{2\omega+2})$ ce qui entraîne que  $\int_{S(r)}\bar R\di\sigma_r=0$ lorsque $\mu<2\omega+2$. 
\end{lemma}
\begin{proof}
Par le développement limité \eqref{expansion metric} (voir le chapitre suivant),  $g_{ij}=\delta_{ij}+O(r^{\omega+2})$. Puisque la courbure scalaire $R_g$ est obtenue, en dérivant deux fois les composantes de la métrique, alors $R_g=O(r^\omega)$. Ce qui veut dire que la partie principale $\bar R$ est d'ordre $\mu\geq\omega$.\\
Intéressons nous à $\bar\int_{S(r)}R_g\di \sigma_r$, elle est d'ordre $2\omega+2$. En effet, on a le développement suivant
$$R_g(Q)=\sum_{m=\mu}^{2\omega+1}(\sum_{|\beta|= m}\nabla_\beta R_g(P)\xi^\beta)r^m+O(r^{2\omega+2})$$
avec $r=d(P,Q)$ et $(r,\xi^j)$  un système de  coordonnées géodésiques. En intégrant cette égalité sur la sphère $S(r)$, sachant que l'intégrale d'un polynôme homogène de degré impair sur la sphère est nulle, on obtient
\begin{equation*}
\bar\int_{S(r)}R_g\di\sigma_r=\sum_{m=\mu}^{\omega}C(m,n)\Delta_g^mR_g(P)r^m+O(r^{2\omega+2})
\end{equation*}
pour une certaine constante $C(m,n)$, qui dépend seulement de  $n$ et $m$. Comme les courbures de la métrique $g$ satisfont le lemme \ref{lemme HV ex}, pour tout $m\leq \omega$, $\Delta_g^mR_g(P)=0$. Donc 
\begin{equation*}
\bar\int_{S(r)}R_g\di\sigma_r=O(r^{2\omega+2})
\end{equation*}
\end{proof}

Soit\label{detailh}  $\{x^\alpha\}$ un système de coordonnées normal en $P$ et $\{r,\xi^i\}$ un système de coordonnées géodésiques. Le lemme \ref{LLLL} entraîne qu'il existe un tenseur symétrique $h$ tel que
\begin{equation*}
 g=\mathcal E+h
\end{equation*}
avec $h=O(r^{\omega+2})$, alors 
\begin{equation*}
 g=\mathcal E+h=(\delta_{\alpha\beta}+h_{\alpha\beta})dx^\alpha \otimes dx^\beta=dr^2+(s_{ij}+h_{ij})(rd\xi^i)\otimes(rd\xi^j)
\end{equation*}
où $(s_{ij})$ sont les composantes de la métrique standard sur la sphère $S_{n-1}$ et 
$$h_{ij}=\frac{\partial x^\alpha}{r\partial \xi^i}\frac{\partial x^\beta}{r\partial \xi^j}h_{\alpha\beta},\quad h_{ir}=h_{rr}=0$$
Remarquons que $h_{ij}=O(r^{\omega+2})$. On peut donc décomposer $(h_{ij})$ de la façon suivante:
\begin{equation}\label{defhij}
h_{ij}=r^{\omega+2}\bar g_{ij}+r^{2(\omega+2)}\hat{g}_{ij}+\tilde h_{ij}
\end{equation}
où $\bar g$, $\hat g$ et $\tilde h$ sont des 2-tenseurs symétriques définis sur la sphère  $S_{n-1}$. On choisit $\{\frac{\partial}{\partial r},\frac{\partial}{r\partial\xi^i}\}_{1\leq i\leq n-1}$ et $\{dr,rd\xi^i\}_{1\leq i\leq n-1}$ comme  bases locales de  l'espace tangent $TM$ et cotangent $T^*M$  respectivement. Notre but dans le choix de ces bases est d'avoir
$$g_{ij}=s_{ij}+h_{ij},\; g_{rr}=1\text{ et }g_{ir}=0$$
et d'éliminer une fois pour toutes le $r^2$ qui apparaît,  en passant aux coordonnées géodésiques.  Les composantes  $g^{ij}$ de l'inverse  de la métrique sont 
$$g^{ij}=s^{ij}-h^{ij}+O(r^{2\omega+4})$$ 
où $h^{ij}=s^{ik}s^{jl}h_{lk}$. On fait monter et baisser les indices, en utilisant la métrique $(s_{ij})$, sauf pour la métrique $g$. On note par $\nabla$ la connexion riemannienne sur la sphère, associée à $s$. Par des calculs directs,  T.~Aubin \cite{Aub5} a montré que:
\begin{theorem}\label{zzz}
$$\bar R=\nabla^{ij}\bar g_{ij}r^\omega\quad\text{et}$$
$$\bar\int_{S(r)}R\di\sigma_r =[B/2-C/4-(1+\omega/2)^2Q]r^{2(\omega+1)}+o(r^{2(\omega+1)})$$
où $B=\bar\int_{S_{n-1}}\hspace{-0.5cm}\nabla^i\bar g^{jk}\nabla_j\bar g_{ik}\di\sigma$, $C=\bar\int_{S_{n-1}}\hspace{-0.5cm}\nabla^i\bar g^{jk}\nabla_i\bar g_{jk}\di\sigma$ et $Q= \bar\int_{S_{n-1}}\bar g_{ij}\bar g^{ij}\di\sigma$
\end{theorem}

Une preuve détaillée de ce lemme est donnée dans l'appendice \ref{calculus}. \\

De plus T.~Aubin \cite{Aub3} a montré que 


\begin{theorem}\label{aaa}
 Si $\mu\geq\omega+1$ alors il existe une constante $C(n,\omega)>0$ telle que 
$$\bar\int_{S(r)}R\di\sigma_r=C(n,\omega)(-\Delta_g)^{\omega+1} R(P)r^{2\omega+2}+o(r^{2\omega+2})$$
$(-\Delta_g)^{\omega+1} R(P)$ est strictement négative et $I_g(u_{\e})<\frac{n(n-2)}{4}\omega_{n-1}^{2/n}$.\\ 
La fonction $u_\e$ est  définie plus bas (voir équation \eqref{uepsilon}). 
\end{theorem}
On rappellera le schéma de la preuve et on donnera des détails sur ce théorème dans l'appendice \ref{calculus}.\\

\textbf{On considèrera à partir de maintenant et jusqu'à la fin de cette section que} $\boldsymbol{\mu=\omega}$.\\ 

On sait que $\bar R$ est un polynôme homogène de degré  $\omega$, $\Delta_{\mathcal{E}}\bar R$ est donc homogène de degré  $\omega-2$ et
$$\Delta_{\mathcal{E}}\bar R=r^{-2}(\Delta_s\bar R-\omega(n+\omega-2) \bar R)$$
où $\Delta_{\mathcal{E}}$ est le Laplacien euclidien et   $\Delta_s$ est le Laplacien de la sphère  $S_{n-1}$, muni de la métrique $s$. $\Delta_{\mathcal{E}}^{k-1}\bar R$ est homogène de degré $\omega-2k+2$ et
$$\Delta^k_{\mathcal E}\bar R=r^{-2}(\Delta_s-\nu_k\mathrm{id})\Delta^{k-1}_{\mathcal E}\bar R= r^{-2k}\prod_{p=1}^k(\Delta_S-\nu_p\mathrm{id}) \bar R$$
avec
\begin{equation}\label{lam}
 \nu_k=(\omega-2k+2)(n+\omega-2k)
\end{equation}
Cette suite d'entiers naturels $(\nu_k)_{\{1\leq k\leq [\omega/2]\}}$ est décroissante. Elle est formée de valeurs propres du Laplacien sur la sphère $S_{n-1}$ (il est bien connu que les valeurs propres du Laplacien géométrique sont positives et qu'elles forment une suite croissante. Nos valeurs $\nu_k$ sont prises dans l'ordre opposé). \\
Puisque $\bar R$ est homogène de degré $\omega$, deux cas se présentent. Soit $\omega$ est pair, alors $\Delta^{[\omega/2]}_{\mathcal E}\bar R$ est une constante, mais d'après le \ref{scal item}ème point du lemme~\ref{lemme HV ex},  $\Delta^{[\omega/2]}_{\mathcal E}\bar R(P)=0$, d'où 
$$\Delta^{[\omega/2]}_{\mathcal E}\bar R=0$$
Soit $\omega$ est impair, alors $\Delta^{[\omega/2]}_{\mathcal E}\bar R$ est une forme linéaire. D'après le \ref{scal item}ème point du lemme~\ref{lemme HV ex}
$$\Delta^{[\omega/2]}_{\mathcal E}\bar R(P)=0\text{ et }\nabla\Delta^{[\omega/2]}_{\mathcal E}\bar R(P)=0$$
Finalement $\Delta^{[\omega/2]}_{\mathcal E}\bar R=0$ dans tous les cas.\\
On a $r^{-\omega}\bar R\in \bigoplus_{k=1}^{q}E_k $, où $E_k$ l'espace propre associé à la valeur propre  $\nu_k$, du  Laplacien $\Delta_s$, sur la sphère $S_{n-1}$, et où on a noté 
$$q=\min\{k\in\mathbb N/ \Delta_{\mathcal E}^{k}\bar R= 0\}$$
Si $j\neq k$, $E_k$ est bien orthogonal à $E_j$, pour le produit scalaire dans $L^2(S_{n-1})$ et le produit scalaire sur $H_1(S_{n-1})$ définis ci-dessous, puisque si $j\neq k$ et $\varphi_k\in E_k$
\begin{equation}\label{proscal1}
\nu_k(\varphi_k,\varphi_j)_{L^2}=(\Delta_s\varphi_k,\varphi_j)_{L^2}=(\varphi_k,\Delta_s\varphi_j)_{L^2}=\nu_j(\varphi_k,\varphi_j)_{L^2}
\end{equation}

Le produit  suivant est bien un produit scalaire sur l'ensemble des fonctions dans $H_1(S_{n-1})$, d'intégrales nulles
\begin{equation}\label{proscal2}
(\varphi_k,\varphi_j)_{H_1}=(\nabla\varphi_k,\nabla\varphi_j)_{L^2}=\nu_k(\varphi_k,\varphi_j)_{L^2}=\nu_j(\varphi_k,\varphi_j)_{L^2} 
\end{equation}

De plus, puisque $\int_{S(r)}\bar R d\sigma_r=0$ (d'après le lemme \ref{LLLL}), il existe des $\varphi_k\in E_k$ (fonctions propres de $\Delta_s$) telles que
\begin{equation}\label{dddd}
\bar R=r^{\omega}\Delta_s\sum_{k=1}^q\varphi_k=r^{\omega}\sum_{k=1}^q\nu_k\varphi_k
\end{equation}

On pose 
 $$b_{ij}=\sum_{k=1}^q\frac{1}{(n-2)(\nu_k+1-n)}[(n-1)\nabla_{ij}\varphi_k+\nu_k\varphi_ks_{ij}]$$
et   $a_{ij} =\bar g_{ij}-b_{ij}$.\\ 
On note $\bar R_a=\bar R$ lorsque $\bar g_{ij}=a_{ij}$ et $\bar R_b=\bar R$ lorsque $\bar g_{ij}=b_{ij}$. En tenant compte de l'expression \eqref{dddd} de $\bar R$, on établit les relations suivantes:
\begin{lemma}\label{formule abg}
 \begin{equation*}
\nabla^ib_{ij}=-\sum_{k=1}^q\nabla_j\varphi_k,\;\bar R=\bar R_b=\nabla^{ij}b_{ij}r^{\omega},\;\bar R_a=\nabla^{ij}a_{ij}r^{\omega}=0 \text{ et } s^{ij}b_{ij}=s^{ij}a_{ij}=0
\end{equation*}
\end{lemma}
La preuve détaillée est donnée dans l'appendice \ref{calculus}.\\

Regardons les deux cas particuliers suivants:\\ 
Si  $\bar g_{ij}=a_{ij}$. Alors $\bar R=\bar R_a=0$, ce qui entraîne que  la partie principale de $R_g$, est de degré $\mu\geq\omega+1$. Par le  théorème \ref{aaa}
$$\bar\int_{S(r)}R\di\sigma_r=\bar\int_{S(r)}R_a\di\sigma_r<0$$ 
Si $\bar g_{ij}=b_{ij}$. D'après le théorème \ref{zzz}, on a
$$\bar\int_{S(r)}\hspace{-0.5cm}R\di\sigma_r=\bar\int_{S(r)}\hspace{-0.5cm}R_b\di\sigma_r=[B_b/2-C_b/4-(1+\omega/2)^2Q_b]r^{2(\omega+1)}+o(r^{2(\omega+1)})$$
où l'on note par $B_b$, $C_b$ et $Q_b$ les intégrales $B$, $C$ et $Q$ respectivement, définies dans le  théorème \ref{zzz} lorsque $\bar g_{ij}=b_{ij}$. On peut les calculer en fonction  des fonctions propres $\varphi_k$, on trouve:
\begin{gather*}
Q_b=\bar\int_{S_{n-1}} b_{ij} b^{ij}\di\sigma=\frac{n-1}{n-2}\sum_{k=1}^{q}\frac{\nu_k}{\nu_k-n+1}\bar\int_{S_{n-1}}\varphi_k^2\di\sigma\label{Qb}\\
B_b=-(n-1)Q_b+\sum_{k=1}^{q}\nu_k\bar\int_{S_{n-1}}\hspace{-0.6cm}\varphi_k^2\di\sigma\label{Bb}\\ C_b=-(n-1)Q_b+\frac{n-1}{n-2}\sum_{k=1}^{q}\nu_k\bar\int_{S_{n-1}}\hspace{-0.6cm}\varphi_k^2\di\sigma\label{Cb}
\end{gather*}
Dans le calcul de ces expressions on a utilisé plusieurs fois l'identité $\nabla^ib_{ij}=-\sum_{k=1}^q\nabla_j\varphi_k$ et la formule de Stokes (les calculs détaillés sont donnés dans l'appendice \ref{calculus}).\\
 Dans le cas général (i.e. $\bar g_{ij}=a_{ij}+b_{ij}$), on obtient le lemme suivant:
 \begin{lemma}\label{mmm}
Si  $\mu=\omega$ et $\bar g_{ij}=a_{ij}+b_{ij}$, alors
\begin{equation*}
\bar\int_{S(r)}\hspace{-0.4cm}R\di\sigma_r=\bar\int_{S(r)}\hspace{-0.4cm}R_a+R_b\di\sigma_r+o(r^{2(\omega+1)})\leq [B_b/2-C_b/4-(1+\omega/2)^2Q_b]r^{2(\omega+1)}+o(r^{2(\omega+1)})
\end{equation*} 
\begin{equation*}
B_b/2-C_b/4-(1+\omega/2)^2Q_b=\sum_{k=1}^q u_k\bar\int_{S_{n-1}}\varphi_k^2\di\sigma
\end{equation*} 
avec
\begin{equation}\label{ukk}
 u_k=\biggl(\frac{n-3}{4(n-2)}-\frac{(n-1)^2+(n-1)(\omega+2)^2}{4(n-2)(\nu_k-n+1)}\biggr)\nu_k
\end{equation}
\end{lemma}

les nombres réels $u_k$ sont obtenus à partir des  expressions  $Q_b$, $B_b$ et $C_b$ ci-dessus (voir l'appendice  \ref{calculus} pour une preuve détaillée de ce lemme).  

\section{Généralisation d'un théorème de T.~Aubin}\label{GTA}

Dans son article  sur le problème de Yamabe, T.~Aubin \cite{Aub} a démontré que s'il existe un point $P_0\in M$ tel que $\omega(P_0)=0$ (voir la définition \ref{omega defintione} ), alors il existe une fonction  $\varphi_\e$ telle que 
$$I_g(\varphi_\e)<\frac{n(n-2)}{4}\omega_{n-1}^{2/n}$$ 
(cf  lemme \ref{lemme aub yam}). \\
Le but de cette section est de généraliser ce résultat pour tout $\omega\leq (n-6)/2$. \\
Soit $u_\e$ et $\varphi_\e$ deux fonctions définies par:
\begin{gather}\label{fonction test}
\varphi_{\e}(Q)=(1- r^{\alp} f(\xi))u_{\e}(Q)\\
u_{\e}(Q)=\begin{cases}\biggl(\displaystyle\frac{\varepsilon}{r^2+\varepsilon^2}\biggr)^{\frac{n-2}{2}}-\biggl(\frac{\varepsilon}{\delta^2+\varepsilon^2}\biggr)^{\frac{n-2}{2}} &\mbox{ si }Q\in B_{P}(\delta)\label{uepsilon}\\
\hspace{2cm}0 &\mbox{ si }Q\in M-B_{P}(\delta)
\end{cases}
\end{gather}
pour tout $Q\in M$, où $r=d(Q,P)$ est la distance entre $P$ et $Q$. $(r,\xi^j)$ sont les coordonnées  géodésiques  de $Q$ au voisinage de $P$ et 
$B_{P}(\delta)$ est une boule géodésique de centre $P$, de rayon $\delta$, fixé suffisamment petit. $f$ est une fonction qui dépend seulement de  $\xi$ telle que 
$\int_{S_{n-1}}f d\sigma=0$ et le choix précis sera décidé plus tard. \\

Soit

\begin{equation*}
I_a^b(\e)=\int_0^{\delta/\e}\frac{t^b}{(1+t^2)^a}dt\text{ et }I_a^b=\lim_{\e\to 0}I_a^b(\e)
\end{equation*}
alors $I_a^{2a-1}(\e)=\log\e^{-1}+O(1)$. Si $2a-b>1$ alors $I_a^b(\e)=I_a^b+O(\e^{2a-b-1})$  et par intégration par parties, on établit les relations suivantes  : 
\begin{equation}\label{rela}
 I_a^b=\frac{b-1}{2a-b-1}I_a^{b-2}=\frac{b-1}{2a-2}I_{a-1}^{b-2}=\frac{2a-b-3}{2a-2}I_{a-1}^{b},\quad \frac{4(n-2)I_n^{n+1}}{(I^{n-2}_n)^{(n-2)/n}}=n
\end{equation}
Rappelons que la fonctionnelle de  Yamabe $I_g$   (cf. \eqref{fonctionnel} page \pageref{fonctionnel}) est définie, pour tout $\psi\in H_1(M)$, par 
\begin{equation}\label{yamm}
I_g(\psi)=\biggr(\int_M|\nabla_g\psi|^2\di v+\frac{(n-2)}{4(n-1)}\int_MR_g\psi^2\di v\biggl)\|\psi\|_N^{-2} 
\end{equation}
où $N=2n/(n-2)$ et $\nabla_g$ est le gradient de la métrique $g$. \\
Voici donc le résultat principal de ce chapitre:
\begin{theorem}\label{theo}
Soit $(M,g)$ une variété riemannienne compacte de dimension $n$. Pour tout $P\in M$, si $\omega(P)\leq (n-6)/2$, alors il existe  $f\in C^\infty(S_{n-1})$, d'intégrale moyenne nulle et $\e>0$ telles que 
$$\mu(g)\leq I_g(\varphi_\e)< \frac{n(n-2)}{4}\omega_{n-1}^{2/n}$$
où $\varphi_\e$ est définie par \eqref{fonction test}.
\end{theorem}

\paragraph{\textbf{Remarque}}
L'hypothèse "$\omega$ est fini" affirme que la variété  $(M,g)$ n'est pas conformément difféomorphe à $(S_n,g_{can})$.

\begin{proof}
Soit $P\in M$. On écrit $\omega$ au lieu de $\omega (P)$.  Si $\mu\geq\omega+1$ alors l'inégalité est vraie par le théorème \ref{aaa}. On peut donc supposer que  $\mu=\omega$ jusqu'à la fin de la preuve. On commence par calculer la première intégrale dans la fonctionnelle \eqref{yamm}, avec $\psi=\varphi_\e$ et $f$ inconnue pour l'instant, en utilisant la formule
$$|\nabla_g\varphi_\e|^2=(\partial_r\varphi_\e)^2+r^{-2}|\nabla_s\varphi_\e|^2$$
On trouve :
\begin{multline}\label{nabla varphi}
\int_M|\nabla_g\varphi_\e|^2\di v=\int_M|\nabla_g u_\e|^2\di v+\int_0^\delta[\partial_r(r^{(\alp)} u_\e)]^2r^{n-1}\di r\int_{S_{n-1}}
\hspace{-0.5cm}f^2 \di\sigma +\\
\int_0^\delta u^2_\e r^{n+2\omega+1}\di r\int_{S_{n-1}}\hspace{-0.5cm} |\nabla f|^2 \di \sigma 
\end{multline}
Le changement de variable $t=r/\e$ donne
\begin{multline}\label{grad}
\int_M|\nabla_g\varphi_\e|^2\di v=(n-2)^2\omega_{n-1}I_n^{n+1}(\e)+\e^{2\omega+4}\biggl\{\int_{S_{n-1}}|\nabla f|^2 \di\sigma I_{n-2}^{2\omega+n+1}(\e)+\\
\int_{S_{n-1}}\hspace{-0.5cm}f^2\di\sigma [(\omega-n+4)^2I_n^{2\omega+n+5}(\e)+2(\alp)(\omega-n+4)I_n^{2\omega+n+3}(\e)+(\alp)^2I_n^{2\omega+n+1}(\e)]\biggr\}
\end{multline}
Pour la seconde intégrale qui contient la courbure scalaire $R_g$, on a
\begin{equation}\label{scal}
\begin{split}
\int_MR_g\varphi_\e^2\di v & =\int_MR_gu_\e^2\di v-2\int_Mfu_\e^2R_gr^{\omega+2} \di v+\int_Mf^2u_\e^2R_gr^{2\omega+4} \di v\\
& =\e^{2\omega+4}\omega_{n-1}\bar\int_{S(r)}r^{-2\omega-2}R_g\di \sigma_r I_{n-2}^{n+2\omega+1}(\e)-\\ 
& 2\e^{2\omega+4}I_{n-2}^{2\omega+n+1}(\e)\omega_{n-1}\bar\int_{S(r)}\hspace{-0.3cm}r^{-\omega}f\bar R \di\sigma_r+O(\e^{n-2})\\
\end{split}
\end{equation} 
où $\omega$ est l'ordre de la partie principale $\bar R$ (voir définition \ref{mu defin}). La  fonction $f$ est définie sur $S_{n-1}$. Sans aucune difficulté, on peut la redéfinir sur  $S(r)$, pour tout $r>0$, en posant $f(\xi/r)$, où $\xi\in S(r)$. On garde la même notation  pour cette redéfinition de $f$.\\
On calcule d'abord le développement limité de $\|\varphi_\e\|_N^{-2}$, on a:
$$\varphi_\e^N(Q)=\bigl[1-Nr^{\omega+2}f(\xi)+\frac{N(N-1)}{2}r^{2\omega+4}f^2(\xi)+O(r^{3\omega+6})\bigr]u_\e^N$$
En utilisant le fait que $\int_{S_{n-1}} f\di\sigma=0$,  on conclut que
\begin{equation*}
\begin{split}
\|\varphi_\e\|_N^N &=\int_0^\delta\int_{S_{n-1}}\hspace{-0.5cm}[1+\frac{N(N-1)}{2} r^{2(\alp)} f^2(\xi)+O( r^{3(\alp)})]r^{n-1}u^N_\e \di r\di\sigma(\xi)\\
& = \omega_{n-1}I^{n-1}_n+\frac{N(N-1)}{2}\e^{2(\alp)}\int_{S_{n-1}}\hspace{-0.5cm}f^2 \di\sigma I_n^{2\omega+n+3}+ o(\e^{2\omega+4})
\end{split} 
\end{equation*}
alors
\begin{multline}\label{norm}
\|\varphi_\e\|_N^{-2}=(\omega_{n-1}I^{n-1}_n)^{-2/N}\bigl\{1+\\
-(N-1)\e^{2(\alp)}\int_{S_{n-1}}\hspace{-0.5cm}f^2 \di\sigma I_n^{2\omega+n+3}/(\omega_{n-1}I^{n-1}_n)\bigr\}+o(\e^{2\omega+4})
\end{multline}

Par   \eqref{grad}, \eqref{scal}, \eqref{norm} et les relations \eqref{rela}, on trouve que (les détails de ces calculs sont dans l'appendice \ref{nablafi}):\\

 Si $n>2\omega+6$ alors :

\begin{multline*}
I_g(\varphi_\e)= \frac{n(n-2)}{4}\omega_{n-1}^{2/n}+(\omega_{n-1}I^{n-1}_n)^{-2/N}I_{n-2}^{n+2\omega+1}\e^{2\omega+4}\times \\
\biggl\{\frac{(n-2)\omega_{n-1}}{4(n-1)}\bar\int_{S(r)}r^{-2\omega-2}R_g\di \sigma_r -\frac{n-2}{2(n-1)}\int_{S_{n-1}}\hspace{-0.3cm}f\bar R \di\sigma+\int_{S_{n-1}}\hspace{-0.5cm}|\nabla f|^2 \di\sigma +\\
-\frac{n(n-2)^2-(\omega+2)^2(n^2+n+2)}{(n-1)(n-2)}\int_{S_{n-1}}\hspace{-0.5cm}f^2\di\sigma\biggr\}+
o(\e^{2\omega+4})
\end{multline*}

Si  $n=2\omega+6$ alors

\begin{multline*}
I_g(\varphi_\e)= \frac{n(n-2)}{4}\omega_{n-1}^{2/n}+(\omega_{n-1}I^{n-1}_n)^{-2/N}\e^{2\omega+4}\log\e^{-1}\times\\
\biggl\{\frac{(n-2)\omega_{n-1}}{4(n-1)}\bar\int_{S(r)}r^{-2\omega-2}R_g\di \sigma_r
-\frac{n-2}{2(n-1)}\int_{S_{n-1}}\hspace{-0.3cm}f\bar R \di\sigma+\\
\int_{S_{n-1}}|\nabla f|^2 \di\sigma +(\omega+2)^2\int_{S_{n-1}}\hspace{-0.5cm}f^2\di\sigma\biggr\}+O(\e^{2\omega+4})
\end{multline*}

On considère maintenant la fonctionnelle $I_S$, définie sur la  sphère $S_{n-1}$, pour les fonctions dans $H_1(S_{n-1})$, d'intégrale  moyenne nulle, par
\begin{multline*}
I_S(f)=\bar\int_{S_{n-1}} \hspace{-0.3cm}4(n-1)(n-2)|\nabla f|^2-[4n(n-2)^2-4(\omega+2)^2(n^2+n+2)]f^2+\\
-2(n-2)^2f\bar R\di\sigma
\end{multline*}
Alors si $n>2\omega+6$ 
\begin{multline}\label{estime phi1}
I_g(\varphi_\e)=\frac{n(n-2)}{4}\omega_{n-1}^{2/n}+\frac{\omega_{n-1}^{2/n}I_{n-2}^{n+2\omega+1}\e^{2\omega+4}} {4(n-1)(n-2)(I^{n-1}_n)^{2/N}}\times\\
\{(n-2)^2\bar\int_{S(r)}r^{-2\omega-2}R_g\di\sigma_r+I_S(f)\}+o(\e^{2\omega+4}) 
\end{multline}
et si $n=2\omega+6$ 
\begin{multline}\label{estime phi2}
I_g(\varphi_\e)=\frac{n(n-2)}{4}\omega_{n-1}^{2/n}+ 
\frac{\omega_{n-1}^{2/n}I_{n-2}^{n+2\omega+1}\e^{2\omega+4}\log\e^{-1}} {4(n-1)(n-2)(I^{n-1}_n)^{2/N}}\times\\
\{(n-2)^2\bar\int_{S(r)}r^{-2\omega-2}R_g\di\sigma_r+I_S(f)\}+O(\e^{2\omega+4}) 
\end{multline}

Remarquons que si $k\neq j$ alors $I_S(\varphi_k+\varphi_j)=I_S(\varphi_k)+I_S(\varphi_j)$. En effet, $\varphi_k$ et $\varphi_j$ sont orthogonales pour le produit scalaire sur $H_1(S_{n-1})$. D'où
\begin{equation*}
\begin{split} 
I_S(c_k\nu_k\varphi_k) &=\bigl\{d_kc_k^2-2(n-2)^2c_k \bigr\}\nu_k^2\bar\int_{S_{n-1}}\varphi_k^2\di\sigma\\
 & =-\frac{(n-2)^4}{d_k}\nu_k^2\bar\int_{S_{n-1}}\varphi_k^2\di\sigma
\end{split}
\end{equation*}  
avec
\begin{gather}
 d_k=4[(n-1)(n-2)\nu_k-n(n-2)^2+(\omega+2)^2(n^2+n+2)]\label{dkk}\\
\text{et }c_k=\frac{(n-2)^2}{d_k}
\end{gather}
Ici on choisit les $c_k$ de sorte que  $I_S(c_k\nu_k\varphi_k)$ soit minimal. En utilisant \eqref{lam}, on peut vérifier aisément que les $d_k$ sont strictement positifs pour tout  $1\leq k\leq \omega/2$. \\ 
Maintenant,  On pose  
\begin{equation}\label{def de f}
f=\sum_{1}^qc_k\nu_k\varphi_k
\end{equation}
Il est clair que $f$ ainsi définie est d'intégrale nulle sur $S_{n-1}$. C'est bien la définition de $f$ qu'on utilisera dans la suite de la preuve. Par l'orthogonalité des fonctions $\varphi_k$, on trouve que 
$$I_S(f)=-\sum_{1}^q\frac{(n-2)^4}{d_k}\nu_k^2\bar\int_{S_{n-1}}\varphi_k^2\di\sigma$$ 
et par  le  lemme \ref{mmm}, on trouve l'inégalité suivante:

$$(n-2)^2\bar\int_{S(r)}\hspace{-0.5cm}r^{-2\omega-2}R_g\di\sigma_r+I_S(f)\leq\sum_{1}^q(u_k(n-2)^2-\frac{(n-2)^4}{d_k}\nu_k^2)\bar\int_{S_{n-1}}\hspace{-0.5cm}\varphi_k^2\di\sigma+o(1)$$
Le lemme ci-dessous énoncé, assure que le membre de droite de cette dernière inégalité est strictement négatif. En utilisant les inégalités \eqref{estime phi1}, \eqref{estime phi2}, on en déduit que  
$I_g(\varphi_\e)<\frac{n(n-2)}{4}\omega_{n-1}^{2/n}$
\end{proof}

\begin{lemma}\label{lemme poly}
Pour tout $k\leq q\leq [\omega/2]$, l'inégalité suivante est toujours vraie
$$ u_k-\frac{(n-2)^2}{d_k}\nu_k^2< 0$$
\end{lemma}

\begin{proof}

On rappelle l'expression des $\nu_k$ donnée dans \eqref{lam}:
\begin{equation*}
 \nu_k=(\omega-2k+2)(n+\omega-2k)
\end{equation*}
Pour tout $k\in \intl 1,\omega/2\intr$, on définit les nombres $(U_k)$  par
$$U_k:=(\nu_k-n+1)d_k\{(n-2)\frac{u_k}{\nu_k}-\frac{(n-2)^3}{d_k}\nu_k\}$$ 
On remarque que l'expression de $U_k$ est polynomiale, décroissante en $\nu_k$ quand $\nu_k\geq 0$.  $U_k=P(\nu_k)$, où $P$ est le polynôme défini par 
\begin{multline*}
 P(x)=[(n-1)(n-2)x-n(n-2)^2+(\omega+2)^2(n^2+n+2)]\times\\
[(n-3)(x-n+1)-(n-1)^2-(n-1)(\omega+2)^2]-(n-2)^3(x^2-(n-1)x)
\end{multline*}
Le polynôme dérivé  est 
\begin{equation*}
 P'(x)=-2(n-2)x-2n(n-2)^3+2(n^2-3n-2)(\omega+2)^2
\end{equation*}
Par hypothèse $\omega+2\leq (n-2)/2$, donc $P$ est décroissant sur $\mathbb R_+$. Ce qui entraîne  que 
$$U_k=P(\nu_k)\leq P(\nu_{\omega/2})=U_{\omega/2}$$ 
pour tout $1\leq k\leq \omega/2$. Il est facile de vérifier que  $u_{\omega/2}$  est strictement négatif et donc $U_k\leq U_{\omega/2}<0$.
\end{proof}

\chapter{Autour de la conjecture de Hebey--Vaugon}\label{CHYSr}

Dans la section \ref{Equi Yam pro}, on a étudié le problème de Yamabe équivariant, considéré par E.~Hebey et  M.~Vaugon \cite{HV}, lorsque la métrique n'est pas nécessairement $C^\infty$. On a démontré que la condition suffisante pour résoudre ce problème est que la conjecture \ref{HVcon} soit vraie (cf. théorème \ref{HVGINVD}). Malheureusement, on ne peut pas donner une preuve de cette conjecture, lorsque $g\in H^p_2(M,T^*M\otimes T^*M)$. En effet, la courbure scalaire appartient à $L^p$, et plusieurs arguments utilisés dans le cas $C^\infty$ ne sont plus valables dans ce cas.\\

Dans tout ce chapitre, on suppose que $M$ est une variété compacte $C^\infty$, de dimension $n\geq 3$, $g$ est une métrique riemannienne $C^\infty$, munie de sa connexion riemannienne, notée $\nabla_g$. On note par $I(M,g)$, $C(M,g)$ le groupe d'isométries et le groupe des transformations conformes respectivement (voir la définition dans la section \ref{iso conf}). Soit $G$ un sous groupe du groupe d'isométries $I(M,g)$. \\
Ce chapitre utilise beaucoup de résultats déjà démontrés dans le chapitre précédent.

\section{La conjecture de Hebey--Vaugon }\label{HVCONJ}

\begin{conjecture}[E.~Hebey et  M.~Vaugon \cite{HV}]\label{HVcon}
Soit $G$ un sous groupe d'isométries de $I(M,g)$. Si $(M,g)$ n'est pas conformément difféomorphe à $(S_n, g_{can})$ ou bien si $G$ n'a pas de point fixe, alors l'inégalité stricte suivante a toujours lieu 
\begin{equation}\label{HVI}
\inf_{g'\in [g]^G} J(g')<n(n-1)\omega_n^{2/n}(\inf_{Q\in M}\Ca O_G(Q))^{2/n}
\end{equation}
\end{conjecture}

\paragraph{\textsc{Remarques}}\label{remarques}
\begin{itemize} 
\item Cette conjecture est la généralisation de la conjecture de T.~Aubin \ref{Aubincon} pour le problème de Yamabe, qui correspond à $G=\{\mathrm{id}\}$. Dans ce cas, la conjecture est complètement prouvée. Elle est prouvée aussi dans le cas où la métrique satisfait l'hypothèse $(H)$, définie dans la section \ref{hypG} (voir théorème \ref{conj aub}).
\item Cette inégalité est triviale si $\inf_{g'\in [g]^G} J(g')$ est négatif. 
\item Si pour tout $Q\in M$,  $\Ca O_G(Q)=+\infty$, alors la conjecture est vérifiée trivialement. 
\end{itemize}

\medskip
Rappelons que la partie principale de la courbure scalaire $\bar R$ est définie dans la section \ref{aubin computations} (voir définition \ref{mu defin}).\\

Les résultats principaux de ce chapitre sont  

\begin{theorem}\label{contheo}
La conjecture \ref{HVcon} est vraie, s'il existe un point $P$ d'orbite minimale (finie) pour lequel $\omega(P)\leq 15$, ou si au voisinage de $P$, $\mathrm{deg}\bar R\geq \omega(P)+1$
\end{theorem}

\begin{corollary}\label{cor prin}
La conjecture \ref{HVcon} est vraie si $M$ est de dimension $n\in \intl 3, 37 \intr$.
\end{corollary} 

\begin{proof}
Supposons que $P$ est un point d'orbite minimale (finie) sous $G$ (sinon la conjecture est  trivialement vérifiée). \\
Si $\omega(P)> (n-6)/2$, on conclut par le troisième point du théorème \ref{HV theorem} ci-dessous.\\
Si $\omega(P)\leq [(n-6)/2] \leq 15$, on conclut par le théorème \ref{contheo}.
\end{proof}

\section{Les travaux de Hebey--Vaugon}\label{hv workss}

E.~Hebey et  M.~Vaugon~\cite{HV} ont prouvé la conjecture \ref{HVcon} dans les cas suivants:
\begin{theorem}\label{HV theorem}
Soit $(M,g)$ une variété riemannienne compacte, de dimension $n\geq 3$ et $G$ un sous groupe d'isométries du groupe $I(M,g)$.
On a toujours:  
$$\inf_{g'\in [g]^G} J(g')\leq n(n-1)\omega_n^{2/n}(\inf_{Q\in M}\Ca O_G(Q))^{2/n}$$
et l' inégalité  stricte \eqref{HVI} est au moins vérifiée dans chacun des cas suivants:
\begin{enumerate}
 \item $G$ opère librement su $M$
 \item $3\leq \dim M\leq 11$
 \item\label{item 3} Il existe un point $P$ d'orbite minimale (finie) sous $G$, pour lequel soit $\omega(P)>(n-6)/2$ soit $\omega(P)\in \{0,1,2\}$.

\end{enumerate}
\end{theorem}
\begin{proof}[\textbf{Idées de la preuve}]
On s'intéresse à  la démonstration du point \ref{item 3} du théorème ci-dessus (c'est le cas qui manque dans le théorème \ref{theo}). Les hypothèses sont: 
\begin{enumerate}
 \item $ \Ca O_G(P)<+\infty$.
 \item Il existe $P\in M$ tel que $\Ca O_G(P)=\inf_{Q\in M}\mathrm{card} O_G(Q)$.
\item $\omega>[\frac{n-6}{2}]\Longleftrightarrow\forall \beta\in \intl 1, n\intr^i/i\leq [(n-6)/2],\quad \nabla^\beta W_g(P)=0$.
\end{enumerate}
Notons  $k=\Ca O_G(P)$, le cardinal de l'orbite $O_G(P)=\{P_i,1\leq i\leq k\}$, où l'on a posé $P_1=P$. La troisième hypothèse implique que pour tout $1\leq i\leq k$, $\omega(P_i)>[\frac{n-6}{2}]$, puisque le tenseur de Weyl est invariant sous l'action du groupe d'isométries $I(M,g)$.\\
Par les travaux de J.M.~Lee et T.~Parker \cite{LP}, on sait qu'on peut trouver un système de coordonnées et une métrique conforme $g'$ tels que  $g'$ satisfait: 
\begin{equation}\label{det dl}
\det(g')=1+O(r^m)\quad\text{ pour tout }m\gg 1
\end{equation}
(cf. section \ref{metric cg} ou \cite{LP} pour l'existence). Dans le cas équivariant, on ne peut pas considérer n'importe quelle métrique dans la classe conforme $[g]$, cependant E.~Hebey et  M.~Vaugon ont démontré que dans chaque classe $[g]^G$ on peut trouver au moins une métrique qui satisfait \eqref{det dl}. En utilisant les champs de Jacobi, ils ont obtenu le développement limité de la métrique $g$ suivant: 

\begin{lemma}
\begin{multline}\label{expansion metric}
 g_{ij}(Q)=\delta_{ij}+\hspace{-0.4cm}\sum_{\omega+4\leq m \leq 2\omega+5}\hspace{-0.4cm}C_m\nabla_{p_3\cdots p_{m-2}}R_{ip_1p_2j}(P)x^{p_1}\cdots x^{p_{m-2}}\\
+ C_\omega\sum_{pj}\nabla_{p_3\cdots p_{2\omega+4}}R_{ip_1p_2j}(P)x^{p_1}\cdots x^{p_{2\omega+4}}\\
+C'_\omega\sum_{q=1}^n\sum_{pj}(\nabla_{p_3\cdots p_{\omega+2}}R_{ip_1p_2q}(P))(\nabla_{p_{\omega+5}\cdots p_{2\omega+4}}R_{jp_{\omega+3}p_{\omega+4}q}(P))x^{p_1}\cdots x^{p_{2\omega+4}}+O(r^{2\omega+5})
\end{multline}

pour tout  $Q$ au voisinage de $P$, où $\{x^l\}$ sont les coordonnées locales de $Q$. $C_\omega$, $C'_{\omega}$ et $C_m$ sont des nombres réels, qui dépendent de $\omega$ et $m$ respectivement. Ces nombres sont donnés explicitement dans \cite{HV}. 
\end{lemma}
Ce développement est le point crucial dans la preuve du lemme suivant:
\begin{lemma}\label{lemme HV ex}
Dans chaque classe $[g]^G$, des métriques conformes $G-$invariantes, on peut trouver une métrique $g'$ qui satisfait
\begin{enumerate}
\item\label{det met} $\det(g')=1+O(r^m)$, $m\gg 1$
\item\label{riem na} $\forall i<\omega,\; \nabla^iR'_{jklm}(P)=0$
\item\label{ric scal} pour tout $\beta\in\intl 1, n\intr^i$ tel que $i\leq 2\omega+1$
$$\nabla_\beta R'(P)=\partial_\beta R'(P),\;\nabla_\beta Ric'(P)=\partial_\beta Ric'(P),\; \nabla_\beta R_{g'}(P)=\partial_\beta R_{g'}(P)$$

\item\label{scal item} $\forall j\leq \omega\quad \Delta_{g}^jR_{g'}(P)=0\text{ et }\nabla\Delta_{g'}^{\omega'} R_{g'}(P)=0$
\end{enumerate}
où $R'$, $Ric'$ et $R_{g'}$  sont le tenseur de courbure de Riemann, le tenseur de Ricci et la courbure scalaire de $g'$ respectivement.
\end{lemma}
\paragraph{\textbf{Remarque}}Dans leur article  \cite{HV}, E.~Hebey et  M.~Vaugon ont noté par $Sym_\beta T_\beta$ le symetrisé du tenseur $T$, et par $C(2,2)$ l'application de contraction des indices deux à deux pour les tenseurs symétrique. A titre d'exemple $C(2,2)T_{ij}=\sum_iT_{ii}$, $(C(2,2)T_{ijk})_l=\sum_i T_{iil}$ et $C(2,2)T_{ijkl}=\sum_{i,j}T_{iijj}$. Ils ont montré que pour tout $\beta\in\intl 1, n\intr^i$ tel que $i\leq 2\omega+1$ 
$$C(2,2)(Sym_\beta\nabla_\beta R_g(P))=0$$
ce qui est équivalent au point \ref{scal item} du lemme ci-dessus.\\

L'invariance $G-$conforme de $\mu_G(g)$ et de $\omega$ (cf. propriétés \ref{omega invariant}, \ref{mu invariance}) nous permettent de considérer n'importe quelle  métrique, $G-$invariante  dans la classe $[g]^G$ (cf. définition \ref{def gG}, page \pageref{def gG}). Sans perte de généralités, on suppose que la métrique $g$ et les courbures associées à $g$, satisfont le lemme \ref{lemme HV ex}. Soit $G_{P_i}$ la fonction de Green du Laplacien conforme $L_g$ au point $P_i$ (voir la section \ref{glc} pour l'existence). En utilisant les points \ref{det met} et \ref{scal item} du lemme \ref{lemme HV ex}, on montre que le développement limité de la fonction $G_{P_i}$ au voisinage de $P_i$ est

\begin{equation*}
 G_{P_i}(x)= \frac{1}{(n-2)\omega_{n-1}r_i^{n-2}}(1+\sum_{p=1}^n\psi_p(x))+O''(1)
\end{equation*}

où $r_i=d(P_i,x)$ et les $\psi_p$ sont des polynômes homogènes de degré $p$ qui s'annulent si $1\leq p\leq [(n-2)/2]$.\\
Considérons la métrique $\gt=G_P^{\N}g$. $G_P$ est $C^\infty$ sur $M-\{P\}$ et la variété $(M-\{P\},\gt)$ est asymptotiquement plate d'ordre $\frac{n}{2}$. Les coordonnées  asymptotiques sont $z^i=\frac{x^i}{|x|^2}$ et $\rho=|z|$, où $\{x^i\}$ est un système de coordonnées normal en $P$. La masse $m(\gt)$ est bien définie positive car $\tau=\frac{n}{2}>\frac{n-2}{2}$.   Soit  $\mathcal G=\sum_{i=1}^kG_{P_i}$  une fonction $C^\infty$, $G-$invariante, définie sur $M-O_G(P)$. La fonction test utilisée par E.~Hebey et  M.~Vaugon pour démontrer la conjecture est $w_\e$, définie comme suit:
\begin{equation*}
 w_{i,\e}=\begin{cases}
 \mathcal G r_i^{n-2}\biggl(\displaystyle\frac{\varepsilon}{r_i^2+\varepsilon^2}\biggr)^{\frac{n-2}{2}}&\text{ si }r_i\leq \delta\\   
\mathcal G \delta^{n-2}\biggl(\displaystyle\frac{\varepsilon}{\delta^2+\varepsilon^2}\biggr)^{\frac{n-2}{2}}&\text{ si }r_i\geq \delta
      \end{cases}
\end{equation*}

\begin{equation*}
 w_\e=\sum_{i=1}^kw_{i,\e}
\end{equation*}
Si $\delta$ est suffisamment petit, alors les fonctions $w_{i,\e}$ et $w_\e$ sont bien définies sur $M$. Il est clair que la fonction $w_\e$ est $G-$invariante. Après calculs, E.~Hebey et  M.~Vaugon obtiennent l'inégalité suivante:

\begin{equation*}
 E(w_\e)\leq \frac{n(n-1)}{4}\omega_n^{2/n}k^{2/n}\|w_\e\|_N^{-2}-C_1( m(\tilde g)+(n-2)K)\e^{n-2}+\e^{n-2}O(\delta)+O(\e^{n-1})
\end{equation*}
où $C_1$ et $K$ deux constantes positives. Alors $m(\tilde g)+(n-2)K>0$, et on peut choisir $\delta$ et $\e$ suffisamment petits tels que $I_g(w_\e)<\frac{n(n-1)}{4}\omega_n^{2/n}k^{2/n}$. Par conséquent 
$$\mu_G(g)<\frac{n(n-2)}{4}\omega_n^{2/n}(\Ca O_G(P))^{2/n}$$

\end{proof}

\section{Preuve du théorème principal}\label{proof hv}

En tenant compte des remarques de la section \ref{HVCONJ} (cf. page \pageref{remarques}) et du théorème \ref{HV theorem}, on considère seulement le cas où $\inf_{Q\in M}\Ca O_G(Q)$ est fini, strictement positif (i.e. $\mu_G(g)>0$) et $\omega\leq(n-6)/2$  . Alors il existe $P\in M$ tel que
$$O_G(P)=\{P_i\}_{1\leq i\leq m},\;\; m=\Ca O_G(P)=\inf_{Q\in M}\Ca O_G(Q)\text{ et }P_1=P$$

Un élément très important, dans la démonstration du théorème principal \ref{contheo}, est le choix des fonctions test dans la fonctionnelle $I_g$. Les fonctions test précédemment utilisées par T.~Aubin et R.~Schoen (voir la   preuve du théorème \ref{conj aub}) ne fonctionnent pas ici, comme cela avait été remarqué par E.~Hebey et M.~Vaugon \cite{HV}. Les "bonnes" fonctions test seront construites de la manière suivante, en modifiant les fonctions test de T.~Aubin: on construit une fonction test $G-$invariante, à partir des fonctions $\tilde\varphi_{\e,i}$, définie de la même façon que $\varphi_\e$ (voir section \ref{GTA}), dont on rappelle la définition. $P$ est un point d'orbite minimale. Pour tout $Q\in M$

\begin{gather}\label{fonction test5}
\tilde\varphi_{\e,i}(Q)=(1- r_i^{\omega+2} \tilde f_i(Q))u_{\e,i}(Q)\\
u_{\e,i}(Q)=\begin{cases}\biggl(\displaystyle\frac{\varepsilon}{r_i^2+\varepsilon^2}\biggr)^{\frac{n-2}{2}}-\biggl(\frac{\varepsilon}{\delta^2+\varepsilon^2}\biggr)^{\frac{n-2}{2}} &\mbox{ si }Q\in B_{P_i}(\delta)\label{uepsilon5}\\
\hspace{2cm}0 &\mbox{ si }Q\in M-B_{P_i}(\delta)
\end{cases}
\end{gather}
où $r_i=d(Q,P_i)$ est la distance entre $P_i$ et $Q$. Pour la simplicité: $P=P_1$, $r=r_1$, $\tilde\varphi_\e=\tilde\varphi_{\e,1}$, $\tilde f=\tilde f_1$ et $u_{\e,1}=u_\e$. $B_{P}(\delta)$ est une boule géodésique de centre $P$, de rayon $\delta$, fixé suffisamment petit. Les $\tilde f_i$ sont définies de la façon suivante : Soit $\exp_{P_i}$ l'application exponentielle, définie de $B(\delta)$, la boule euclidienne centrée en 0 et de rayon $\delta$, dans $B_{P_i}(\delta)$. Pour tout $Q\in B_{P_i}(\delta)$, on pose 
\begin{equation}\label{fitil}
 \tilde f_i(Q)=c r_i^{-\omega}\nabla_g^\omega R_{(P_i)}(\exp_{P_i}^{-1}Q,\cdots,\exp_{P_i}^{-1}Q)
\end{equation}   
où  $\omega=\omega(P)$ et $\nabla^\omega_gR(P)$ est la $\omega-$ème dérivée covariante de $R_g$ au point $P$, c'est  un tenseur $\omega$ fois covariant.  Dans le système de coordonnées géodésiques $\{r,\xi^j\}$, centré en $P$, induit par  l'application $\exp_P$, $\tilde f$ s'écrit: 

\begin{equation*}
\tilde f=c r^{-\omega}\bar R=c\sum_{k=1}^q\nu_k\varphi_k
\end{equation*}
où $\bar R$, $\varphi_k$ et $\nu_k$ sont définis dans la section \ref{aubin computations} (page \pageref{aubin computations}). La fonction $\tilde f$ est définie sur la sphère $S_{n-1}$. Le choix de la constante $c$ est très important dans le lemme suivant. 

\begin{lemma}\label{lemme infg}
Supposons que $\omega\leq (n-6)/2$. Si $\omega\in\intl 3,15\intr$ ou si $\mathrm{deg}\bar R\geq\omega+1$ alors il existe $c\in \mathbb R$ telle que, pour la fonction $\tilde\varphi_\e$ correspondante, on a:
\begin{equation}\label{inegde}
I_g(\tilde\varphi_{\e})<\frac{1}{4} n(n-2)\omega_{n}^{2/n}
\end{equation} 
\end{lemma}

\paragraph{Remarque}
\begin{enumerate} 
\item Dans le chapitre précédent, on a démontré que l'inégalité de ce lemme est vérifiée, pour tout   $\omega\leq (n-6)/2$, pour une fonction test $\varphi_\e$ (voir théorème \ref{theo}).  On remarque que la seule différence entre les définitions de $\varphi_\e$ et $\tilde\varphi_\e$ est  dans la construction des fonctions $f$ et $\tilde f$. En effet, $\tilde f$ est définie à l'aide d'une constante globale $c$ et $f$ à l'aide  des constantes $c_k$ qui changent avec les fonctions propres $\varphi_k$. On verra dans la preuve du théorème \ref{contheo}, qu'à partir de $\tilde\varphi_\e$, on peut construire une fonction $G-$invariante qui possède les "bonnes" propriétés, cette chose n'est pas possible avec les fonctions $\varphi_\e$.  
\item Pour $\omega=16$ et $n$ suffisamment grand, on  peut vérifier qu'il n'existe pas une valeur de $c$ pour laquelle l'inégalité \eqref{inegde} est vraie. 
\end{enumerate}

\begin{proof}
1. Si $\mathrm{deg}\bar R\geq \omega+1 $, alors d'après le théorème   \ref{aaa}
$$I_g(u_{\e,1})<\frac{n(n-2)}{4}\omega_{n}^{2/n}$$
où $u_{\e,1}=u_\e$ est définie par \eqref{uepsilon5}.
Il suffit donc de prendre $c=0$ et $\tilde\varphi_\e=u_\e$.\\

2. Si $\mathrm{deg}\bar R= \omega $. D'après les estimées données dans la preuve du théorème \ref{theo} (voir page \pageref{estime phi1}), il suffit de montrer qu'il existe $c\in\mathbb R$ telle que
\begin{equation}\label{gfgf}
I_S(\tilde f)+(n-2)^2\bar\int_{S(r)}r^{-2\omega-2}R_g\di\sigma_r<0
\end{equation}

Cherchons donc cette constante $c$. On garde les notations de la preuve du théorème \ref{theo}. On a
\begin{equation*}
I_S(\tilde f)=\sum_{k=1}^q I_S(c\nu_k\varphi_k)=\bigl\{d_kc^2-2(n-2)^2c \bigr\}\nu_k^2\bar\int_{S_{n-1}}\varphi_k^2\di\sigma
\end{equation*}
\begin{equation*}
\text{et }\bar\int_{S(r)}r^{-2\omega-2}R_g\di\sigma_r=\sum_{k=1}^q u_k\bar\int_{S_{n-1}}\varphi_k^2\di\sigma
\end{equation*}
Pour montrer l'inégalité \ref{gfgf}, il suffit de montrer que
\begin{equation}\label{ineggg}
\forall k\leq q\leq [\omega/2]\quad \frac{d_k}{2(n-2)}c^2-(n-2)c+(n-2)\frac{u_k}{2\nu_k^2}<0
\end{equation} 
On a donc un trinôme du second degré en $c$, son discriminant est
\begin{equation*}
\Delta_k=(n-2)^2-\frac{d_ku_k}{\nu_k^2}
\end{equation*}
D'après le lemme \ref{lemme poly}, $\Delta_k>0$ pour tout $k\leq q\leq [\omega/2]$. Par conséquent, le trinôme ci-dessus admet deux racines, notées $x_k<y_k$ et données par
\begin{equation*}
x_k=\frac{(n-2)^2-(n-2)\sqrt{\Delta_k}}{d_k},\qquad y_k=\frac{(n-2)^2+(n-2)\sqrt{\Delta_k}}{d_k}
\end{equation*}
L'inégalité \eqref{ineggg} est vérifiée si et seulement si 
\begin{equation}\label{propo}
\bigcap_{k=1}^q]x_k,y_k[\neq \varnothing
\end{equation}
Le lemme est donc démontré, si l'intersection ci-dessus n'est pas vide dans les cas énoncés.\\ 
Puisque $(d_k)_k$ est décroissante, il est facile de vérifier que 
\begin{equation}\label{ine solut}
 \forall k< j\leq [\frac{\omega}{2}]\qquad x_k< y_j
\end{equation}
 (voir équations \eqref{lam}, \eqref{dkk}, pour la définition de $\nu_k$ et $d_k$). On vérifie aussi que $u_{\omega/2}<0$ (voir équation \eqref{ukk}), cela entraîne que si $\omega$ est pair alors $x_{\omega/2}<0$.
\begin{itemize}
\item[$i.$] Si $\omega=3$ alors $k=q=1$, l'intersection ci-dessus est donc non vide. Il suffit de prendre $c=(x_1+y_2)/2$. 
\item[$ii.$] Si $\omega=4$ alors $k\in\{1,2\}$, $x_2<0$ (car $u_2<0$) et $0<x_1<y_2$. L'intersection $]x_1,y_1[\cap]x_2,y_2[\neq \varnothing$. Ce qui entraîne l'inégalité \eqref{inegde}.
\item[$iii.$]Si $\omega=5$  alors $k\in\{1,2\}$. Par des calculs directs, on montre que $x_2<y_1$ (voir les détails dans l'appendice \ref{ch5}). Puisque $y_2>x_1$, l'intersection des deux intervalles n'est pas vide. 
\item[$iv.$]Si $\omega=6$  alors $k\in\{1,2,3\}$ et il est immédiat de voir que $x_3<0$ (car $u_3<0$), $y_3>x_2>0$ et $y_3>x_1>0$. Par des calculs directs, on montre que $x_2<y_1$ (voir les détails dans l'appendice \ref{ch5}). Ce qui entraîne que l'intersection
\begin{equation}\label{propo}
\bigcap_{k=1}^3]x_k,y_k[
\end{equation}
est non vide.
\item[$iv.$]Si $\omega=7$  alors $k\in\{1,2,3\}$. Il y a trois intervalles. Par des calculs directs, on montre que pour tout  $3\geq j>k\geq 1$, $y_k>x_j$ (voir appendice \ref{ch5}). Puisque $y_j>x_k$ pour tout $3\geq j>k\geq 1$ (voir inégalité \eqref{ine solut}), l'intersection des trois intervalles n'est donc pas vide. 
\item En se servant du logiciel "Maple", on montre que le lemme reste vrai jusqu'à $\omega=15$ (voir appendice \ref{ch5} pour plus de détails).
\end{itemize}
\end{proof}

\begin{proof}[\textbf{Fin de la preuve du  théorème \ref{contheo}}]
Sans perte de généralités, on suppose que $3\leq \omega\leq (n-6)/2$, car si $\omega>(n-6)/2$ ou si $\omega\leq 2$, il suffit d'appliquer le théorème \ref{HV theorem}. L'orbite de $P$ sous l'action de $G$ est supposée être  de cardinal fini et minimal (i.e. $\Ca O_G(P)=\inf_{Q\in M}\Ca O_G(Q)$). À partir de la fonction  $\tilde\varphi_\e$, définie au début de la section \ref{proof hv}, on définit  la fonction  $\phi_\e$ comme suit: 
$$\phi_\e=\sum_{k=1}^m\tilde\varphi_{\e,i}$$
$\phi_\e$ est $G-$invariante. En effet, pour tout $\sigma \in G$, si $\sigma(P_i)=P_j$ alors
$$u_{\e,i}=u_{\e,j}\circ\sigma$$ 
d'après la définition de $\tilde f_i$, donnée par \eqref{fitil}, $\tilde f_i=\tilde f_j\circ \sigma$ et donc $$\tilde\varphi_{\e,i}=\tilde\varphi_{\e,j}\circ\sigma$$
Le support de la fonction $\tilde\varphi_\e$ est inclus dans la boule $B_P(\delta)$. On  choisit $\delta$ suffisamment petit tel que  pour tout  $i\in\intl 2,m\intr $, l'intersection $B_P(\delta)\cap B_{P_i}(\delta)=\varnothing$. Donc
$$E(\phi_\e)=(\Ca O_G(P))E(\varphi_\e)\text{ et }\|\phi_\e\|_N^N=(\Ca O_G(P))\|\varphi_\e\|_N^N$$
alors
$$I_g(\phi_\e)=(\mathrm{card}O_G(P))^{2/n}I_g(\varphi_\e)$$
Par le lemme \ref{lemme infg}, on en déduit que
$$I_g(\phi_\e)<\frac{n(n-2)}{4}\omega_{n-1}^{2/n}(\Ca O_G(P))^{2/n}$$
Il nous reste à remarquer que si $\tilde g=\phi_\e^{4/(n-2)}g$ alors
\begin{equation}\label{rapport J I}
 J(\tilde g)=4\frac{n-1}{n-2}I_g(\phi_\e)
\end{equation}
cette relation est déjà établie dans la preuve des propriétés  \ref{mu invariance} ( voir page \pageref{mu invariance}). On conclut que  
$$J(\tilde g)<n(n-1)\omega_{n-1}^{2/n}(\Ca O_G(P))^{2/n}$$
où $\e$ est choisi suffisamment petit par rapport à $\delta$.
\end{proof}

\appendix

\chapter{Détails des calculs (Chapitre \ref{CHAub})}\label{calculus}
\markboth{D\'ETAILS DES CALCULS}{D\'ETAILS DES CALCULS}

\section*{Preuve du théorème \ref{zzz}}\label{lemme diff}

On reprend les notations et les définitions de la section \ref{aubin computations}. Voici l'énoncé du théorème que l'on démontre dans cette section:

\begin{theorem}\label{rtzzz}
$$\bar R=\nabla^{ij}\bar g_{ij}r^\omega\quad\text{et}$$
$$\bar\int_{S(r)}R\di\sigma_r =[B/2-C/4-(1+\omega/2)^2Q]r^{2(\omega+1)}+o(r^{2(\omega+1)})$$
où $B=\bar\int_{S_{n-1}}\hspace{-0.5cm}\nabla^i\bar g^{jk}\nabla_j\bar g_{ik}\di\sigma$, $C=\bar\int_{S_{n-1}}\hspace{-0.5cm}\nabla^i\bar g^{jk}\nabla_i\bar g_{jk}\di\sigma$ et $Q= \bar\int_{S_{n-1}}\bar g_{ij}\bar g^{ij}\di\sigma$
\end{theorem}

\begin{proof}

Soit donc $\{x^\alpha\}$ un système de coordonnées normal en $P$.  $\{r,\xi^i\}$ un système de coordonnées géodésiques.  On a vu que la métrique se décompose de la façon suivante:
\begin{equation*}
 g=\mathcal E+h=(\delta_{\alpha\beta}+h_{\alpha\beta})dx^\alpha \otimes dx^\beta=dr^2+(s_{ij}+h_{ij})(rd\xi^i)\otimes(rd\xi^j)
\end{equation*}
où $(s_{ij})$ sont les composantes de la métrique standard sur la sphère $S_{n-1}$ et 
$$h_{ij}=\frac{\partial x^\alpha}{r\partial \xi^i}\frac{\partial x^\beta}{r\partial \xi^j}h_{\alpha\beta},\; \text{ and }h_{ir}=h_{rr}=0$$
et que $h_{ij}=O(r^{\omega+2})$. On a aussi décomposé $h$ de la façon suivante
\begin{equation}\label{defhij}
h_{ij}=r^{\omega+2}\bar g_{ij}+r^{2(\omega+2)}\hat{g}_{ij}+\tilde h_{ij}
\end{equation}
où $\bar g$, $\hat g$ et $\tilde h$ sont des 2-tenseurs symétriques définis sur la sphère  $S_{n-1}$. On choisit $\{\frac{\partial}{\partial r},\frac{\partial}{r\partial\xi^i}\}_{1\leq i\leq n-1}$ et $\{dr,rd\xi^i\}_{1\leq i\leq n-1}$ comme  bases locales de  l'espace tangent $TM$ et cotangent $T^*M$  respectivement. Alors
$$g_{ij}=s_{ij}+h_{ij},\; g_{rr}=1\text{ et }g_{ir}=0$$
Les composantes  $g^{ij}$ de l'inverse  de la métrique sont 
$$g^{ij}=s^{ij}-h^{ij}+O(r^{2\omega+4}),\; {g^{rr}=1}\text{ et }g^{ir}=0$$ 
où $h^{ij}=s^{ik}s^{jl}h_{lk}$. On fait monter et baisser les indices, en utilisant la métrique $(s_{ij})$, sauf pour la métrique $g$. À partir de maintenant, on omet  $O(r^{2\omega+4})$ qui apparaît dans l'expression de $g^{ij}$ ci-dessus, car nos calculs sont à $o(r^{2\omega+2})$ près. On note par $\nabla$ la connexion riemannienne sur la sphère, associée à $s$. $\tilde\nabla$ la connexion associée à la métrique euclidienne $\mathcal E$ dans le corepère  $\{dr,rd\xi^i\}$, alors
$$\tilde\nabla_i=\frac{1}{r}\nabla_i\text{ et }\tilde\nabla_r=\partial_r$$
et $\tilde\partial_i=\frac{1}{r}\partial_i$. Dans le système de coordonnées $\{x^\alpha\}$, 
$\det g=1+O(r^m)$, et dans le système $\{r,\xi^i\}$,  $\det g=r^{2(n-1)}\det s+O(r^m)$, avec $m$ suffisamment grand. D'où $tr\log((\delta^k_i+s^{jk}h_{ij}))=1$. Par le développement limité
$$(\log((\delta^k_i+s^{jk}h_{ij})))^k_{i}=s^{jk}h_{ij}-\frac{1}{2}s^{mk}s^{jl}h_{mj}h_{il}+o(r^{2\omega+4})$$
en tenant compte de la décomposition \eqref{defhij}, on trouve  que $\bar g$, $\hat g$ et $\tilde h$ doivent satisfaire les relations suivantes 
\begin{equation*}
 s^{ij}\bar g_{ij}=0,\; \bar g^{ij}\bar g_{ij}=2 s^{ij}\hat g_{ij}\text{ et }\bar\int_{S(r)}s^{ij}\tilde h_{ij}\di\sigma_r=o(r^{2\omega+2})
\end{equation*}
La première relation vient du fait que le terme d'ordre $\omega+2$ dans le développement de $tr\log((\delta^k_i+s^{jk}h_{ij}))$ est $s^{ij}\bar g_{ij}r^{\omega+2}$ qui doit être nul. Le terme d'ordre $2\omega+4$ est $(s^{ij}\hat g_{ij}-1/2\bar g^{ij}\bar g_{ij})r^{2\omega+4}$ qui doit être également nul. Dans $s^{ij}\hat h_{ij}$, il y a des termes d'ordre entre $\omega+3$ et $2\omega+3$ qui doivent être nuls, les termes d'ordre supérieur à $2\omega+5$ sont négligeables.\\ 
Soient $\tilde\Gamma^k_{ij}$ et $\Gamma^k_{ij}$ les Christoffels de la métrique $g$ et de la métrique euclidienne  $\mathcal E =dr^2+r^2s_{ij}d\xi^id\xi^j$ respectivement. On sait que les $$C^m_{jl}=\tilde\Gamma^m_{lj}-\Gamma^m_{lj}$$ sont les  composantes  d'un certain tenseur  $C$, défini sur la sphère $S_{n-1}$, données par
\begin{equation}\label{cijn}
 C^m_{jl}=\frac{1}{2}g^{mp}(\tilde\nabla_j h_{pl}+\tilde\nabla_l h_{pj}-\tilde\nabla_p h_{jl}),\; C^r_{jl}=-\frac{1}{2}\partial_rh_{jl}\text{ et }C^m_{rj}=\frac{1}{2}g^{mp}\partial_rh_{pj}
\end{equation}
et $C^i_{rr}=C^r_{ri}=0$. Ici les indices  latins  varient entre 1 et $n-1$ et les indices  grecs varient entre  1 et $n$. Dans le système de coordonnées $\{x^\alpha\}$, $g_{\alpha\beta}=\delta_{\alpha\beta}+h_{\alpha\beta}$, les composantes du tenseur de Ricci de la métrique $g$ sont
$$R_{\alpha\beta}=\partial_\gamma \tilde\Gamma^\gamma_{\alpha\beta}-\partial_\beta \tilde\Gamma^\gamma_{\gamma\alpha}+\tilde\Gamma^\gamma_{\gamma\mu}\tilde\Gamma^\mu_{\alpha\beta}-\tilde\Gamma^\gamma_{\beta\mu}\tilde\Gamma^\mu_{\gamma\alpha}$$
D'après la définition du tenseur $C$ et le fait que les Christoffels de la métrique euclidienne $\Gamma^{\alpha}_{\beta\gamma}$ sont identiquement nuls, on obtient l'expression suivante :  
$$R_{\alpha\beta}=\tilde\nabla_\gamma C^\gamma_{\alpha\beta}-\tilde\nabla_\beta C^\gamma_{\gamma\alpha}+C^\gamma_{\gamma\mu}C^\mu_{\alpha\beta}-C^\gamma_{\beta\mu}C^\mu_{\gamma\alpha}$$
T.~Aubin \cite{Aub2} montre que cette expression de Ricci est encore valable si $g=g_0+h$, où $g_0$ est une métrique riemannienne quelconque (pas nécessairement la métrique euclidienne $\mathcal E$).\\
Dans le système de coordonnées $\{r, \xi^i\}$, l'expression du tenseur $C$ ci-dessus devient : 
$$R_{jl}=\partial_rC^r_{jl}+ \tilde\nabla_m C^m_{jl}-\tilde\nabla_j C^m_{ml}+C^m_{mr}C^r_{jl}-C^m_{jr}C^r_{ml}-C^r_{jp}C^p_{rl}+C^m_{mp}C^p_{jl}-C^m_{jp}C^p_{ml}$$
En utilisant la définition du tenseur $C$, on en déduit l'expression suivante des composantes du tenseur de Ricci:
\begin{gather}\label{riccicann}
R_{jl}=-\frac{1}{2}\partial^2_rh_{jl}+\tilde\nabla_mC^m_{jl}-\frac{1}{4}g^{mp}\partial_rh_{mp}\partial_rh_{jl}+
\frac{1}{2}\partial_rh_{ij}\partial_rh_{kl}g^{ik}+C^i_{ik}C^k_{jl}-C^i_{jk}C^k_{il} \\
R_{rr}=-\partial_rC^m_{mr}-C_{rp}^mC_{mr}^p
\end{gather}

Si  $h=O(r^{\omega+2})$ alors $R_g=O(r^\omega)$. De plus, on peut calculer $\bar R$ la partie principale de $R_g$. Pour cela, on doit se focaliser uniquement sur les termes d'ordre  $\omega$ dans l'expression de  $R_g=R_rr+g^{jl}R_{jl}$. Tous les termes de $R_g$ sont négligeables par rapport à $r^\omega$, sauf à priori les deux termes suivants:

$$-\frac{1}{2}g^{jl}(\partial_{rr}h_{jl}+\frac{n-1}{r}\partial_rh_{jl})=-\frac{1}{2}(\omega+2)(\omega+n)s^{jl}\bar g_{jl}r^\omega+o(r^\omega)=o(r^\omega)$$
Finalement ce terme également est négligeable par rapport à $r^\omega$. Il ne fera pas partie des termes de $\bar R$. Le second candidat est 
$$g^{jl}\tilde\nabla_mC^m_{jl}=(s^{jl}-h^{jl})\tilde\nabla_m[(s^{mp}-h^{mp})\tilde\nabla_{l}h_{jp}]+o(r^\omega)$$
car $g^{jl}\tilde\nabla_ph_{jl}=o(r^\omega)$. Donc $g^{jl}\tilde\nabla_mC^m_{jl}=s^{jl}s^{mp}\nabla_{ml}\bar g_{jp}r^\omega+o(r^\omega)$. On conclut que
\begin{equation}\label{rbar}
\bar R=\nabla^{jp}\bar g_{jp}r^\omega
\end{equation}
La première formule du théorème \ref{rtzzz} est démontrée.
Par le lemme \ref{LLLL}, on sait que 
$$\bar\int_{S(r)}R_g\di\sigma_r=O(r^{2\omega+2})$$ 
On cherche les termes d'ordre $2\omega+2$ de cette intégrale. En utilisant l'expression \eqref{riccicann} des composantes de $R_{jl}$, on a: 
$$\bar\int_{S(r)}R_g\di\sigma_r=\bar\int_{S(r)}R_{rr}+g^{jl}R_{jl}\di\sigma_r$$
Ici encore, on doit se focaliser uniquement sur les termes d'ordre $2\omega+2$, d'intégrales non nulles. On doit examiner sept intégrales, six correspondent aux termes de $R_{jl}$ et une à $R_{rr}$. Les calculs suivants sont à $o(r^{2\omega+2})$ près. On a $g^{ij}=s^{ij}-h^{ij}$ à $o(r^{2\omega+2})$ près. Comme $\int_{S_{n-1}}s^{ij}\hat h_{ij}\di \sigma=o(r^{2\omega+2})$, on n'aura pas à se soucier des termes qui proviennent de $\tilde h_{ij}$.  En se servant des  relations $s^{ij}\bar g_{ij}=0$ et  $\bar g^{jl}\bar g_{jl}=2s^{jl}\hat g_{jl}$, on trouve que  l'intégrale correspondant aux premiers termes de $R_{jl}$ donne: 
$$-\frac{1}{2}\bar\int_{S(r)}(s^{jl}-h^{jl})\partial^2_rh_{jl}\di\sigma_r=-\frac{(\omega+2)^2}{2} Q r^{2\omega+2}+o(r^{2\omega+2})$$ 
où $ Q=\bar\int_{S_{n-1}}\bar g^{jl}\bar g_{jl}\di\sigma$

et que l'intégrale  correspondant   au troisième terme de $R_{jl}$ est 
$$-\frac{1}{4}\bar\int_{S(r)}(s^{mp}-h^{mp})(s^{jl}-h^{jl})\partial_rh_{mp}\partial_rh_{jl}\di\sigma_r=o(r^{2\omega+2})$$
L'intégrale correspondant au quatrième terme de $R_{jl}$ devient
$$\frac{1}{2}\bar\int_{S(r)}s^{ik}s^{jl}\partial_rh_{ij}\partial_rh_{kl}\di\sigma_r=\frac{(\omega+2)^2}{2} Q r^{2\omega+2}+o(r^{2\omega+2})$$
La dernière intégrale qui donne des termes du type $Qr^{2\omega+2}$ est
$$\bar\int_{S(r)}R_{rr}\di\sigma_r=-\bar\int_{S(r)}\partial_rC^m_{mr}-C_{rp}^mC_{mr}^p\di\sigma_r=-\frac{(\omega+2)^2}{4} Q r^{2\omega+2}+o(r^{2\omega+2})$$
où $C^p_{mr}$ et $C_{mp}^r$ sont définis par l'expression \eqref{cijn}.\\
En utilisant la formule de Stokes (intégration par parties) et le fait que les intégrales de type
$$\bar\int_{S(r)}s^{jl}s^{mp}\nabla_{mj}h_{pl}\di\sigma_r=\bar\int_{S(r)}\nabla_{mj}h^{mj}\di\sigma_r=0$$
sont nulles (ce sont des intégrales de la divergence  d'un champ de vecteur), l'intégrale correspondant au second terme de $R_{jl}$, se calcule de la façon suivante:

\begin{equation*}
 \begin{split}
\bar\int_{S(r)}\hspace{-0.6cm}g^{jl}\tilde\nabla_mC^m_{jl}\di\sigma_r & =\frac{1}{2r^2}\bar\int_{S(r)}g^{jl}g^{mp}(\nabla_{mj} h_{pl}+\nabla_{ml} h_{pj}-\nabla_{mp} h_{jl})\\
& +g^{jl}(\nabla_mg^{mp})(\nabla_{j} h_{pl}+\nabla_{l} h_{pj}-\nabla_{p} h_{jl})\di\sigma_r\\
& = r^{2\omega+2}\bar\int_{S(r)}\hspace{-0.8cm}-s^{jl}\bar g^{mp}\nabla_{mj} \bar g_{pl}-\bar g^{jl}s^{mp}\nabla_{mj} \bar g_{pl}+\frac{1}{2}\bar g^{jl}s^{mp}\nabla_{mp} \bar g_{jl} \\
&- s^{jl}\nabla_m\bar g^{mp}\nabla_{j} \bar g_{pl}\di\sigma_r+o(r^{2\omega+2})\\
& = (B-\frac{C}{2})r^{2\omega+2}+o(r^{2\omega+2})
 \end{split}
\end{equation*}
où l'on a posé
\begin{equation}\label{BCdef}
  B=\bar\int_{S_{n-1}}\hspace{-0.5cm}\nabla^i\bar g^{jk}\nabla_j\bar g_{ik}\di\sigma\text{ et } C=\bar\int_{S_{n-1}}\hspace{-0.5cm}\nabla^i\bar g^{jk}\nabla_i\bar g_{jk}\di\sigma
\end{equation}

En utilisant $s^{ij}\bar g_{ij}=0$, et la définition des $C^i_{jk}$, on a 
$$C^i_{ik}=\frac{r^{\omega+2}}{2}g^{ip}(\nabla_{i} \bar g_{kp}+\nabla_{k} \bar g_{pi}-\nabla_{p}\bar g_{ik})+o(r^{\omega+2})=o(r^{\omega+2})$$
L'intégrale correspondant au cinquième terme $R_{jl}$ vérifie donc
\begin{equation*}
\bar\int_{S(r)}g^{jl}C^i_{ik}C^k_{jl} \di\sigma_r =o(r^{2\omega+2})
\end{equation*}
et est négligeable devant $r^{2\omega+2}$. Il nous reste à calculer l'intégrale correspondant au sixième terme  $R_{jl}$. 
\begin{equation*}
\begin{split}
-\bar\int_{S(r)}\hspace{-0.6cm}g^{jl}C^m_{jp}C^p_{ml} \di\sigma_r &=-\frac{r^{2\omega+2}}{4} \bar\int_{S(r)} (\nabla^{l} \bar g^{mk}+\nabla^{k} \bar g^{lm}-\nabla^{m}\bar g^{kl})\\
&\hspace{2cm} \times(\nabla_{m} \bar g_{lk}+\nabla_{l} \bar g_{mk}-\nabla_{k}\bar g_{ml})\di\sigma_r +o(r^{2\omega+2})\\
&=(\frac{C}{4}-\frac{B}{2})r^{2\omega+2}+o(r^{2\omega+2})
\end{split}
\end{equation*}

Finalement
\begin{equation}\label{eqrgfg}
\bar\int_{S(r)}R_g\di\sigma_r=(B/2-C/4-(1+\omega/2)^2\bar Q)r^{2\omega+2}+o(r^{2\omega+2})
\end{equation}

\end{proof}

\section*{Preuve du lemme \ref{formule abg}}
Rappelons l'énoncé de ce lemme:
\begin{lemma}\label{formule abg ann}
 \begin{equation*}
\nabla^ib_{ij}=-\sum_{k=1}^q\nabla_j\varphi_k,\;\bar R=\bar R_b=\nabla^{ij}b_{ij}r^{\omega},\;\bar R_a=\nabla^{ij}a_{ij}r^{\omega}=0 \text{ et } s^{ij}b_{ij}=s^{ij}a_{ij}=0
\end{equation*}
\end{lemma}

\begin{proof}
 Les $b_{ij}$ sont définis comme suit:
\begin{equation*}
 b_{ij}=\sum_{k=1}^q\frac{1}{(n-2)(\lambda_k+1-n)}[(n-1)\nabla_{ij}\varphi_k+\lambda_k\varphi_ks_{ij}]
\end{equation*}
En contractant par $\nabla^i$
\begin{equation}\label{esddf}
 \nabla^ib_{ij}=\sum_{k=1}^q\frac{1}{(n-2)(\lambda_k+1-n)}[(n-1)s^{im}\nabla_{mj}\nabla_{i}\varphi_k+\lambda_k\nabla_j\varphi_k]
\end{equation}

D'après la définition du tenseur de courbure de Riemann (voir section \ref{curvature}), on a $$\nabla_{mj}\nabla_{i}\varphi_k=\nabla_{jm}\nabla_{i}\varphi_k-R^l_{imj}\nabla_l\varphi_k$$ 
avec
\begin{equation}\label{curv sphere}
R_{lijm}=s_{lj}s_{im}-s_{lm}s_{ij},\quad R^l_{imj}= \delta^l_ms_{ij}-\delta_j^ls_{mi}
\end{equation}
qui sont  les composantes du tenseur de courbure de Riemann de la sphère  $S_{n-1}$ muni de la métrique standard $s$. Ici les $\varphi_k$ sont des fonctions propres du Laplacien sur la sphère (il ne faut pas les confondre avec les composantes d'un tenseur une fois covariant).\\

D'après les propriétés \ref{prop courbure}, on en déduit que  
$$\nabla_{mj}\nabla_{i}\varphi_k-\nabla_{jm}\nabla_{i}\varphi_k=R_{imj}^l\nabla_{l}\varphi_k$$
Puisque $\Delta\varphi_k=-s^{im}\nabla_{mi}\varphi_k$,
$$s^{im}\nabla_{mj}\nabla_{i}\varphi_k=-\nabla_j\Delta\varphi_k+(n-2)\nabla_j\varphi_k=-(\lambda_k-n+2)\nabla_j\varphi_k$$
qu'on substitue dans l'équation \eqref{esddf}. On trouve
\begin{equation}\label{bij}
\nabla^ib_{ij}=-\sum_{k=1}^{q}\nabla_j\varphi_k
\end{equation}
La première formule est démontrée. Pour la seconde, il suffit de calculer
$$\nabla^{ij}b_{ij}=-\sum_{k=1}^{q}\nabla^j_j\varphi_k=\sum_{k=1}^{q}\Delta_s\varphi_k=r^{-\omega}\bar R$$ 
d'après l'expression \ref{dddd} qui définit $\bar R$. D'autre part, d'après le théorème \ref{zzz} 
$$\bar R=\nabla^{ij}\bar g_{ij}r^\omega=\nabla^{ij}a_{ij}r^\omega+\nabla^{ij}b_{ij}r^\omega$$ 
On en conclut que $\nabla^{ij}a_{ij}=0$.\\
Les deux dernières identités se déduisent aisément de la relation $s^{ij}\bar g_{ij}=0$ et de la définition des $b_{ij}$.
\end{proof}

\section*{Preuve du lemme \ref{mmm}}\label{lemme mmm}

Rappelons d'abord la définition des intégrales $Q_b$, $B_b$ et $C_b$:
\begin{equation*}
 Q_b=\bar\int_{S_{n-1}} b_{ij} b^{ij}\di\sigma,\;B_b=\bar\int_{S_{n-1}}\hspace{-0.5cm}\nabla^i b^{jk}\nabla_j b_{ik}\di\sigma\text{ et }C_b=\bar\int_{S_{n-1}}\hspace{-0.5cm}\nabla^ib^{jk}\nabla_ib_{jk}\di\sigma
\end{equation*}

On commence par démontrer les formules  suivantes (voir équations \eqref{Qb}, \eqref{Bb} et \eqref{Cb}) :

\begin{gather*}
Q_b=\bar\int_{S_{n-1}} b_{ij} b^{ij}\di\sigma=\frac{n-1}{n-2}\sum_{k=1}^{q}\frac{\lambda_k}{\lambda_k-n+1}\bar\int_{S_{n-1}}\varphi_k^2\di\sigma\label{Qba}\\
B_b=-(n-1)Q_b+\sum_{k=1}^{q}\lambda_k\bar\int_{S_{n-1}}\hspace{-0.6cm}\varphi_k^2\di\sigma\label{Bba}\\ C_b=-(n-1)Q_b+\frac{n-1}{n-2}\sum_{k=1}^{q}\lambda_k\bar\int_{S_{n-1}}\hspace{-0.6cm}\varphi_k^2\di\sigma\label{Cba}
\end{gather*}
où les $b_{ij}$ sont donnés par  
\begin{equation}\label{bij def}
 b_{ij}=\sum_{k=1}^q\frac{1}{(n-2)(\lambda_k+1-n)}[(n-1)\nabla_{ij}\varphi_k+\lambda_k\varphi_ks_{ij}]
\end{equation}
Concernant l'intégrale $Q_b$, par une intégration par parties et le fait que $s^{ij}b_{ij}=0$ (voir lemme \ref{formule abg ann}), on obtient 
$$\bar\int_{S_{n-1}}b^{ij}b_{ij}\di \sigma =\sum_{k=1}^q\frac{1}{(n-2)(\lambda_k+1-n)}\bar\int_{S_{n-1}}-(n-1)\nabla^{j}\varphi_k\nabla^i b_{ij}\di\sigma$$ 
D'après \eqref{bij} (rappelons que $\Delta\varphi_k=\lambda_k\varphi_k$),  l'égalité \eqref{Qba} est démontrée.\\
Montrons la formule \eqref{Bba}. Par définition 
\begin{equation*}
 B_b=\bar\int_{S_{n-1}}\hspace{-0.5cm}\nabla^i b^{jk}\nabla_j b_{ik}\di\sigma=-\bar\int_{S_{n-1}}\hspace{-0.5cm} b^{jk}s^{li}\nabla_{lj} b_{ik}\di\sigma
\end{equation*}
On permute les dérivées covariantes dans $\nabla_{lj} b_{ik}$, ensuite on utilise \eqref{bij}, pour avoir
\begin{equation}\label{nablfd}
s^{li} \nabla_{lj} b_{ik}=s^{li}(\nabla_{jl} b_{ik}-R_{ilj}^mb_{mk}-R_{klj}^mb_{im})=-\sum_{l=1}^q\nabla_{jk}\varphi_l+(n-1)b_{jk}
\end{equation}
on a utilisé \eqref{curv sphere} et le fait que $s^{ij}b_{ij}=0$. En reprenant la dernière expression de $B_b$, on en déduit que  
\begin{equation*}
B_b=-(n-1)Q_b-\sum_{l=1}^q\bar\int_{S_{n-1}}\hspace{-0.5cm}\nabla_j b^{jk}\nabla_{k}\varphi_l\di\sigma=
-(n-1)Q_b+\sum_{l=1}^q\sum_{p=1}^q\bar\int_{S_{n-1}}\hspace{-0.5cm}\nabla^k\varphi_p\nabla_{k}\varphi_l\di\sigma
\end{equation*}
Sachant que   $\{\varphi_l\}_{1\leq l\leq q}$ est une famille de fonctions orthogonales pour le produit scalaire  dans $L^2$ et celui de $H_1(S_{n-1})$ (voir équations \eqref{proscal1},\eqref{proscal2}, page \pageref{proscal1}), l'égalité \eqref{Bba} est démontrée. Pour montrer l'égalité \eqref{Cba}, on établit d'abord l'identité suivante:
\begin{equation}\label{nablab}
 \nabla_ib_{jk}=\nabla_jb_{ik}+\frac{1}{n-2}(\nabla^mb_{jm}s_{ik}-\nabla^mb_{im}s_{jk})
\end{equation}
En effet, en utilisant \eqref{bij def}, \eqref{bij} et \eqref{curv sphere}, on obtient
\begin{equation*}
 \begin{split}
  \nabla_ib_{jk} &=\sum_{l=1}^q\frac{1}{(n-2)(\lambda_l+1-n)}[(n-1)\nabla_{ij}\nabla_k\varphi_l+
\lambda_l\nabla_i\varphi_ls_{jk}]\\
& =\sum_{l=1}^q\frac{1}{(n-2)(\lambda_l+1-n)}[(n-1)\nabla_{ji}\nabla_k\varphi_l-(n-1)R_{kij}^m\nabla_m\varphi_l+\lambda_l\nabla_i\varphi_ls_{jk}]\\
& = \nabla_jb_{ik}+\sum_{l=1}^q\frac{1}{n-2}[\nabla_i\varphi_ls_{jk}-\nabla_j\varphi_l
s_{ik}]
 \end{split}
\end{equation*}

Alors
\begin{equation*}
 C_b=\bar\int_{S_{n-1}}\hspace{-0.5cm}\nabla^ib^{jk}\nabla_ib_{jk}\di\sigma=B_b+\frac{1}{n-2}
\bar\int_{S_{n-1}}\hspace{-0.5cm}\nabla_ib^{ji}\nabla^mb_{jm}\di\sigma
\end{equation*}
Si on substitue \eqref{bij} et \eqref{Bba} dans la dernière égalité, on trouve l'expression \eqref{Cba}. 

Rappelons l'énoncé du lemme \ref{mmm}:

 \begin{lemma}\label{mmmann}
Si  $\mu=\omega$ et $\bar g_{ij}=a_{ij}+b_{ij}$, alors
\begin{equation*}
\bar\int_{S(r)}\hspace{-0.4cm}R\di\sigma_r=\bar\int_{S(r)}\hspace{-0.4cm}R_a+R_b\di\sigma_r+o(r^{2(\omega+1)})\leq [B_b/2-C_b/4-(1+\omega/2)^2Q_b]r^{2(\omega+1)}+o(r^{2(\omega+1)})
\end{equation*} 
\begin{equation*}
B_b/2-C_b/4-(1+\omega/2)^2Q_b=\sum_{k=1}^q u_k\bar\int_{S_{n-1}}\varphi_k^2\di\sigma
\end{equation*} 
avec
\begin{equation*}
 u_k=\biggl(\frac{n-3}{4(n-2)}-\frac{(n-1)^2+(n-1)(\omega+2)^2}{4(n-2)(\nu_k-n+1)}\biggr)\nu_k
\end{equation*}
\end{lemma}

\begin{proof}
D'après le lemme \ref{formule abg} (démontré ci-dessus):
\begin{equation}\label{relaab}
 r^{-\omega}\bar R=\nabla^{ij}\bar g_{ij}=\nabla^{ij}b_{ij}\text{ et }\nabla^{ij}a_{ij}=s^{ij}a_{ij}=s^{ij}b_{ij}=0
\end{equation}

Montrons que $Q=Q_a+Q_b$, $B=B_a+B_b$ et $C=C_a+C_b$. 
$$Q=\bar\int_{S_{n-1}}(a^{ij}+b^{ij})(a_{ij}+b_{ij})\di \sigma=Q_a+Q_b+2\bar\int_{S_{n-1}}a^{ij}b_{ij}\di \sigma$$
On a 
$$a^{ij}b_{ij}=\sum_{k=1}^q\frac{1}{(n-2)(\lambda_k+1-n)}a^{ij}[(n-1)\nabla_{ij}\varphi_k+\lambda_k\varphi_ks_{ij}]$$
En intégrant sur $S_{n-1}$ l'expression ci-dessus et en utilisant les relations \ref{relaab}, on en déduit que 
\begin{equation}\label{aszfd}
\bar\int_{S_{n-1}}a^{ij}b_{ij}\di \sigma=0\quad \text{et que } Q=Q_a+Q_b
\end{equation}
 
Par un raisonnement analogue au précédent, montrons que $B=B_a+B_b$. D'après la définition de $B$ (voir \ref{BCdef}), on a 
$$B=\bar\int_{S_{n-1}}\hspace{-0.5cm}\nabla^i(a^{jk}+b^{jk})\nabla_j(a_{ik}+b_{ik})\di\sigma=B_a+B_b+2\bar\int_{S_{n-1}}\hspace{-0.5cm}\nabla^ia^{jk}\nabla_jb_{ik}\di\sigma$$
Par une intégration par parties, on obtient
\begin{equation*}
B=B_a+B_b-2\bar\int_{S_{n-1}}\hspace{-0.5cm}a^{jk}\nabla^i\nabla_jb_{ik}\di\sigma
\end{equation*}
En utilisant l'identité \eqref{nablfd}, écrite sous la forme suivante
\begin{equation*}
\nabla^i \nabla_{j} b_{ik}=-\sum_{l=1}^q\nabla_{jk}\varphi_l+(n-1)b_{jk}
\end{equation*}
et les relations \eqref{relaab}, \eqref{aszfd}, on en conclut que
\begin{equation}\label{gfdg}
\bar\int_{S_{n-1}}\hspace{-0.5cm}a^{jk}\nabla^i\nabla_jb_{ik}\di\sigma=0\text{ et }B=B_a+B_b 
\end{equation}

La dernière formule à établir est $C=C_a+C_b$. Or, d'après la définition de $C$ (voir \eqref{BCdef}),
$$C=\bar\int_{S_{n-1}}\hspace{-0.5cm}\nabla^i(a^{jk}+ b^{jk})\nabla_i(a_{jk}+b_{jk})\di\sigma=C_a+C_b-2\bar\int_{S_{n-1}}\hspace{-0.5cm}a^{jk}\nabla^i\nabla_ib_{jk}\di\sigma$$
D'après l'identité \eqref{nablab}
\begin{equation*}
\begin{split}
 \nabla^i\nabla_ib_{jk} & =\nabla^i\nabla_jb_{ik}+\frac{1}{n-2}(\nabla^{i}\nabla^mb_{jm}s_{ik}-\nabla^{i}\nabla^mb_{im}s_{jk})\\
& =\nabla^i\nabla_jb_{ik}-\frac{1}{n-2}\sum_{l=1}^{q}(\nabla_{kj}\varphi_l+\lambda_l\varphi_ls_{jk})
\end{split}
\end{equation*}
Ici on a juste utilisé l'expression \eqref{bij} et le fait que $\Delta\varphi_l=\lambda\varphi_l$. En contractant cette expression de $\nabla^i\nabla_ib_{jk}$ avec $a^{jk}$, en utilisant \eqref{gfdg} et les relations \eqref{relaab}, on en conclut que $\bar\int_{S_{n-1}}\hspace{-0.5cm}a^{jk}\nabla^i\nabla_ib_{jk}\di\sigma=0$ et $C=C_a+C_b$.\\
D'après le théorème \ref{zzz}
$$\bar\int_{S(r)}R\di\sigma_r =[B/2-C/4-(1+\omega/2)^2Q]r^{2(\omega+1)}+o(r^{2(\omega+1)})$$
et par ce qu'on vient de prouver, on en déduit que 
$$\bar\int_{S(r)}R\di\sigma_r=\bar\int_{S(r)}R_a+R_b\di\sigma_r +o(r^{2\omega+2})$$ 
Comme $\nabla^{ij}a_{ij}=0$, $\bar R_a=0$,  l'ordre de la partie $R_a$ est donc supérieur à $\omega+1$. D'après le théorème \ref{aaa}, $\bar\int_{S(r)}R_a\di\sigma_r\leq 0$. D'où l'inégalité du lemme.
\end{proof}

\section*{Détails des calculs du théorème \ref{theo}}\label{nablafi}

On commence par rappeller les définitions données dans la section \ref{aubin computations}.

\begin{equation}\label{def I}
I_a^b(\e)=\int_0^{\delta/\e}\frac{t^b}{(1+t^2)^a}\di t\text{ et }I_a^b=\lim_{\e\to 0}I_a^b(\e)
\end{equation}
alors 
\begin{equation}\label{Iab appendice}
I_a^b(\e)=
\begin{cases}
 I_a^b+O(\e^{2a-b-1}) \text{ si }2a-b>1\\
\log\e^{-1}+O(1)\text{ si }b=2a-1
\end{cases} 
\end{equation}
En effet, si $2a-b>1$, 
$$I_a^b-I_a^b(\e)=\int_{\delta/\e}^{+\infty}\frac{t^b}{(1+t^2)^a}\di t\leq \int_{\delta/\e}^{+\infty}t^{b-2a}\di t\leq \frac{\e^{2a-1-b}}{(2a-1-b)\delta^{2a-b-1}}$$
Si $b=2a-1$ alors pour $\e$ suffisamment petit
$$I_a^{2a-1}(\e)\leq\int_0^{1}\frac{t^{2a-1}}{(1+t^2)^a}\di t+\int_1^{\delta/\e}\frac{1}{t}\di t$$

Par des intégrations par parties, on établit les relations suivantes : 
\begin{equation}\label{relations appendice}
 I_a^b=\frac{b-1}{2a-b-1}I_a^{b-2}=\frac{b-1}{2a-2}I_{a-1}^{b-2}=\frac{2a-b-3}{2a-2}I_{a-1}^{b},\quad \frac{4(n-2)I_n^{n+1}}{(I^{n-2}_n)^{(n-2)/n}}=n
\end{equation}

Soit $\varphi_\e$ une fonction test définie dans \eqref{fonction test} (voir page \pageref{fonction test}). On calcule $I_g(\varphi_\e)$. En utilisant l'inégalité $(a-b)^\beta\geq a^\beta-\beta a^{\beta-1}b$ pour $0<b<a$, on a  $\beta\geq 2$, $0\leq \alpha< (n-2)(\beta-1)-n$ 
\begin{equation}\label{hhh}
 \int_Mr^\alpha u_\e^\beta \di v=\omega_{n-1}\int_0^\delta r^{\alpha+n-1}u_\e^\beta(r)\di r=\omega_{n-1}I_{(n-2)\beta/2}^{\alpha+n-1}\e^{\alpha+n-\beta(n-2)/2}+O(\e^{n-2})
\end{equation}
Ce type d'intégrales apparait plusieurs fois dans les calculs suivants, il permet de négliger le terme constant dans l'expression de $u_\e$, définie dans \eqref{fonction test}, lorsque l'on choisit  $\delta$ suffisamment petit et $\e$ plus petit que $\delta$. On commence par calculer $\|\nabla\varphi_\e\|^2$ (la définition de $\varphi_\e$ est donnée dans la section \ref{proof hv}). D'après la formule $$|\nabla_g\varphi_\e|^2=(\partial_r\varphi_\e)^2+r^{-2}|\nabla_s\varphi_\e|^2$$   
on a l'équation \eqref{nabla varphi} suivante:

\begin{multline*}
\int_M|\nabla_g\varphi_\e|^2\di v=\int_M|\nabla_g u_\e|^2\di v+\int_0^\delta[\partial_r(r^{(\alp)} u_\e)]^2r^{n-1}\di r\int_{S_{n-1}}
\hspace{-0.5cm}f^2 \di\sigma +\\
\int_0^\delta u^2_\e r^{n+2\omega+1}\di r\int_{S_{n-1}}\hspace{-0.5cm} |\nabla f|^2 \di \sigma 
\end{multline*}

On exprime les intégrales ci-dessus, en utilisant les intégrales $I_b^a$, définies plus haut. On effectue le changement de variable $t=r/\e$.  Ce qui donne les expressions suivantes
\begin{gather*}
 \int_M|\nabla_g u_\e|^2\di v=(n-2)^2\omega_{n-1} I_n^{n+1}+O(\e^{n-2})\text{ et }\\ \int_0^\delta u^2_\e r^{n+2\omega+1}\di r\int_{S_{n-1}}\hspace{-0.5cm} |\nabla f|^2 \di \sigma =I^{n+2\omega+1}_{n-2}\|\nabla_s f\|^2
\end{gather*}

\begin{equation*}
\begin{split}
 \int_0^\delta[\partial_r(r^{\alp} u_\e)]^2r^{n-1}\di r\int_{S_{n-1}}
\hspace{-0.5cm}f^2 \di\sigma & =\|f\|^2\int_0^\delta\e^{n-2}\biggl(\frac{(\omega-n+4)r^{\omega+3}+\e^2(\omega+2)r^{\omega+1}}{(\e^2+r^2)^{n/2}}\biggr)^2r^{n-1}\di r \\
& =[(\omega-n+4)^2I_n^{2\omega+n+5}(\e)+2(\alp)(\omega-n+4)I_n^{2\omega+n+3}(\e)\\
& +(\alp)^2I_n^{2\omega+n+1}(\e)]\|f\|^2\e^{2\omega+4}+o(\e^{2\omega+4})
\end{split}
\end{equation*}
Si on regroupe ensemble ces trois intégrales, on obtient \eqref{grad}:

\begin{multline}\label{nabla appendice}
\int_M|\nabla_g\varphi_\e|^2\di v=(n-2)^2\omega_{n-1}I_n^{n+1}(\e)+\e^{2\omega+4}\biggl\{\int_{S_{n-1}}|\nabla f|^2 \di\sigma I_{n-2}^{2\omega+n+1}(\e)+\\
\int_{S_{n-1}}\hspace{-0.5cm}f^2\di\sigma [(\omega-n+4)^2I_n^{2\omega+n+5}(\e)\\+2(\alp)(\omega-n+4)I_n^{2\omega+n+3}(\e)+(\alp)^2I_n^{2\omega+n+1}(\e)]\biggr\}
\end{multline}

Pour avoir \eqref{norm} (page \pageref{norm}),  il suffit d'écrire le développement limité de $\varphi_\e^N$ et ensuite utiliser l'égalité \eqref{hhh}.

\begin{multline}\label{norme appendice}
\|\varphi_\e\|_N^{-2}=(\omega_{n-1}I^{n-1}_n)^{-2/N}\bigl\{1+\\
-(N-1)\e^{2(\alp)}\int_{S_{n-1}}\hspace{-0.5cm}f^2 \di\sigma I_n^{2\omega+n+3}/(\omega_{n-1}I^{n-1}_n)\bigr\}+O(\e^{\min(3\omega+6,n-2)})
\end{multline}

Il nous reste seulement à calculer $\int_MR_g\varphi_\e^2\di v$. La fonction $f$ est définie sur la sphère $S_{n-1}$. On sait qu'on peut la définir sur $S(r)$ pour tout $r>0$ en posant $f(\xi/r)$ si $\xi\in S(r)$. On garde la même notation pour la fonction ainsi redéfinie.  D'après le lemme \ref{LLLL}, on sait que  $\bar\int_{S(r)}r^{-2\omega-2}R_g\di\sigma=O(1)$, on en déduit, en effectuant le changement de variable $t=r/\e$, que
\begin{equation*}
\begin{split}
\int_MR_gu_\e^2\di v & =\e^{2\omega+4}\omega_{n-1}\bar\int_{S(r)}r^{-2\omega-2}R_g\di \sigma I_{n-2}^{n+2\omega+1}(\e)\\
& =\begin{cases}
 \e^{2\omega+4}\omega_{n-1}\bar\int_{S(r)}r^{-2\omega-2}R_g\di \sigma I_{n-2}^{n+2\omega+1}+o(\e^{2\omega+4})\text{ si }n>2\omega +6\\
\e^{2\omega+4}\log\e^{-1} \omega_{n-1}\bar\int_{S(r)}r^{-2\omega-2}R_g\di \sigma +O(\e^{2\omega+4})\text{ si }n=2\omega +6
\end{cases}
\end{split}
\end{equation*}
D'autre part $R=\bar R+o(r^\mu)$ avec $\mu\geq \omega$ (cf. lemme \ref{LLLL}), d'où 
\begin{equation*}
\begin{split}
\int_Mfu_\e^2R_gr^{\omega+2} \di v & =\e^{\omega+\mu+4}I_{n-2}^{\omega+\mu+n+1}(\e)\omega_{n-1}\bar\int_{S(r)}\hspace{-0.3cm}r^{-\mu}f(\xi)\bar R\di\sigma+o(\e^{\omega+\mu+4})\\
& =\begin{cases}
\e^{\omega+\mu+4}I_{n-2}^{\omega+\mu+n+1}\omega_{n-1}\bar\int_{S(r)}\hspace{-0.1cm}r^{-\mu}f(\xi)\bar R\di\sigma+o(\e^{\omega+\mu+4}) \text{ si }n-6> \omega+\mu\\
\e^{\omega+\mu+4}\log\e^{-1}\omega_{n-1}\bar\int_{S(r)}\hspace{-0.1cm}r^{-\mu}f(\xi)\bar R\di\sigma+O(\e^{\omega+\mu+4})\text{ si }n-6= \omega+\mu
\end{cases}
\end{split}
\end{equation*}
Si $n>\omega+\mu+6$ alors

\begin{equation}\label{scalaire appendice}
\begin{split}
\int_MR_g\varphi_\e^2\di v & =\int_MR_gu_\e^2\di v-2\int_Mfu_\e^2R_gr^{\omega+2} \di v+\int_Mf^2u_\e^2R_gr^{2\omega+4} \di v\\
& =\e^{2\omega+4}\omega_{n-1}\bar\int_{S(r)}r^{-2\omega-2}R_g\di \sigma I_{n-2}^{n+2\omega+1}-\\ 
& 2\e^{\omega+\mu+4}I_{n-2}^{\omega+\mu+n+1}\omega_{n-1}\bar\int_{S(r)}\hspace{-0.3cm}r^{-\mu}f(\xi)\bar R \di\sigma(\xi)+o(\e^{2\omega+4})\\
\end{split}
\end{equation}

Si $n=2\omega+6$ et $\mu=\omega$ alors

\begin{equation}\label{scalaire log}
\int_MR_g\varphi_\e^2\di v =\e^{2\omega+4}\log\e^{-1}\omega_{n-1}\{\bar\int_{S(r)}r^{-2\omega-2}R_g\di \sigma - 
 2\bar\int_{S(r)}\hspace{-0.3cm}r^{-\mu}f(\xi)\bar R \di\sigma(\xi)\}+O(\e^{2\omega+4})
\end{equation}
Rappelons que
$$I_g(\varphi_\e)=\biggl(\int_M|\nabla\varphi_\e|^2\di v+\frac{n-2}{4(n-1)}\int_MR_g\varphi_\e^2\di v\biggr)\|\varphi_\e\|_N^{-2}$$
Maintenant, on a tout les ingrédients nécessaires pour donner l'expression détaillée de $I_g(\varphi_\e)$. On l'obtient, en combinant \eqref{nabla appendice}, \eqref{norme appendice}, \eqref{scalaire appendice} et \eqref{scalaire log} et le lemme \ref{norme f2} ci-dessous. On en conclut que si $n>2\omega+6$ alors
\begin{multline*}
I_g(\varphi_\e)= \frac{n(n-2)}{4}\omega_{n-1}^{2/n}+(\omega_{n-1}I^{n-1}_n)^{-2/N}I_{n-2}^{n+2\omega+1}\e^{2\omega+4}\times \\
\biggl\{\frac{(n-2)\omega_{n-1}}{4(n-1)}\bar\int_{S(r)}r^{-2\omega-2}R_g\di \sigma -\frac{n-2}{2(n-1)}\int_{S_{n-1}}\hspace{-0.3cm}f(\xi)\bar R \di\sigma+\int_{S_{n-1}}\hspace{-0.5cm}|\nabla f|^2 \di\sigma +\\
-\frac{n(n-2)^2-(\omega+2)^2(n^2+n+2)}{(n-1)(n-2)}\int_{S_{n-1}}\hspace{-0.5cm}f^2\di\sigma\biggr\}+
o(\e^{2\omega+4)})
\end{multline*}

si $n=2\omega+6$ alors

\begin{multline*}
I_g(\varphi_\e)= \frac{n(n-2)}{4}\omega_{n-1}^{2/n}+(\omega_{n-1}I^{n-1}_n)^{-2/N}\e^{2\omega+4}\log\e^{-1}\times\\
\biggl\{\frac{(n-2)\omega_{n-1}}{4(n-1)}\bar\int_{S(r)}r^{-2\omega-2}R_g\di \sigma
-\frac{n-2}{2(n-1)}\int_{S_{n-1}}\hspace{-0.3cm}f(\xi)\bar R \di\sigma+\\
\int_{S_{n-1}}|\nabla f|^2 \di\sigma +(\omega+2)^2\int_{S_{n-1}}\hspace{-0.5cm}f^2\di\sigma\biggr\}+O(\e^{2\omega+4})
\end{multline*}
\begin{lemma}\label{norme f2}
 On a les relations suivantes pour tout $n>2\omega+6$: 
\begin{multline*}
 (\omega-n+4)^2I_n^{2\omega+n+5}+2(\alp)(\omega-n+4)I_n^{2\omega+n+3}+(\alp)^2I_n^{2\omega+n+1}\\
-(N-1)(n-2)^2 \frac{I_n^{2\omega+n+3}I_n^{n+1}}{I_n^{n-1}}=-\frac{n(n-2)^2-(\omega+2)^2(n^2+n+2)}{(n-1)(n-2)}I_{n-2}^{n+2\omega+1}
\end{multline*}
Si $n=2\omega+6$ alors
\begin{multline*}
 (\omega-n+4)^2I_n^{2\omega+n+5}(\e)+2(\alp)(\omega-n+4)I_n^{2\omega+n+3}(\e)+(\alp)^2I_n^{2\omega+n+1}(\e)\\
-(N-1)(n-2)^2 \frac{I_n^{2\omega+n+3}(\e)I_n^{n+1}}{I_n^{n-1}}=(\omega+2)^2\log\e^{-1}+O(1)
\end{multline*}
Ces relations apparaissent dans l'expression de $I_g(\varphi_\e)$, comme étant le coefficient du terme $\int_{S_{n-1}}f^2\di \sigma$.
\end{lemma}
\begin{proof}
Si $n=2\omega+6$ alors  
$I_n^{2\omega+n+3}(\e)=I_n^{2\omega+n+3}+O(\e^{n-2})$, $I_n^{2\omega+n+1}(\e)=I_n^{2\omega+n+1}+O(\e^{n-2})$   et $I_n^{2\omega+n+5}(\e)=\log\e^{-1}+O(1)$ (cf. équation \eqref{Iab appendice}); la deuxième expression du lemme est démontrée.\\
Maintenant, on suppose que $n>2\omega+6$. En utilisant les relations \eqref{relations appendice}, on trouve

\begin{eqnarray*}
I_n^{2\omega+n+5}& =\frac{(2\omega+n+4)(2\omega+n+2)}{4(n-1)(n-2)}I_{n-2}^{n+2\omega+1}\qquad I_n^{2\omega+n+3} & =\frac{(2\omega+n+2)(n-2\omega-6)}{4(n-1)(n-2)}I_{n-2}^{n+2\omega+1}\\ 
I_n^{2\omega+n+1}& =\frac{(n-2\omega-4)(n-2\omega-6)}{4(n-1)(n-2)}I_{n-2}^{n+2\omega+1}\qquad I_n^{n+1} 
& =\frac{n}{n-2}I^{n-1}_n
\end{eqnarray*}
Il suffit de montrer que le polynôme $P_2$, défini pour tout $\omega\in \mathbb N$ par
\begin{multline*}
P_2(\omega+2)=(\omega-n+4)^2(2\omega+n+4)(2\omega+n+2)+2(\omega+2)(\omega-n+4)(2\omega+n+2)(n-2\omega-6)\\
+(\omega+2)^2(n-2\omega-4) (n-2\omega-6)-n(n+2)(2\omega+n+2)(n-2\omega-6)
\end{multline*}
est de degré 2 et est égal à
$$P_2(\omega+2)=4(\omega+2)^2(n^2+n+2)-4n(n-2)^2$$
En effet, on vérifie aisément que les termes de degré 4 se simplifient et que $P_2(-X)=P_2(X)$, alors $P_2$ est pair de degré 2. On en déduit que $P_2(X)=a_nX^2+b_n$, où $b_n=P_2(0)=-4n(n-2)^2$ et $a_n=P_2''(0)/2=4(n^2+n+2)$ 
\end{proof}

\section*{Théorème \ref{aaa}}
Dans son article \cite{Aub3}, T.~Aubin démontre le résultat suivant:
\begin{theorem}\label{aaa ann}
 Si $\mu\geq\omega+1$ alors il existe une constante $C(n,\omega)>0$ telle que 
$$\bar\int_{S(r)}R\di\sigma_r=C(n,\omega)(-\Delta_g)^{\omega+1} R(P)r^{2\omega+2}+o(r^{2\omega+2})$$
$(-\Delta_g)^{\omega+1} R(P)$ est strictement négative et $I_g(u_{\e})<\frac{n(n-2)}{4}\omega_{n-1}^{2/n}$.\\ 
où $u_\e$ est  définie dans la section \ref{proof hv} (voir équation \eqref{uepsilon}).
\end{theorem}
Tout d'abord, remarquons que si $\bar\int_{S(r)}R\di\sigma_r<0$, d'après ce qui a été fait à la section \ref{proof hv}, il suffit de prendre $f=0$ pour que $\varphi_\e=u_\e$. L'inégalité
$$I_g(u_{\e})<\frac{n(n-2)}{4}\omega_{n-1}^{2/n}$$
est une conséquence immédiate des inégalités \eqref{estime phi1}, \eqref{estime phi2}.\\
Il suffit de montrer que $(-\Delta_g)^{\omega+1} R(P)<0$. Pour cela, T.~Aubin  donne un schéma assez détaillé de la preuve. Le cas $\omega= 1 $ ou $2$ sont des conséquences des travaux de E.~Hebey et  M.~Vaugon \cite{HV}. Le cas $\omega=3$ est fait par L.~Zhang (communication privée). La méthode de T.~Aubin marche pour $\omega$ quelconque. Notons par $SymT$ le symétrisé du tenseur $T$ par rapport à tout ses indices, et par $C(2,2)$  l'application de contraction des indices deux à deux (voir la remarque de la section \ref{hv workss} pour des exemples). On pose
\begin{eqnarray*}
A &=C(2,2)Sym\nabla_\alpha R_{pijq}\nabla_\beta R_{pq} & B=C(2,2)Sym\nabla_\alpha R_{pijq}\nabla_{\tilde\beta l}\nabla_pR_{qk}\\
\tilde C &=C(2,2)Sym\nabla_\alpha R_{ip}\nabla_\beta R_{jp} & Z=C(2,2)Sym\nabla_\alpha R_{pklq} 
\end{eqnarray*} 
$R_{ijkl}$, $R_{ij}$ sont les composantes du tenseur de courbure de Riemann et de Ricci. Tout les calculs sont faits au point $P$, qu'on omettra dans les expressions pour des raisons de simplicité. Les indices grecs sont des multi-indices de longueur $\omega$ (i.e. $|\beta|=|\alpha|=\omega$), si ils contiennent un tilde, alors ils deviennent de longueur $\omega-2$ (i.e. $|\tilde\beta|=|\tilde\alpha|=\omega-2$). Les indices latins sont de longueur 1. Un indice ou multi-indice noté deux fois, il y a sommation sur cet indice. sur les autres indices on considère toutes les permutations, afin d'avoir le symétrisé. Par des calculs combinatoires et les identités de Bianchi, on a le résultat suivant:
\begin{equation*}
 2(\omega+2)^2C(2,2)Sym\nabla_{\alpha\beta kl}R+C(\omega)I=0
\end{equation*}
 avec 
\begin{equation*}
 I=Z+2(\omega+3)^2(A+\tilde C)+2\omega(\omega+3)B\text{ et }C(\omega)=\frac{(\omega+1)^2(\omega+2)^2(2\omega+2)!}{[(\omega+3)!]^2}
\end{equation*}
On sait qu'il existe une constante $K>0$ telle que $(-\Delta)^{\omega+1}R=KC(2,2)Sym\nabla_{\alpha\beta kl}R$. Pour démontrer le théorème, il suffit de montrer que $I>0$. Pour cela T.~Aubin considère de nouveaux termes et de  types de contractions qui lui permettent d'écrire $I$ comme somme de ces termes qui vérifient certaines relations et inégalités entre eux (ces relations sont obtenues par des contractions, en utilisant les identités de Bianchi). Grâce à ces nouvelles relations, il en déduit la positivité de $I$.

\chapter{Détails des calculs (Chapitre \ref{CHYSr})}\label{ch5}
\markboth{D\'ETAILS DES CALCULS}{D\'ETAILS DES CALCULS}

\section*{Lemme \ref{lemme infg}}
On a vu que la preuve du lemme est ramenée à prouver que

\begin{equation}\label{propo ann}
\bigcap_{k=1}^q]x_k,y_k[\neq \varnothing
\end{equation}

où 
\begin{equation*}
x_k=\frac{(n-2)^2-(n-2)\sqrt{\Delta_k}}{d_k},\; y_k=\frac{(n-2)^2+(n-2)\sqrt{\Delta_k}}{d_k}\text{ et }\Delta_k=\{(n-2)^2-\frac{d_ku_k}{\nu_k^2}\}
\end{equation*}
D'après le lemme \ref{lemme poly}, $\Delta_k>0$ pour tout $k\leq q\leq [\omega/2]$. Puisque $(d_k)_k$ est décroissante, il est facile de vérifier que 
\begin{equation}\label{ine solut ann}
 \forall k< j\leq [\frac{\omega}{2}]\qquad x_k< y_j
\end{equation}
 (voir équations \eqref{lam},  \eqref{dkk} pour la définition de $\nu_k$ et $d_k$). On vérifie aussi que $u_{\omega/2}<0$ (voir équations \eqref{ukk}), cela entraîne que si $\omega$ est pair alors $x_{\omega/2}<0$.

\subsection*{Le cas $\boldsymbol{\omega=5}$} 
D'après les remarques ci-dessus, il suffit de montrer que $x_2<y_1$. Ce qui revient à montrer que 
\begin{equation*}
 (n-2)(d_2-d_1)+d_1\sqrt{\Delta_2}+d_2\sqrt{\Delta_1}>0
\end{equation*}
Dans ce cas 
\begin{gather*}
 \nu_1=5(n+3),\quad\nu_2=3(n+1)\\
d_1=4(4n^3+53n^2+10n+128),\quad d_2=4(2n^3+47n^2+42n+104)\\
\frac{u_2}{\nu_2}=\frac{n^2-49n+36}{8(n-2)(n+2)}
\end{gather*}
Après une décomposition en éléments simples de la fraction rationnelle $\displaystyle\frac{d_2}{\nu_2}\frac{u_2}{\nu_2}$ par rapport à $n$, on établit que 
\begin{equation*}
\begin{split}
 \Delta_2=(n-2)^2-\frac{d_2}{\nu_2}\frac{u_2}{\nu_2} &=
 \frac{2}{3}n^{2} + \frac{29}{6}n  +\frac{1076}{3} +\frac{2842}{9(n - 2)}  -  \frac {1104}{n + 2} 
 +  \frac {4601}{9(n + 1)} \\
& > \frac{2}{3} (n+\frac{29}{8})^2
\end{split}
\end{equation*}
D'où
\begin{multline*}
 (n-2)(d_2-d_1)+d_1\sqrt{\Delta_2}>-8(n-2)(n^3+3n^2-16n+12)\\
+4(4n^3+53n^2+10n+128)\sqrt{\frac{2}{3}} (n+\frac{29}{8})>0
\end{multline*}

\subsection*{Le cas $\boldsymbol{\omega=6}$} 

On doit encore montrer que $x_2<y_1$. En effet l'intersection avec l'intervalle $]x_3,y_3[$ n'est pas vide car $x_3<0$, $y_3>x_2$ et $y_3>x_1$. Il suffit donc de montrer que 
\begin{equation*}
 (n-2)(d_2-d_1)+d_1\sqrt{\Delta_2}+d_2\sqrt{\Delta_1}>0
\end{equation*}
Dans ce cas 
\begin{gather*}
 \nu_1=6(n+4),\quad\nu_2=4(n+2)\\
d_1=4(5n^3+74n^2+176),\quad d_2=4(3n^3+64n^2+44n+144)\\
\frac{u_2}{\nu_2}=\frac{n^2-31n+18}{6(n-2)(n+3)}
\end{gather*}
On répète les mêmes calculs que dans le cas précédent. On établit que
\begin{equation*}
\begin{split}
 \Delta_2=(n-2)^2-\frac{d_2}{\nu_2}\frac{u_2}{\nu_2} &= \frac{1}{2}n^2+\frac{7}{3}n+\frac{892}{3}+ \frac{512}{3(n-2)}+\frac{1008}{n+2}-\frac{2028}{n+3} \\
& > \frac{1}{2}(n+\frac{7}{3})^2
\end{split}
\end{equation*}
D'où
\begin{multline*}
 (n-2)(d_2-d_1)+d_1\sqrt{\Delta_2}>-8(n-2)(n^3+5n^2-22n+16)\\
+2\sqrt{2}(5n^3+74n^2+176) (n+\frac{7}{3})>0
\end{multline*}

\subsection*{Le cas $\boldsymbol{\omega=7}$}

Par contre dans ce cas, on doit vérifier que l'intersection 
$$\bigcap_{k=1}^3]x_k,y_k[$$
est non vide. On a déjà les inégalités suivantes: $y_3>x_2>0$, $y_3>x_1>0$ et $y_2>x_1$. Il suffit de montrer que $y_1>x_3$, $y_1>x_2$ et $y_2>x_3$, ce qui est équivalent à montrer que 

\begin{equation*}
\forall 1\leq i<j\leq 3\quad (n-2)(d_j-d_i)+d_i\sqrt{\Delta_j}+d_j\sqrt{\Delta_i}>0
\end{equation*}

 On reprend les mêmes calculs.

\begin{gather*}
 \nu_1=7(n+5),\quad\nu_2=5(n+3),\quad \nu_3=3(n+1)\\
d_1=4(6n^3+99n^2-14n+232),\quad d_2=4(4n^3+85n^2+42n+192)\\
d_3=4(2n^3+79n^2+74n+168)\\
\frac{u_2}{\nu_2}=\frac{3n^2-75n+32}{16(n-2)(n+4)},\quad \frac{u_3}{\nu_3}=\frac{n^2-81n+68}{8(n-2)(n+2)}
\end{gather*}

\begin{equation*}
\begin{split}
 \Delta_2=(n-2)^2-\frac{d_2}{\nu_2}\frac{u_2}{\nu_2} &= \frac{2}{5}n^2+\frac{5}{4}n+\frac{1413}{5}-\frac{3572}{n+4}+\frac{51333}{25(n+3)}+\frac{2862}{25(n-2)}\\
& > \frac{2}{5}(n+\frac{25}{16})^2
\end{split}
\end{equation*}

\begin{equation*}
\begin{split}
 \Delta_3=(n-2)^2-\frac{d_3}{\nu_3}\frac{u_3}{\nu_3} &= \frac{2}{7}n^2-\frac{9}{14}n+\frac{2708}{21}-\frac{11951}{3(n+6)}+\frac{135809}{49(n+5)}+\frac{1755}{49(n-2)}\\
& > \frac{2}{7}(n-\frac{9}{8})^2
\end{split}
\end{equation*}
On montre que les inégalités suivantes sont strictes.

\begin{multline*}
 (n-2)(d_2-d_1)+d_1\sqrt{\Delta_2}>-8(n-2)(n^3+7n^2-28n+20)\\
+4\sqrt{\frac{2}{5}}(6n^3+99n^2-14n+232)(n+\frac{25}{16})>0
\end{multline*}

\begin{multline*}
 (n-2)(d_3-d_1)+d_1\sqrt{\Delta_3}>-8(n-2)(2n^3-10n^2-44n+32)\\
+4\sqrt{\frac{2}{7}}(6n^3+99n^2-14n+232)(n-\frac{9}{8})>0
\end{multline*}

\begin{multline*}
 (n-2)(d_3-d_2)+d_2\sqrt{\Delta_3}>-8(n-2)(n^3+3n^2-16n+12)\\
+4\sqrt{\frac{2}{7}}(4n^3+85n^2+42n+192)(n-\frac{9}{8})>0
\end{multline*}

\subsection*{Le cas $\boldsymbol{8\leq \omega\leq 15}$}

\'A partir de 8 jusqu'à 15, on utilise le logiciel Maple pour faire la décomposition en éléments simples de la fraction rationnelle $\Delta_k$. On obtient la forme suivante:
$$\Delta_k=a_kn^2+b_kn+d_k+\frac{e_k}{n-2}+\frac{f_k}{\nu_k-n+1}$$ 
En utilisant encore ce logiciel, on montre que
$$\sqrt{\Delta_k}>\sqrt{a_k}(n+\frac{b_k}{2a_k})$$
où les coefficients $a_k$, $b_k$, $d_k$ et $f_k$ sont donnés explicitement en fonction de $\omega$, $n$ et $k$.
Ensuite, on vérifie que pour tout $i<j$
\begin{multline*}
 (n-2)(d_j-d_i)+d_i\sqrt{\Delta_j}+d_j\sqrt{\Delta_i}>(n-2)(d_j-d_i)\\
+d_i\sqrt{a_j}(n+\frac{b_j}{2a_j})+d_j\sqrt{a_i}(n+\frac{b_i}{2a_i})>0
\end{multline*}
D'après ce qui a été dit dans le cas 7 ci-dessus, l'inégalité du lemme est démontrée.

\bibliographystyle{amsplain}
\bibliography{bibliographie}\addcontentsline{toc}{chapter}{Bibliographie}

\end{document}